\documentclass[a4paper,french,openright,11pt]{smfartVScomptage}

\usepackage{smfenum,smfthm,amsmath,amssymb,amscd,mathrsfs,euscript}

\usepackage{stmaryrd,mathabx,upgreek,color}
\usepackage[latin1]{inputenc}
\input{xypic}

\numberwithin{equation}{section} 
\theoremstyle{plain}

\def\CC{\mathbf{C}}
\def\NN{\mathbf{N}}

\def\ZZ{\mathbf{Z}}

\def\A{{\rm A}}
\def\B{{\rm B}}
\def\C{{\rm C}}
\def\D{{\rm D}}
\def\E{{\rm E}}
\def\F{{\rm F}}
\def\G{{\rm G}}
\def\H{{\rm H}}
\def\I{{\rm I}}
\def\J{{\rm J}}

\def\L{{\rm L}}
\def\M{{\rm M}}
\def\N{{\rm N}}
\def\P{{\rm P}}
\def\Q{{\rm Q}}
\def\R{{\rm R}}
\def\SS{{\rm S}}
\def\T{{\rm T}}

\def\X{{\rm X}}
\def\Y{{\rm Y}}
\def\Z{{\rm Z}}

\def\Aa{\mathscr{A}}

\def\Cc{\EuScript{C}}
\def\Dd{\EuScript{D}}
\def\Ee{\mathscr{E}}

\def\Gg{\EuScript{R}}

\def\Oo{\EuScript{O}}
\def\Pp{\EuScript{P}}

\def\Tt{\EuScript{T}}

\def\Zz{\EuScript{Z}}

\def\La{\Lambda}

\def\a{\alpha} 
\def\b{\beta}
\def\d{\delta}
\def\e{\varepsilon}

\def\g{\gamma}

\def\k{\kappa}
\def\l{\lambda}

\def\s{\sigma}
\def\t{\theta}

\def\w{\varpi}

\def\ie{c'est-à-dire }

\def\rp{\rangle}
\def\>{\geqslant}
\def\<{\leqslant}

\def\Hom{{\rm Hom}}

\def\Mat{{\rm M}}
\def\GL{{\rm GL}}
\def\Gal{{\rm Gal}}

\def\mult#1{{#1}^{\times}}

\def\ffr#1{\smash{\mathop{\longrightarrow}\limits^{#1}}}

\def\qlb{{\overline{\mathbf{Q}}_\ell}}
\def\zlb{{\overline{\mathbf{Z}}_\ell}}
\def\flb{{\overline{\mathbf{F}}_{\ell}}}

\def\ip{\boldsymbol{i}}
\def\rp{\boldsymbol{r}}

\def\r{{\textbf{\textsf{r}}}}

\def\ee{\varepsilon}

\def\LJ{\textbf{\textsf{B}}}
\def\TLJ{\widetilde{\LJ}{}}
\def\JL{\textbf{\textsf{J}}}
\def\TJL{\widetilde{\JL}{}}
\def\jl{\boldsymbol{j}}
\def\tjl{\widetilde{\boldsymbol{\jmath}}{}}

\def\tlj{\tjl^{*}}

\def\kk{\boldsymbol{k}}
\def\kd{\boldsymbol{d}}
\def\ke{\boldsymbol{e}}

\def\kl{\boldsymbol{k}'}
\def\d{k}
\def\gg{g}

\def\sp{{\rm sp}}
\def\Sp{{\rm Sp}}

\def\xt{x}
\def\ct{\widetilde\chi}
\def\kt{\widetilde\k}
\def\lt{\widetilde\l}
\def\st{\widetilde\s}
\def\tL{\widetilde\Lambda}

\def\GB{\mathcal{G}}
\def\TT{\boldsymbol{\Theta}}
\def\KM{\textbf{\textsf{K}}}
\def\0{\boldsymbol{0}}
\def\YY{\textbf{\textsf{Y}}}
\def\BB{\textbf{\textsf{B}}}

\def\at{\widetilde\a}
\def\tc{\widetilde\chi}

\def\tp{{\widetilde\pi}^{}{}}
\def\rt{{\widetilde\rho}^{}{}}

\def\tt{\widetilde\uptheta}

\def\tL{\widetilde\La}

\def\kt{\widetilde\k}
\def\lt{\widetilde\l}

\def\RA{\EuScript{R}} 
\def\RAT{\widehat{\EuScript{R}}}
\def\XA{{\rm Irr}}
\def\Seg{\EuScript{D}}
\def\cuspi{\rho}
\def\cuspit{{\widetilde\rho}}
\def\cuspif{\s}

\def\SFZ{\textbf{\textsf{S}}}
\def\SFD{\boldsymbol{\Delta}}
\def\SFDD{\textbf{\textsf{D}}}
\def\SFS{\textbf{\textsf{Z}}}
\def\SFU{\textbf{\textsf{U}}}
\def\SFA{\textbf{\textsf{A}}}
\def\SFB{\textbf{\textsf{C}}}
\def\SFL{\textbf{\textsf{V}}}
\def\SFM{\textbf{\textsf{M}}}
\def\com{{\sf c}}

\long\def\MSC#1\EndMSC{\def\arg{#1}\ifx\arg\empty\relax\else
     {\par\narrower\noindent%
     2010 Mathematics Subject Classification: #1\par}\fi}

\long\def\KEY#1\EndKEY{\def\arg{#1}\ifx\arg\empty\relax\else
	{\par\narrower\noindent Keywords and Phrases: #1\par}\fi}

\title[Correspondance de Jacquet-Langlands locale et congruences modulo $\ell$]
{Correspondance de Jacquet-Langlands locale et congruences modulo $\ell$}

\author{Alberto M\'\i nguez}
\address{Institut de Mathématiques de Jussieu, Paris Rive Gauche, Université 
  Pierre et Marie Curie, 4 place Jussieu, 75005, Paris, France.} 
\email{alberto.minguez@imj-prg.fr}

\author{Vincent Sécherre}
\address{Laboratoire de Mathémati\-ques de Versailles, UVSQ, CNRS, Université
Paris-Saclay, 78035 Versailles, France}
\email{vincent.secherre@math.uvsq.fr}

\begin{altabstract}
Let $\F$ be a non-Archimedean local field of residual characteristic $p$, 
and $\ell$ be a prime number different from $p$. 
We consider the local Jacquet-Langlands correspondence between $\ell$-adic 
discrete series of $\GL_n(\F)$ and an inner form $\GL_m(\D)$. 
We show that it respects the relationship of con\-gruen\-ce modulo $\ell$.
More precisely, we show that two integral $\ell$-adic discrete series of 
$\GL_n(\F)$ are congruent modulo $\ell$ if and only if the same holds for 
their Jacquet-Langlands transfers to $\GL_m(\D)$. 
We also prove that the Langlands-Jacquet morphism from 
the Grothendieck group of finite length $\ell$-adic representations of 
$\GL_n(\F)$ to that of $\GL_m(\D)$ defined by 
Badulescu is compatible~with reduction mod $\ell$.
\end{altabstract}

\thanks{
Le séjour de V.~Sécherre à l'University of East Anglia (Norwich)
en mai 2014, 
durant lequel une partie de ce travail a été ef\-fectuée, 
a été financé par l'EPSRC grant EP/H00534X/1.
Il remercie chaleureusement Shaun Stevens pour son invitation 
et les discussions à propos de ce travail.
A.~Minguez remercie João Pedro Dos Santos et Erez Lapid pour les
discussions à propos de ce travail. }

\begin{document} 

\maketitle


\MSC 
22E50 
\EndMSC
\KEY 
Modular representations of $p$-adic reductive groups, 
Jacquet-Lang\-lands correspondence, 
Cuspidal representations,
$\ell$-adic lifting, Congruences mod $\ell$
\EndKEY

\thispagestyle{empty}

\section{Introduction et énoncé des principaux résultats}

\subsection{} 

Soit $\F$ un corps localement compact non archimédien de caractéristique 
résiduelle $p$, 
soit $\H$ le groupe linéaire général $\GL_n(\F)$, $n\>1$, et soit $\G$ une 
forme intérieure de $\H$ sur $\F$.
Celle-ci est un groupe de la forme $\GL_m(\D)$, 
où $m$ est un diviseur de $n$ et $\D$ une algèbre à division de centre
$\F$, de degré réduit noté $d$, tels que $md=n$.
Notons $\Dd(\G,\CC)$ l'ensemble des classes d'isomorphisme de re\-pré\-sentations 
lisses complexes irréductibles, essentiellement de carré intégrable, de $\G$.
La cor\-res\-pon\-dance de Jacquet-Langlands locale 
\cite{JL,Rog,DKV,BaduJL}
est une bi\-jec\-tion~:
\begin{equation*}
\boldsymbol{\pi} : \Dd(\G,\CC)\to\Dd(\H,\CC)
\end{equation*}
carac\-térisée par une identité de caractères 
sur les classes de conjugaison elliptiques régu\-liè\-res.
Elle relie, dans l'esprit du programme de Langlands, 
la théorie des re\-pré\-sentations lisses comple\-xes de $\G$ à celle de $\H$.


\subsection{} 
\label{introtpil}

Si l'on fixe un nombre premier $\ell$ différent de $p$ et que l'on passe aux
représentations $\ell$-adiques, en fixant un isomorphisme de
corps entre $\CC$ et une clôture algébrique $\qlb$ du corps des nombres 
$\ell$-adiques, on obtient une bijection~:
\begin{equation}
\label{PILT}
\widetilde{\boldsymbol{\pi}}_\ell : \Dd(\G,\qlb)\to\Dd(\H,\qlb)
\end{equation}
indépendante de cet isomorphisme, 
$\Dd(\G,\qlb)$ étant obtenu à partir de $\Dd(\G,\CC)$ par extension
des scalaires de $\CC$ à $\qlb$.
On a alors 
une notion de représentation $\ell$-adique entière, qu'on peut réduire mod $\ell$, 
et l'on peut étudier 
la compatibilité de $\widetilde{\boldsymbol{\pi}}_\ell$ vis-à-vis de 
la réduction modulo $\ell$.
L'ana\-lo\-gue de ce problème 
pour la correspondance de Langlands locale a été
étudié par Vignéras~\cite{Vigl}, puis Dat \cite{Datl} et Bushnell-Henniart 
\cite{BHl}. 

\label{menard}

Une représenta\-tion de $\Dd(\G,\qlb)$ étant entière si et seulement si son 
caractère central l'est, et la bijection 
$\widetilde{\boldsymbol{\pi}}_\ell$ préservant le caractère central, 
elle préserve aussi le fait d'être une re\-pré\-senta\-tion entière. 
Disons que deux représen\-ta\-tions $\ell$-adiques irréductibles entières de 
$\G$ sont \textit{congruentes mod $\ell$} si leurs réductions modulo $\ell$ 
sont identiques dans le groupe de Grothendieck $\RA(\G,\qlb)$ 
des représentations $\ell$-adiques de longueur finie de $\G$.
Enonçons le premier des résultats principaux cet article. 

\begin{theo}
\label{MAINTHEOREM}
Deux représentations entières de $\Dd(\G,\qlb)$ sont congruentes modulo $\ell$ 
si et seulement si leurs images par $\widetilde{\boldsymbol{\pi}}_\ell$ le 
sont. 
\end{theo}

La preuve du théorème \ref{MAINTHEOREM} est en partie inspirée de  
\cite{Datj}, qui traite le cas particulier où $\G$ est 
une forme intérieure compacte modulo le centre, 
et ne considère que les représentations entières 
de $\G$ dont la réduction modulo $\ell$ est irréductible. 
Pour traiter le cas général, 
des modi\-fi\-cations substan\-tielles doivent être apportées. 
Expliquons tout ceci en détail.

\subsection{} 
\label{Intro4}

Soit 
$\A$ le groupe multiplicatif d'une 
$\F$-algèbre à division centrale de degré réduit égal à $n$.
Dat a construit dans \cite{Datj} 
une bijection entre classes de représentations 
irréductibles $\ell$-modulaires --- \ie à coefficients dans une 
clôture algébrique $\flb$ d'un corps fini de caractéristique $\ell$ --- 
de $\A$ et classes de certaines représentations irréductibles 
$\ell$-modulaires de $\H$, baptisées ``super-Speh''. 
Elle est compatible, en un certain sens, à la correspondance 
de Jacquet-Lang\-lands $\ell$-adique 
\eqref{PILT} ci-dessus pour $\G$ égal à $\A$. 

\label{Intro5}

Plutôt que d'étudier directement la série discrète $\ell$-adique, 
qui se réduit mal modulo $\ell$, 
Dat étu\-die son image par l'involution de Zelevinski, \ie l'ensemble des
classes de représen\-ta\-tions de Speh $\ell$-adiques de $\H$.
De telles représentations sont dites 
$\ell$-super-Speh lorsqu'elles sont entières, 
et lors\-que leur réduction modulo $\ell$ est 
irréductible avec un support cuspidal supercus\-pi\-dal. 
La construction de la correspondan\-ce de Jacquet-Langlands locale 
modulo $\ell$ de \cite{Datj} repose sur le fait crucial que la 
correspondan\-ce \eqref{PILT} 
fait se correspondre bijectivement 
représentations $\ell$-adiques entières de $\A$ dont la réduction modulo $\ell$ est 
irréductible, et représentations $\ell$-super-Speh de $\H$.


Pour prouver ce fait, Dat utilise un critère numérique de
$\ell$-supercuspidalité établi par Vignéras pour construire 
une correspondance de Langlands locale modulo $\ell$ (\cite{Vigl}).
La réduction modulo $\ell$ d'une représentation irréductible cuspidale 
$\ell$-adique entière $\rt$ de $\H$ est toujours irréductible et cuspidale, 
mais elle n'est pas toujours supercuspidale~; 
plus précisément, elle est supercuspidale si et seu\-le\-ment si 
le nombre de représentations 
cuspidales $\ell$-adi\-ques entières de $\H$ qui sont 
({stric\-te\-ment}) congrues à $\rt$ est 
``le plus grand possible'' (\cite{Vigl} Proposition 2.3).
Il y a aussi une va\-riante de ce critère numérique pour $\A$ 
(\cite{Datj} Proposition 2.3.2).

\subsection{} 
\label{P14}

Si l'on veut construire une correspondance de Jacquet-Langlands locale mod 
$\ell$ géné\-ra\-le, il est naturel de commencer par étendre à $\G$ le critère
numérique de $\ell$-supercuspidalité.
C'est ce que nous faisons dans la section \ref{Section4},
en le présentant sous une forme légèrement diffé\-rente.
Soit $\rt$ une re\-pré\-sentation ir\-ré\-duc\-tible cuspidale $\ell$-adique entière de
$\G$. 
D'après \cite{MSt} Théorème 3.15, il y a 
une re\-pré\-sen\-ta\-tion irréductible cuspidale $\ell$-modulaire 
$\rho$ de $\G$ et un unique entier $a=a(\rt)\>1$ tels que la ré\-duc\-tion modulo 
$\ell$ de $\rt$ soit égale à~: 
\begin{equation}
\label{redcusp}
\r_\ell(\rt)=\rho+\rho\nu+\dots+\rho\nu^{a-1}
\end{equation}
dans le groupe de Grothendieck 
$\RA(\G,\flb)$
des représentations $\ell$-modulaires de 
longueur finie de $\G$, où $\nu$ désigne le caractère valeur absolue de la 
norme réduite.
La repré\-sentation $\rho$ n'est pas unique en gé\-né\-ral, mais sa classe 
inertielle $[\G,\rho]$ ne dépend que de la clas\-se inertielle $[\G,\rt]$ de $\rt$.
Quand $\G$ est déployé, l'entier $a(\rt)$ est toujours égal à $1$, 
\ie que la ré\-duction modulo $\ell$ de $\rt$ est tou\-jours irréductible.

\begin{defi}
\label{DEFlssintro}
On dit que $\rt$ est \textit{$\ell$-supercuspidale} 
si $\r_\ell(\rt)$ est irréductible et supercuspidale.
\end{defi}

\label{Notrlic}

Etant donnée une représentation irréductible cuspidale $\ell$-adique $\rt$ de $\G$,
on note~:
\begin{equation*}
\r_\ell([\G,\rt]) 
\end{equation*}
l'ensemble des réductions modulo $\ell$ des représentations 
entières iner\-tiellement équivalentes à $\rt$, 
et on appelle cet ensemble la {réduction modulo $\ell$} de $[\G,\rt]$. 
On note $n(\rt)$ le nom\-bre de ca\-rac\-tè\-res $\ell$-adiques 
non ramifiés $\ct$ de $\G$
tels que $\rt\ct$ est isomorphe à $\rt$ et $c(\rt)$
la plus grande puis\-san\-ce de $\ell$ divisant $q^{n(\rt)}-1$. 
Le résultat suivant généralise \cite{Vigl} et \cite{Datj}.

\begin{prop}
\label{CongCusp}
Soit $\rt$ une représentation irréductible cuspidale $\ell$-adique entière de $\G$. 
\begin{enumerate}
\item
L'ensemble des classes inertielles de représentations 
irréductibles cuspidales $\ell$-adi\-ques en\-tiè\-res de $\G$ congrues à $\rt$
est fini, de cardinal noté $t(\rt)$.
\item
On a~:
\begin{equation*}
t(\rt)\<c(\rt)
\end{equation*}
avec égalité si et seulement si $\rt$ est $\ell$-supercuspidale.
\end{enumerate}
\end{prop}

\subsection{} 
\label{DEFK}

Intéressons-nous maintenant au cas où $\rt$ n'est pas
$\ell$-super\-cuspidale~;
en étu\-diant plus finement la façon dont les entiers $t(\rt)$ et $c(\rt)$ diffèrent, 
il est raisonnable de penser qu'on pourra en déduire des in\-for\-mations
sur la structure de $\rt$.
D'après la classification des représentations irréductibles cuspidales 
$\ell$-modulaires de $\G$ en 
fonction des supercuspidales (\cite{MSc} Théorème 6.14), 
il existe un unique entier naturel~:
\begin{equation*}
\d(\rho) \>1
\end{equation*}
tel que $\rho$ apparaisse comme sous-quotient de l'induite parabolique d'une 
représentation irré\-duc\-tible supercuspidale du sous-groupe de Levi standard 
$\GL_{r}(\D)\times\dots\times\GL_{r}(\D)$ avec $r\d(\rho)=m$.
(Autrement dit, $k(\cuspi)$ est le nombre de termes du support supercuspidal 
de $\cuspi$.)
En particulier, $\rho$ est supercuspidale si et seulement si $\d(\rho)=1$.
Posons~:
\begin{equation*}
w(\rt) = \d(\rho) a(\rt).
\end{equation*}
Ainsi $\rt$ est $\ell$-super\-cuspidale si et seulement si $w(\rt)=1$. 
Le résultat suivant montre qu'on peut déterminer la valeur de $w(\rt)$ 
en com\-pa\-rant $t(\rt)$ et $c(\rt)$.

\begin{prop}
\label{CalculW}
Soit $\rt$ une $\qlb$-représentation irréductible cuspidale entière et 
non $\ell$-super\-cuspidale de $\G$. 
Alors~:
\begin{equation*}
t(\rt) w(\rt) = 
\left\{
\begin{array}{ll}
c(\rt)-1 & \text{si $t(\rt)$ est premier à $\ell$,} \\
c(\rt)(\ell-1)\ell^{-1} & \text{sinon.}
\end{array}
\right.
\end{equation*}
\end{prop}

\subsection{} 
\label{par7}

Changeons maintenant de point de vue. 
Quand $\G$ est déployé, 
Vignéras a montré (\cite{Vigb}) qu'une représentation irréductible
cus\-pi\-dale $\ell$-modu\-lai\-re $\rho$ de $\G$ 
se relève toujours en une 
représentation $\ell$-adique de $\G$, 
\ie qu'il existe une re\-pré\-sentation $\ell$-adique entière de $\G$ dont la 
réduction modulo $\ell$ est isomorphe à $\rho$. 
Si main\-te\-nant $\G$ n'est pas déployé, toute représentation 
irréductible \textit{super\-cuspidale}
$\ell$-mo\-du\-laire de $\G$ se relève à $\qlb$
(voir \cite{MSt,MSc})
mais il existe 
des représentations cuspida\-les qui ne se relèvent pas.
Etant don\-née une représentation 
cuspidale non super\-cus\-pidale $\ell$-mo\-dulaire $\rho$ de $\G$,
il est naturel de demander à quelle condition elle admet un relè\-ve\-ment.

\label{DEFds}

Pour répondre à cette question, nous avons besoin de l'invariant~:
\begin{equation*}
s(\rho) \>1
\end{equation*} 
introduit dans \cite{MSt},
dont la définition repose sur la construction des
représentations irré\-duc\-ti\-bles cuspidales de $\G$ par la théorie des types 
de Bushnell-Kutzko (voir le paragraphe \ref{PARA21}). 
C'est un diviseur de $d$~; en particulier il est toujours égal à $1$ quand 
$\G$ est déployé.
Cet invariant est relié à un autre invariant, le 
\textit{degré para\-métrique} $\delta(\rho)$ 
introduit dans \cite{BHJL3}, par l'identité $\delta(\rho)s(\rho)=md$.

\begin{prop}
\label{lift}
Soit $\rho$ une $\flb$-représentation irréductible cuspidale non 
supercuspidale de $\G$.
Pour que $\rho$ se relève à $\qlb$, il faut et il suffit que 
les entiers $s(\rho)$ et $\d(\rho)$ soient premiers entre eux
et que la représentation tordue $\rho\nu$ soit isomorphe à $\rho$.
\end{prop}

Quand $\G$ est déployé, on a 
toujours $s(\rho)=1$ et une repré\-sen\-tation irréductible cuspidale non 
super\-cuspidale $\rho$ est toujours isomorphe à sa tordue $\rho\nu$.
La condition de la proposition \ref{lift}
est donc toujours vérifiée~; on retrouve ainsi le résultat de relèvement de 
Vignéras. 


Plus généralement, 
on peut déterminer 
les valeurs possibles de $a(\rt)$ lorsque $\rt$ décrit les repré\-sen\-tations 
irré\-duc\-tibles cuspidales $\ell$-adiques entières 
de $\G$ dont la réduction modulo $\ell$ contient~$\rho$. 
La proposition suivante
répond à cette question et complète ainsi la proposition \ref{lift}.
Notons $v$ la valuation $\ell$-adique sur $\ZZ$
(normalisée par $v(\ell)=1$) et notons $\e(\rho)$ l'ordre de $q^{n(\rho)}$ dans 
$(\ZZ/\ell\ZZ)^\times$, \ie le plus 
petit entier $k\>1$ tel que $\rho\nu^k$ soit isomorphe à $\rho$
(voir le lemme \ref{STEP0}).

\begin{prop}
\label{lifta}
Soit $\rho$ une $\flb$-représentation irréductible cuspidale 
de $\G$ et soit un entier $a>1$.
Pour qu'il existe une $\qlb$-représentation irréductible cuspidale entière 
$\rt$ dont la réduc\-tion modulo $\ell$ contienne $\rho$ et soit de 
longueur $a$, il faut et il suffit que~:
\begin{enumerate}
\item
il existe un entier $u\in\{0,\dots,v(s(\rho))\}$ tel que $a=\e(\rho)\ell^u$~;
\item
les entiers $s(\rho)a^{-1}$ et $\d(\rho)$ soient premiers entre eux. 
\end{enumerate}
\end{prop}

\subsection{} 
\label{Anwj}

Nous utilisons ensuite la proposition \ref{lifta} pour 
obtenir une formule de comptage de classes iner\-tiel\-les de représentations 
cuspidales $\ell$-modulaires, dans l'esprit de~\cite{BHcgds}. 
Contrairement à Bush\-nell et Henniart, qui obtiennent leur formule en 
s'appuyant sur la correspondance de Jacquet-Lang\-lands locale et sur 
l'existence préalable 
d'une telle for\-mu\-le dans le cas 
du groupe mul\-tiplicatif d'une algèbre à division,
nous établissons la nôtre par un calcul direct, en termes 
de $\F$-endo\-classes de caractères simples \cite{BSS}.

Fixons un entier $w$ divisant $n$ et un nombre rationnel $j\>0$,
et notons $\Aa_\ell(\D,w,j)$ l'en\-sem\-ble des réductions mod $\ell$ de 
classes inertielles de re\-pré\-sen\-ta\-tions ir\-ré\-ductibles cuspidales 
$\ell$-adiques $\rt$ telles que~: 
\begin{enumerate}
\item
il existe un entier $u\>1$ divisant $m$ 
tel que $\rt$ soit une représentation irréductible
cuspidale $\ell$-adique de $\GL_{u}(\D)$~;
\item
on a $w(\rt)=w$ et le niveau normalisé de $\rt$ est inférieur ou égal à $j$.
\end{enumerate}
C'est un ensemble fini, de cardinal noté $\boldsymbol{a}_{\ell}(\D,w,j)$.
Fixons par ailleurs une clôture al\-gé\-brique $\overline{\kk}$ du corps
résiduel de $\F$, notons $q$ le cardinal du corps résiduel de $\F$ et 
$\boldsymbol{y}_{\ell}^{1}(q,n,w)$ le nom\-bre de 
$y\in\overline{\kk}{}^\times$ tels que~:
\begin{enumerate}
\item
l'ordre de $y$ est premier à $\ell$~;
\item
le degré de $y$ sur le corps résiduel de $\F$,
noté $\deg(y)$, divise $nw^{-1}$~; 
\item
l'ordre de $q^{\deg(y)}$ 
dans $(\ZZ/\ell\ZZ)^\times$ est égal au 
plus grand diviseur de $w$ premier à $\ell$. 
\end{enumerate}
On a la formule suivante~;
pour la notion d'endo-classe, on renvoie au paragraphe \ref{ENDO} 
et à \cite{BSS}.

\begin{prop}
\label{comptagedeclasses}
On a~:
\begin{equation}
\label{comptagedeclassesFORMULEintro}
\boldsymbol{a}_{\ell}(\D,w,j) 
= \sum\limits_{\TT}\ \boldsymbol{y}_{\ell}^{1}(q(\TT),n(\TT),w),
\end{equation}
la somme portant sur les $\F$-endoclasses $\TT$ de niveau normalisé inférieur 
ou égal à $j$ et de degré $\deg(\TT)$ divisant $nw^{-1}$, et où~:
\begin{equation*}
n(\TT)=\frac{n}{\deg(\TT)},
\quad
q(\TT)=q^{f(\TT)},
\end{equation*}
l'entier $f(\TT)$ désignant le degré résiduel de $\TT$.
\end{prop}

Cette somme ne dépendant que de $\ell$, $n$, $w$, $j$ et $q$, 
on en déduit le corollaire suivant.

\begin{coro}
\label{DonaTartt}
On a $\boldsymbol{a}_{\ell}(\D,w,j)=\boldsymbol{a}_{\ell}(\F,w,j)$. 
\end{coro}

\subsection{}
\label{Mafoi}
\label{P16}

Revenons à la correspondance \eqref{PILT}. 
Comme dans \cite{Datj}, nous allons passer au dual de Zelevinski
et nous allons avoir besoin d'une version des propositions \ref{CongCusp} et 
\ref{CalculW} pour les représenta\-tions de Speh $\ell$-adiques de $\G$.
Décrivons plus en détail la structure de ces représentations.
Etant don\-née une représentation de Speh $\ell$-adique entière $\tp$,
il existe un unique diviseur $r$ de $m$ et une unique représentation 
irréductible cuspi\-dale $\cuspit$ de $\GL_{mr^{-1}}(\D)$ tels que $\tp$ soit 
l'unique sous-représentation irréductible, notée $\Z(\cuspit,r)$, de l'induite 
parabolique normalisée~:
\begin{equation}
\label{induite}
\cuspit\times\cuspit\nu_{\cuspit}^{}\times\dots\times\cuspit\nu_{\cuspit}^{r-1}
\end{equation}
où $\nu_{\cuspit}$ est un caractère non ramifié associé à $\cuspit$
(voir le \S \ref{PARA21}). 
On pose alors $w(\tp) = w(\cuspit)$.

\begin{defi}
\label{DEFlsspehintro}
On dit que $\tp$ est \textit{$\ell$-super-Speh} si $w(\tp)=1$.
\end{defi}

Dans le cas déployé, \ie quand $d=1$,
l'entier $a(\cuspit)$ vaut toujours $1$, 
\ie que la ré\-duc\-tion mod $\ell$ d'une représentation irréductible 
cuspidale $\ell$-adique de $\H$ est toujours ir\-ré\-duc\-ti\-ble. 
A l'autre extrême, si $m=1$, l'entier $k(\cuspi)$ vaut toujours $1$. 
Pris séparément, ces entiers ne peuvent donc pas être invariants par la 
correspondance \eqref{PILT}. 
Nous allons voir qu'en re\-van\-che leur produit $w(\tp)$ l'est.
Il joue un rôle important dans la preuve du théorème \ref{MAINTHEOREM}.
Tout d'abord, nous prouvons la formule suivante, qui généralise la formule 
\eqref{redcusp}. 

\begin{prop}
\label{ConjRedModlSpehIntro}
Soit $\cuspit$ une $\qlb$-représentation irréductible cuspidale 
entiè\-re de $\G$.
Soit $\cuspi$ un facteur irréductible de sa réduction
mod $\ell$, et soit $a=a(\cuspit)$.
Pour tout entier $r\>1$, on a~:
\begin{equation*}
\label{ForrZIntro}
\r_\ell(\Z(\cuspit,r)) = \sum\limits
\Z(\cuspi,r_0)\times\Z(\cuspi\nu,r_1)\times\dots\times\Z(\cuspi\nu^{a-1},r_{a-1})
\end{equation*}
la somme portant sur les familles $(r_0,\dots,r_{a-1})$ d'entiers $\>0$ de 
somme $r$.
\end{prop}

Cette proposition a pour conséquence 
une propriété remarqua\-ble de compa\-tibilité de la classi\-fication 
de Zelevinski de \cite{MSc} à la réduction mod $\ell$.
On renvoie au para\-graphe \ref{P51} pour les termes et les notations non 
définis, ainsi qu'à la proposition \ref{Zcong}.

\begin{prop}
\label{Zcongintro}
Soient $\cuspit_1$, $\cuspit_2$ deux représentations irréductibles cuspidales 
$\ell$-adiques en\-tiè\-res congruentes mod $\ell$, 
et soit $\upmu$ un multisegment formel. 
Alors les représentations 
$\Z(\upmu\boxtimes\cuspit_1)$ et $\Z(\upmu\boxtimes\cuspit_2)$ 
sont congruentes mod $\ell$. 
\end{prop}

\subsection{}

La \textit{classe de torsion} de $\tp$ est l'ensemble 
$\langle\tp\rangle$ des classes de représentations obtenues en tordant $\tp$ 
par un caractère non ramifié de $\G$. 
Lorsque $\tp$ est cuspidale, 
l'ensemble $\langle\tp\rangle$ est donc simplement la classe d'inertie de $\tp$.
Généralisant la notation du paragraphe \ref{Notrlic}, on note~:
\begin{equation*}
\r_\ell(\langle\tp\rangle)
\end{equation*}
l'ensemble des $\r_\ell(\tp\tc)$ où $\tc$ décrit les caractères non 
ramifiés $\ell$-adiques de $\G$ tels que la représen\-ta\-tion 
$\tp\tc$ soit en\-tiè\-re. 
Notons également $n(\tp)$ le nom\-bre de ca\-rac\-tè\-res $\ell$-adiques 
non ramifiés $\ct$ de $\G$
tels que $\tp\ct$ soit isomorphe à $\tp$ et $c(\tp)$
la plus grande puis\-san\-ce de $\ell$ divisant $q^{n(\tp)}-1$. 
La pro\-po\-si\-tion \ref{ConjRedModlSpehIntro} 
(jointe aux propositions \ref{CongCusp}, \ref{CalculW} et \ref{comptagedeclasses})
implique les deux résultats suivants.

\begin{theo}
\label{CongSpehIntro}
Soit $\tp$ une $\qlb$-représentation de Speh entière de $\G$. 
L'ensemble des clas\-ses de torsion de $\qlb$-représen\-tations de Speh
entières congrues à $\tp$ est fini, et son cardinal $t(\tp)$ vérifie~:
\begin{equation*}
t(\tp) w(\tp) = 
\left\{
\begin{array}{ll}
c(\tp) & \text{si $w(\tp)=1$,} \\
c(\tp)-1 & \text{si $1<w(\tp)<\ell$,} \\
c(\tp)(\ell-1)\ell^{-1} & \text{si $w(\tp)\>\ell$.}
\end{array}
\right.
\end{equation*}
\end{theo}

\begin{theo}
\label{ComptagePIIntro}
Soient $w$ un entier divisant $n$ et $j\>0$ un nombre ra\-tion\-nel. 
L'ensemble $\Ee_\ell(\G,w,j)$ des $\r_\ell(\langle\tp\rangle)$,
où $\tp$ décrit les représentations de Speh $\ell$-adiques de $\G$ 
telles que $w(\tp)=w$ et dont le ni\-veau normalisé est inférieur ou égal à 
$j$, est fini et de cardinal $\boldsymbol{a}_{\ell}(\D,w,j)$.
\end{theo}

D'après le corollaire \ref{DonaTartt}, les ensembles 
$\Ee_\ell(\G,w,j)$ et $\Ee_\ell(\H,w,j)$ ont donc le même cardinal. 

\subsection{} 
\label{ParIntro5}

A partir de là, la preuve du théorème \ref{MAINTHEOREM} se fait par récurrence 
sur $w(\tp)$.
Comme au para\-gra\-phe \ref{Intro4}, fixons une $\F$-algèbre à division de 
degré réduit $n$, et notons $\A$ son groupe multiplicatif.
La correspondance de Jacquet-Langlands locale $\ell$-adique 
détermine une bi\-jection entre $\Dd(\G,\qlb)$ et l'en\-semble $\XA(\A,\qlb)$ des
représentations ir\-ré\-ductibles $\ell$-adi\-ques de~$\A$, 
que l'on peut prolonger en un morphisme surjectif 
entre groupes de Grothendieck~:
\begin{equation}
\label{TJLell}
\TJL_\ell : 
\RA(\G,\qlb) \to \RA(\A,\qlb)
\end{equation} 
trivial sur les indui\-tes paraboliques 
(paragraphe \ref{DEFLJ}).
Grâce à la théorie du caractère de Brauer de Dat \cite{Datj} \S2.1, 
il y a un unique morphisme de groupes
$\JL_\ell$ de $\RA(\G,\flb)$ dans $\RA(\A,\flb)$
qui soit com\-pa\-ti\-ble à $\TJL_\ell$ par réduction mod $\ell$
(voir la proposition \ref{Leighton}),
ce qui permet de transpor\-ter les relations de 
congruence mod $\ell$ de $\G$ à $\A$.

Restreignant le morphisme \eqref{TJLell} à l'ensemble $\Zz(\G,\qlb)$ des classes 
d'isomorphisme de représen\-tations de Speh $\ell$-adiques de $\G$, l'image 
d'une re\-présenta\-tion de Speh $\tp$ est égale, 
à un signe~près, à une représentation irréductible de $\A$, 
correspondante de Jac\-quet-Langlands de la duale de 
Ze\-le\-vinski de $\tp$~;
on en déduit une bijection~:
\begin{equation}
\label{LJTl}
\Zz(\G,\qlb) \to \XA(\A,\qlb).
\end{equation}
L'existence de $\JL_\ell$ assure que des représentations entières 
de $\Zz(\G,\qlb)$ con\-gru\-en\-tes mod $\ell$ ont des ima\-ges dans 
$\XA(\A,\qlb)$ qui sont congruentes mod~$\ell$. 
Ensuite, grâce au théorème \ref{CongSpehIntro}, on montre que cette bijection 
préserve l'invariant $w(\tp)$ et qu'elle induit par ré\-duc\-tion 
modulo $\ell$, pour tout entier $w\>0$, une application injective
de $\Zz_w(\G,\flb)$ dans $\Zz_w(\A,\flb)$, où l'on a posé~:
\begin{equation}
\label{LJlintro}
\Zz_w(\G,\flb) 
= \{\r_\ell(\tp)\ |\ \tp\in\Zz(\G,\qlb) \text{ est entière et } w(\tp)=w\}
\subseteq\RA(\G,\flb).
\end{equation}
La correspondance de Jacquet-Langlands locale
préservant le niveau normalisé, et les éléments~de 
l'ensemble \eqref{LJlintro}
de niveau normalisé fixé étant 
-- à torsion non ramifiée près --
en nombre fini, 
le théorème \ref{ComptagePIIntro} et le corollaire \ref{DonaTartt} 
impliquent que cette application injective est une bijection. 
Appli\-quant à nouveau l'involution de Zelevinski \cite{MSb,MSi} 
pour revenir à $\Dd(\G,\qlb)$,
ceci met fin à la preuve du théorème \ref{MAINTHEOREM} 
(voir le théorème \ref{RogerCarbury}).

\subsection{}
\label{Finale}

Si l'on restreint \eqref{LJTl} à l'ensemble 
des représentations $\ell$-super-Speh, \ie aux $\tp$ vérifiant $w(\tp)=1$, 
on ob\-tient une bijection entre les représentations $\ell$-super-Speh de $\G$ 
et les représenta\-tions $\ell$-adiques entières de $\A$ dont la réduction
modulo $\ell$ est irréductible. 
Réduisant modulo $\ell$, on obtient le résultat suivant
(voir le corollaire \ref{superSpeh}) qui généralise \cite{Datj} Théorème 
1.2.4. 

\begin{coro}
\label{superSpehIntrotro}
La bijection \eqref{LJTl} induit une bijection entre 
re\-pré\-sen\-ta\-tions $\ell$-modulaires super-Speh de $\G$ et 
re\-pré\-sen\-ta\-tions $\ell$-modulaires irréductibles de $\A$. 
\end{coro}

\subsection{} 
\label{noneedtoswitch}

Signalons que, dans la preuve du théorème \ref{MAINTHEOREM},
il n'est pas à proprement parler né\-ces\-saire de passer par les 
représentations de Speh~: contrairement à \cite{Datj}, 
dont la preuve s'appuie sur le fait que
la ré\-duc\-tion mod $\ell$ de $\Z(\rt,r)$ est irréductible 
pour toute représentation 
$\ell$-supercuspidale~$\rt$ 
(ce qui n'est pas vrai de sa duale de Zelevinski), 
notre argument fonctionne encore si l'on utili\-se 
directement $\Dd(\G,\qlb)$ 
et les réductions mod~$\ell$ de ses éléments entiers. 
L'argument de comptage, qui~porte de toutes façons sur des 
ensembles \eqref{LJlintro}
de représenta\-tions qui sont en général non irréductibles, reste valable. 
Nous avons choisi d'utiliser les re\-pré\-sentations de Speh d'une part pour 
obtenir le corollaire \ref{superSpehIntrotro}, généralisant à une forme 
inté\-rieure quelconque la cor\-res\-pondance de Jacquet-Langlands locale 
mod $\ell$ de Dat, 
d'autre part parce que la proposition \ref{Zcongintro}
s'exprime au moyen de la classifica\-tion à la Zelevinski. 

\subsection{}

Abandonnons maintenant la forme intérieure auxilliaire $\A$, 
et laissons momentanément de côté la corres\-pondance de Jacquet-Langlands. 
Etant donnés 
une représentation ir\-ré\-ductible $\ell$-modu\-lai\-re $\rho$
de $\G$ supposée cuspidale mais pas supercuspidale
et un entier $r\>1$, 
la représentation~de Speh $\Z(\rho,r)$ doit en 
vertu de \cite{MSc} Lemme~9.41 s'exprimer, 
dans le groupe $\RA(\GL_{mr}(\D),\flb)$, 
dans la base des induites de représentations super-Speh~: 
\begin{equation*}
\pi_1\times\dots\times\pi_s,
\quad
\pi_i \text{ représentation super-Speh de } \GL_{m_i}(\D),
\quad
m_1+\dots+m_s=mr.
\end{equation*} 
Nous allons expri\-mer $\Z(\rho,r)$ non pas directement dans cette 
base, mais en fonction de représen\-tations de Speh associées~à une 
représentation 
irréductible cuspidale $\s$ de degré $<m$, 
et surtout 
telle que $k(\s)<k(\rho)$. 
Plus préci\-sé\-ment, posons~:
\begin{equation*}
e = 
\left\{
\begin{array}{ll}
k(\rho) & \text{si $k(\rho)$ est premier à $\ell$,} \\
\ell & \text{sinon.}
\end{array}
\right.
\end{equation*}
Il y a alors une représentation irréductible cuspidale $\s$ de 
$\GL_{me^{-1}}(\D)$ telle que $\rho$ soit isomorphe à un facteur irréductible 
de $\s\times\s\nu_\s^{}\times\dots\times\s\nu_\s^{e-1}$.
Introduisons les séries formelles~: 
\begin{eqnarray*}
\SFS &=& \sum\limits_{r\>0} {(-1)^{r}} \Z(\rho,r) \X^{er}, \\
\SFZ(a,b) &=& \sum\limits 
(-1)^{r} \Z(\s\nu_\s^{a},r) \X^{r},
\quad
a,b\in\ZZ,
\end{eqnarray*}
la seconde somme portant sur les $r\>0$ qui sont congrus à $b-a+1$ mod $e$. 
Notant~$\SFZ$ la matrice carrée de taille $e$ 
de terme général $\SFZ(i+1,j)$ pour $i,j\in\{1,\dots,e\}$, 
on obtient en s'inspirant de \cite{KL} la formule suivante. 

\begin{prop}
\label{FORUMEZD}
Le déterminant $\det(\SFZ)$ est égal à $\SFS$.
\end{prop}

\subsection{}

Si l'on étend par linéarité la correspondance \eqref{PILT} en un morphisme 
de groupes de $\RA(\G,\qlb)$ vers $\RA(\H,\qlb)$, on sait 
(voir \cite{Datj} (1.2.2))
qu'il n'y a pas de morphisme de $\RA(\G,\flb)$ vers $\RA(\H,\flb)$ qui lui 
soit compatible par réduction mod $\ell$.
En d'autres termes, la correspondance induite 
par le théorème \ref{MAINTHEOREM} entre réductions mod $\ell$ de séries 
discrètes entières $\ell$-adiques ne s'étend pas aux groupes de 
Grothendieck. 
Définissons maintenant comme Badules\-cu~\cite{BaduJIMJ} \S 3.1 
un morphisme~:
\begin{equation}
\label{TLJlintro}
\TLJ_\ell : \RA(\H,\qlb) \to \RA(\G,\qlb)
\end{equation} 
trivial sur les indui\-tes paraboliques à partir d'un sous-groupe 
de Levi de $\H$ dont les blocs ne~sont pas tous de taille divisible par $d$, 
et interpolant la réciproque de la correspondance \eqref{PILT} sur les 
induites paraboliques des séries discrètes $\ell$-adiques de sous-groupes 
de Levi dont les blocs sont de taille divisible par $d$
(voir le \S \ref{BaduTLJ}).
Il est naturel de demander s'il y a un morphisme compatible à 
$\TLJ_\ell$ par réduction mod $\ell$.
Le théorème suivant répond à cette question par l'affirmative.

\begin{theo}
\label{caciqueintro}
Il y a un unique morphisme de groupes $\LJ_\ell$ de 
$\RA(\H,\flb)$ vers $\RA(\G,\flb)$ 
qui soit com\-pa\-ti\-ble à $\TLJ_\ell$ par réduction mod $\ell$. 
\end{theo}

Dans le cas où $\G$ est une forme intérieure compacte modulo le centre, 
on se retrouve dans la situation du paragraphe \ref{ParIntro5},
où l'on sait 
que la réponse est oui grâce à la théorie du caractère de Brauer de Dat.
Dans le cas général, cet argument ne suffit plus.
Pour prouver le résultat,~on introduit l'anneau (commutatif)
de Grothendieck~: 
\begin{equation*}
\RA(\D,\flb) = \bigoplus\limits_{m\>0} \RA(\GL_m(\D),\flb)
\end{equation*}
qui est libremement engendré par l'ensemble des représentations 
$\ell$-modulaires super-Speh.
Grâce au corollaire \ref{superSpehIntrotro}, il y a un unique 
morphisme surjectif d'anneaux $\LJ_\ell$ de 
$\RA(\F,\flb)$ vers $\RA(\D,\flb)$ tel qu'on ait l'égalité~:
\begin{equation}
\label{resteafaireintro2}
\Big(\LJ_\ell\circ\r_\ell\Big)(\tp) = \Big(\r_\ell\circ\TLJ_\ell\Big)(\tp)
\end{equation}
pour toute représentation $\ell$-adique $\tp$
entière \textit{et $\ell$-super-Speh} de $\GL_n(\F)$, $n\>1$.
Pour prouver le théorème, il faut alors prouver que 
\eqref{resteafaireintro2} vaut 
pour toute représentation de Speh $\ell$-adique entière $\tp$, 
pas nécessairement $\ell$-super-Speh.
Pour ce faire, il s'agit de décrire 
explicitement la réduction mod $\ell$ de $\tp$ dans la base des 
représentations super-Speh, ce qui se fait de proche en proche grâce à 
la proposition \ref{FORUMEZD} et à la formule de factorisation donnée 
par la proposition \ref{Flyte}.

\section{Préliminaires}
\label{GeEm}

\subsection{}
\label{GeEm1}

Fixons 
un corps localement compact non archi\-médien $\F$ 
de caracté\-ristique résiduelle $p$.
Notons $q$ le cardinal de son corps résiduel.

Fixons une $\F$-algèbre à division centrale $\D$ de dimension finie, 
et de degré réduit noté $d$. 
Pour tout $m\>1$, on note $\Mat_{m}(\D)$ la 
$\F$-algèbre des matrices carrées de taille $m$ à coefficients dans $\D$
et $\GL_{m}(\D)$ le groupe de ses élé\-ments in\-ver\-si\-bles, noté aussi 
$\G_{m}$. 
Celui-ci est un groupe localement profini. 
On convient de noter $\G_{0}$ le groupe trivial. 

Soit $\R$ un corps algébriquement clos de 
caractéristique différente de $p$. 
Pour tout entier $m\>0$, 
notons $\XA(\G_m,\R)$ l'ensemble des classes d'iso\-mor\-phisme de
représen\-ta\-tions ir\-réductibles de $\G_m$
et $\RA(\G_m,\R)$ le groupe de Gro\-then\-dieck des 
re\-pré\-sen\-ta\-tions de lon\-gueur finie de $\G_m$ identifié au groupe 
abélien libre de base $\XA(\G_m,\R)$.
Posons~:
\begin{equation}
\label{vieillenot}
\XA(\R) = \XA(\D,\R) = \bigcup\limits_{m\>0} \XA(\G_{m},\R),
\quad
\RA(\R) = \RA(\D,\R) = \bigoplus\limits_{m\>0} \RA(\G_m,\R).
\end{equation}
(Toutes les représentations considérées dans cet article sont des
représentations lisses de grou\-pes localement profinis.) 
Si $\pi$ est une représentation de longueur finie de $\G_m$, 
l'entier $m$ s'appelle le {\it degré} de $\pi$.
Ceci fait de $\RA(\R)$ un $\ZZ$-module gradué.

Si $\pi$ est une représentation et $\chi$ un caractère 
--- \ie un mor\-phis\-me de groupes à va\-leurs dans $\mult\R$ 
et de noyau ouvert --- de $\G_m$, il existe un unique 
caractère $\mu$ de $\mult\F$ tel que $\chi$ soit égal à $\mu\circ{\rm Nrd}$, 
où ${\rm Nrd}:\G_m\to\mult\F$ désigne la norme réduite. 
On note $\pi\cdot\mu$ ou $\pi\chi$ la représentation tordue 
$g\mapsto\chi(g)\pi(g)$.

\subsection{}

Si $\a=(m_{1},\ldots,m_{r})$ est une {composition} de $m$, 
\ie une famille finie d'entiers positifs de som\-me $m$,
il lui correspond le 
sous-groupe de Levi stan\-dard $\M_{\a}$ de $\G_{m}$ constitué des matrices 
diagonales par blocs de tailles $m_{1},\ldots,m_{r}$ respecti\-ve\-ment, 
identifié à $\G_{m_{1}}\times\cdots\times\G_{m_{r}}$. 
On note $\P_{\a}$ le sous-groupe pa\-ra\-bo\-li\-que de $\G_{m}$ engendré par 
$\M_\a$ et les matrices triangulaires su\-pé\-rieures. 

On fixe une racine carrée de $q$ dans $\R$.
On note $\ip_\a$ le foncteur d'in\-duc\-tion para\-bo\-li\-que
(norma\-li\-sé relativement au choix de cette racine)
de $\M_{\a}$ à $\G_{m}$ le long de $\P_{\a}$,
et on note 
$\rp_\a$ son adjoint à gau\-che, \ie le foncteur de restriction 
para\-bo\-li\-que lui correspondant.
Ces foncteurs sont exacts et préservent le fait d'être de longueur finie.

Si, pour chaque $i\in\{1,\ldots,r\}$, on a une 
représentation $\pi_{i}$ de $\G_{m_i}$, on note~: 
\begin{equation}
\label{VentreDieu}
\pi_1\times\cdots\times\pi_r=\ip_{\a}(\pi_1\otimes\cdots\otimes\pi_r).
\end{equation}
Si les $\pi_{i}$ sont de longueur finie, la semi-simplifiée de cette induite
ne dépend que des semi-sim\-plifiées de $\pi_{1},\dots,\pi_{r}$.
Ceci fait de $\RA(\R)$ une $\ZZ$-al\-gè\-bre com\-mu\-tative gra\-duée 
(voir \cite{MSc} Proposition 2.6). 

Au moyen des foncteurs de restriction parabolique, on définit également une 
comultiplication~:
\begin{equation}
\label{VentreDieu2}
\com : \pi\mapsto\sum\limits_{k=0}^{m} \rp_{(k,m-k)}(\pi) \in \RA(\R)\otimes\RA(\R)
\end{equation}
faisant de $\RA(\R)$ une $\ZZ$-bigèbre commutative. 

Comme dans \cite{MSi}, on adopte encore les notations 
$\ip_\a$, $\rp_\a$ et $\times$
quand $\D$ est rem\-pla\-cé par un corps fini de caractéristique $p$. 

\section{Rappels et compléments sur les représentations cuspidales}

On fixe un entier $m\>1$ et on pose $\G=\GL_m(\D)$.
Pour les notions de représentation irréduc\-tible cuspidale et supercuspidale de 
$\G$, on renvoie le lecteur à \cite{MSc}.

Aux paragraphes \ref{PARA21} et \ref{P32},
$\R$ est un corps algébriquement clos de 
caractéristique différente de $p$.

Dans cette section, nous rappelons brièvement comment décrire les 
représentations cuspidales irréductibles de $\G$ 
en termes de représentations 
de certains sous-groupes ouverts compacts mod le centre
\cite{MSt}.
Grâce à cette description, on associe à toute 
représentation irréductible cuspidale 
un certain nombre d'invariants numériques, et on définit l'invariant $w$
au paragraphe \ref{PARADEFW}.

\subsection{}
\label{PARA21}

Rappelons quelques faits tirés de \cite{MSt} sur les 
$\R$-représentations irréductibles cuspidales de $\G$.
D'abord, il y a une correspondance bijective naturelle~:
\begin{equation}
\label{sichuan}
[\G,\rho] \leftrightarrow [\J,\l]
\end{equation}
entre classes 
inertielles de $\R$-représentation irréductible cuspidale de $\G$ et 
classes de $\G$-conjugai\-son d'objets appelés types simples maximaux 
de $\G$ (\cite{MSt} \S3).
Plus précisément, la classe d'inertie de $\rho$ et la classe de conjugaison 
de $(\J,\l)$ se correspondent par \eqref{sichuan} 
si et seulement si la restric\-tion de $\rho$ à
$\J$ possède une sous-représentation isomorphe à $\l$. 

Un type simple maximal de $\G$ est une paire $(\J,\l)$ formée d'un 
sous-groupe ouvert compact $\J$ de $\G$ et d'une $\R$-repré\-sen\-ta\-tion 
irréductible $\l$ de $\J$ dont la construction est effectuée en \cite{MSt} \S2.
Résumons-en brièvement les prin\-ci\-pa\-les étapes. 

D'abord, on part d'une strate simple $[\La,n_\La,0,\b]$ dans la $\F$-algèbre 
$\Mat_m(\D)$ 
et d'un $\R$-caractère sim\-ple $\t\in\Cc(\La,0,\b)$ d'un sous-groupe ouvert 
compact $\H^1=\H^1(\b,\La)$ de $\G$.
Il y a un sous-groupe ouvert compact $\J^1=\J^1(\b,\La)$ de $\G$ contenant 
et normalisant $\H^1$, possédant une unique re\-pré\-sen\-tation irréductible $\eta$ 
dont la restriction à $\H^1$ contienne $\t$.

La représentation $\eta$ se prolonge en une représentation irréductible $\k$ d'un 
sous-groupe ouvert compact $\J=\J(\b,\La)$ de $\G$ contenant 
et normalisant $\J^1$, de même ensemble d'entrelacement que $\eta$~;
un tel prolongement $\k$ s'appelle une \textit{$\b$-extension} de $\eta$.

On suppose que $\J\cap\mult\B$ est un sous-groupe compact
maximal de $\mult\B$.
On fixe un isomorphisme d'algèbres entre le centralisateur $\B$ de $\E=\F[\b]$ 
dans $\Mat_m(\D)$ et une $\E$-algèbre $\Mat_{m'}(\D')$ pour un $m'\>1$ 
et une $\E$-algèbre 
à division centrale $\D'$ convenables, iden\-ti\-fiant $\J\cap\mult\B$ au
sous-groupe compact maximal standard de $\GL_{m'}(\D')$.

Le groupe $\J$ est égal à $(\J\cap\mult\B)\J^1$, et on a des isomorphismes de 
groupes~: 
\begin{equation*}
\label{iso}
\J/\J^1 \simeq (\J\cap\mult\B)/(\J^1\cap\mult\B) \simeq \GL_{m'}(\kd),
\end{equation*}
$\kd$ étant le corps résiduel de $\D'$ et le second isomorphisme 
étant induit par l'isomorphisme d'algè\-bres fixé pré\-cédemment. 
Notons $\GB$ ce dernier groupe et fixons une re\-présentation irréductible 
cus\-pidale $\s$ de $\GB$.
Elle définit, par inflation, une re\-pré\-sen\-tation irréductible de $\J$ 
triviale sur $\J^1$,
encore notée $\s$. 
Alors la paire $(\J,\k\otimes\s)$ est un type simple maximal de $\G$, et tous 
sont construits de cette façon.

Soit $\rho$ une $\R$-représentation irréductible cuspidale de $\G$ dont la 
classe inertielle $[\G,\rho]$
corres\-pond à la classe de conjugaison d'un type simple maximal $(\J,\l)$. 
Le groupe de Galois de $\kd$ sur $\ke$ (où $\ke$ est le corps résiduel de
$\E$) agit sur les représentations
de $\GB$~; on note~:
\begin{equation}
\label{DEFSRHO}
s(\rho) = s(\s)
\end{equation}
l'ordre du stabilisateur de $\s$ dans ce groupe de Galois.
Notant $\nu$ le caractère non ramifié ``valeur absolue de la norme réduite'', 
le caractère $\nu_\cuspi = \nu^{s(\cuspi)}$ a la propriété importante suivante. 

\medskip

\begin{tabular}{lp{14cm}}
{(\ref{PARA21}.1)} & 
\textit{Si $\cuspi'$ est une représentation irréductible cuspidale de degré $m'\>1$, 
l'induite $\cuspi\times\cuspi'$ est réductible si et seulement si $m'=m$ et 
$\cuspi'$ est isormorphe à $\cuspi\nu_\cuspi^{}$ ou à $\cuspi\nu_{\cuspi}^{-1}$.} \\
\end{tabular}

\medskip

\noindent
Quand $\G$ est déployé, \ie quand $\D$ est égale à $\F$, 
le groupe de Galois $\Gal(\kd/\ke)$
est trivial et on a toujours $s(\rho)=1$.

Notons $n(\rho)$ 
le nombre de caractères non ramifiés $\chi$ de $\G$ tels que la 
représentation tordue $\rho\chi$ soit isomorphe à $\rho$.
Cet entier a la propriété suivante. 

\medskip

\begin{tabular}{lp{14cm}}
{(\ref{PARA21}.2)} & 
\textit{Si $\chi$ est un caractère non ramifié de $\G$, on a
$\rho\chi\simeq\rho$ si et seulement si $\chi^{n(\rho)}=1$.} \\
\end{tabular}

\medskip

Notons enfin $f(\rho)$ 
le quotient de $md$ par l'indice de ramification 
de $\E$ sur $\F$, qui est un multiple de $s(\rho)$. 
Ces trois entiers sont indépendants des choix 
effectués dans la construction de $(\J,\l)$ et ils ne dépendent 
que de la classe inertielle de $\rho$.
Notant $\ell$ l'exposant caractéristique de $\R$, 
il sont liés par la relation suivante. 

\medskip

\begin{tabular}{lp{14cm}}
{(\ref{PARA21}.3)} & 
\textit{L'entier $n(\rho)$ est le plus grand diviseur de $f(\rho)s(\rho)^{-1}$ 
premier à $\ell$.} \\
\end{tabular}

\medskip

\noindent
Notamment, 
lorsque $\R$ est de caractéristique $0$, on a simplement la relation
$f(\rho) = n(\rho)s(\rho)$. 

\subsection{}
\label{P32}

Pour tout entier $n\>1$, l'induite parabolique~: 
\begin{equation*}
\label{INDRHON}
\cuspi\times\cuspi\nu_{\cuspi}^{}\times\dots\times\cuspi\nu_{\cuspi}^{n-1}
\end{equation*}
a un unique sous-quotient irréductible résiduelle\-ment non dégénéré 
au sens de \cite{MSc} Section 8.
Il apparaît avec multiplicité $1$, et on le note~:
\begin{equation*}
\label{DEFSP}
\Sp(\cuspi,n).
\end{equation*}
Si $n\>2$, 
pour que ce sous-quotient soit cuspidal, 
il faut et il suffit que 
$\R$ soit de caractéristique $\ell>0$ et qu'il existe un entier $r\>0$ tel que 
$n=\omega(\cuspi)\ell^r$, avec~:
\begin{equation}
\label{DEFOMEGA}
\omega(\cuspi) = 
\text{l'ordre de $q^{f(\rho)}$ dans $(\ZZ/\ell\ZZ)^\times$}
\end{equation}
(voir \cite{MSc} Proposition 6.4).
D'après \cite{MSc} Corollaire 6.12, si la représentation 
$\cuspi$ n'est pas super\-cuspidale, on a $\omega(\cuspi)=1$. 

Lorsque $\R$ est de caractéristique $\ell$,
il y a d'après \cite{MSc} Théorème 6.14 un unique entier naturel~:
\begin{equation*}
\label{DEFkrho}
k(\cuspi) \>1
\end{equation*}
et une re\-pré\-sen\-tation irréductible supercuspidale 
$\alpha$ de $\GL_{mk(\cuspi)^{-1}}(\D)$ tels que $\cuspi$ soit isomorphe à 
$\Sp(\alpha,k(\cuspi))$.
La représentation $\a$ n'est pas unique en général mais, 
si $\pi$ est irréductible supercus\-pidale et si 
$\cuspi$ est iso\-morphe à 
$\Sp(\pi,k(\cuspi))$, il y a un $i\in\ZZ$ tel que $\pi$ soit 
isomorphe à $\alpha\nu_\alpha^i$.

\begin{rema}
La représentation $\cuspi$ est supercuspidale si et seulement si $k(\cuspi)=1$.
\end{rema}

De façon analogue, l'induite parabolique 
$\s^{\times n}=\s\times\dots\times\s$ ($n$ fois), 
qui est une représentation du groupe $\GL_{m'n}(\kd)$, 
a un unique sous-quotient irré\-duc\-ti\-ble non dégénéré.
On le note $\sp(\s,n)$~; 
il apparaît avec multiplicité $1$ dans $\s^{\times n}$.
Si $n\>2$, 
pour que ce sous-quotient soit cuspidal, 
il faut et il suffit que 
$\R$ soit de caractéristique $\ell>0$ et qu'il existe un entier $r\>0$ tel que 
$n=\omega(\s)\ell^r$, avec~:
\begin{equation*}
\label{DEFOMEGAfini}
\omega(\s) = \text{l'ordre de $q_{\D'}^{m'}$
dans $(\ZZ/\ell\ZZ)^\times$}
\end{equation*}
où $q_{\D'}^{}$ désigne le cardinal du corps résiduel $\kd$.
Remarquons que $q_{\D'}^{m'}$ est égal à $q^{f(\rho)}$,
et donc que $\omega(\rho)=\omega(\s)$.
Lorsque $\R$ est de caractéristique $\ell>0$,
il y a un unique entier naturel~:
\begin{equation*}
\d(\s) \> 1
\end{equation*}
et une unique représentation irré\-duc\-tible supercuspidale 
$\tau$ de $\GL_{m'k(\s)^{-1}}(\kd)$ 
tels que la repré\-sen\-tation $\s$ soit iso\-morphe à $\sp(\tau,\d(\s))$.
On a le résultat suivant. 

\begin{lemm}
\label{EGkrks}
On a $\d(\rho)=\d(\s)$.
\end{lemm}

\begin{proof}
Posons $\d=\d(\rho)$ et écrivons $\rho$ sous la forme $\Sp(\alpha,k)$. 
Fixons un type simple maximal $(\J_0,\k_0\otimes\s_0)$ contenu dans $\a$. 
D'après, par exemple, la preuve de \cite{MSc} Lemme 6.1, 
on peut choisir $\s_0$ de sorte que $\s$ soit égale à $\sp(\s_0,k)$. 
D'après \cite{MSc} Proposition 6.10, la re\-présen\-tation $\s_0$ est 
supercuspidale.
Par unicité de $\d(\s)$, on en déduit que $\d(\s)=\d$.
\end{proof}

\begin{rema}
\label{gavroche}
On en déduit que $\d(\rho)$ divise $m'$, et pas seulement $m$.
\end{rema}

Notons (voir \cite{MSc} Proposition 6.10) 
que $\rho$ est supercuspidale si et seulement si $\s$ l'est. 

\subsection{}
\label{P13}

Fixons un nombre premier $\ell$ différent de $p$
et une clôture algébrique $\qlb$ du corps des nombres $\ell$-adiques.
On note $\zlb$ son anneau d'entiers et $\flb$ son corps résiduel.
Une représentation \textit{$\ell$-adique} 
est une représentation à coefficients dans $\qlb$. 

Pour les notions de représentation $\ell$-adique entière 
et de réduction mod $\ell$, on renvoie le lecteur à \cite{MSc}
(voir aussi \cite{Vigb,Vigw}). 
Si $\tp$ est une représentation $\ell$-adique irréductible entière 
dans $\RA(\qlb)$, on note $\r_\ell(\tp)$ sa réduction mod $\ell$ 
dans $\RA(\flb)$.

\begin{defi}
Deux $\qlb$-représentations irréductibles entières de $\G$ 
sont dites \textit{congruentes} (modulo $\ell$) si elles ont 
la même réduction mod $\ell$, \ie la même image par $\r_\ell$.
\end{defi}

Soit $\cuspit$ une représentation $\ell$-adique irréductible cuspida\-le
entière de $\G$.
D'après le paragraphe \ref{P14}, si $\rho$ est un facteur irréductible, 
et $a$ la longueur, de la réduction mod $\ell$ de $\rt$, alors on a~:
\begin{equation}
\label{Redrt}
\r_\ell(\cuspit) = \cuspi+\cuspi\nu+\dots+\cuspi\nu^{a-1}.
\end{equation}
D'après \cite{MSt} Paragraphe 3.5, l'entier $a=a(\rt)$ a les propriétés 
suivantes. 

\begin{prop}
\label{propa}
\begin{enumerate}
\item
Il y a un entier $u\>0$ tel que $a$ vérifie la relation~:
\begin{equation}
\label{nan}
n(\cuspit) = an(\cuspi)\ell^{u}.
\end{equation}
Plus précisément, $n(\cuspi)$ est le plus grand diviseur de $n(\cuspit)a^{-1}$ 
premier à $\ell$.
\item
Si $a>1$, alors le plus grand diviseur de $a$ premier à $\ell$ est égal à~:
\begin{equation}
\label{DEFERHO}
\ee(\cuspi) = 
\text{l'ordre de $q^{n(\cuspi)}$ dans $(\ZZ/\ell\ZZ)^\times$}.
\end{equation}
\item
On a l'égalité $s(\cuspi)=as(\cuspit)$.
\end{enumerate}
\end{prop}

L'entier $\ee(\cuspi)$ défini par l'identité \eqref{DEFERHO} a la propriété suivante. 

\begin{lemm}
\label{STEP0}
Soit un entier $i\in\ZZ$.
Pour que $\rho\nu^i\simeq\rho$, il faut et il suffit que $\e(\rho)$ 
divise $i$.
\end{lemm}

\begin{proof}
D'après (\ref{PARA21}.2), 
les représentations $\rho\nu^i$ et $\rho$ sont isomorphes 
si et seulement si $\nu^{n(\rho)i}=1$. 
L'ordre de $\nu$ étant égal à 
l'ordre de $q$ dans $(\ZZ/\ell\ZZ)^\times$, qu'on note $e$, 
ceci équivaut à dire que $e$ divise $n(\rho)i$.
Il ne reste plus qu'à remarquer que~:
\begin{equation*}
\e(\rho) = \frac{e}{(e,n(\rho))}
\end{equation*}
pour conclure. 
\end{proof}

\begin{coro}
\label{COROEPS1}
On a $\rho\nu\simeq\rho$ si et seulement si l'entier $\e(\rho)$ est égal à $1$.
\end{coro}

\subsection{}
\label{PARA22}

Fixons une extension $\kk$ de $\kd$ de degré $m'$, 
et notons $\X$ l'ensemble des $\xt\in\mult\kk$ de degré $m'$ sur 
$\kd$.  
D'après Green \cite{Green}, il y a une application surjective~:
\begin{equation}
\label{green}
\xt\mapsto\boldsymbol{\st}(\xt)
\end{equation}
de $\X$ vers l'ensemble des classes de représen\-ta\-tions irréduc\-ti\-bles 
cuspidales $\ell$-adiques de $\GB$~;
les antécédents de $\boldsymbol{\st}(x)$ sont les conjugués de $x$
sous $\Gal(\kk/\kd)$.
Pour $x\in\mult\kk$, notons $[x]$ l'orbite de $x$ sous $\Gal(\kk/\ke)$ et~: 
\begin{equation*}
\deg(x) = {\rm card}([\xt])
\end{equation*}
le degré de $x$ sur $\ke$.
Notons $d'$ le degré réduit de $\D'$ sur $\E$.

\begin{lemm}
\label{lemsm}
Pour $x\in\X$, soit $\st$ la représentation cuspidale lui
correspondant par \eqref{green}.
On a la relation~: 
\begin{equation}
\label{DEGX}
\deg(x) = \frac{m'd'}{s(\st)}
\end{equation}
et $s(\st)$ est premier à $m'$.
\end{lemm}

\begin{proof}
Notons $\phi$ l'automorphisme de Frobenius $x\mapsto x^{q_\E}$, 
où $q_{\E}$ désigne le cardinal du corps résiduel $\ke$.
Pour $k\in\ZZ$, on a~:
\begin{equation*}
\st^{\phi^k}\simeq\st
\quad\Leftrightarrow\quad
\text{il existe $l\in\ZZ$ tel que $x^{q_\E^k}=x^{q_{\D'}^l}$}.
\end{equation*}
Si l'on note $[\st]$ l'orbite de $\st$ sous $\Gal(\kd/\ke)$, 
on en déduit que~:
\begin{equation*}
{\rm card}([\st]) 
= \frac{d'}{s(\st)}
= (d',\deg(x)).
\end{equation*}
Par ailleurs, si $n$ est l'ordre de $x$,
alors l'ordre de $q_\E$ dans $(\ZZ/n\ZZ)^\times$ est $\deg(x)$, 
tandis que l'ordre de $q_{\D'}$ dans $(\ZZ/n\ZZ)^\times$ est $m'$.
On en déduit le résultat voulu. 
\end{proof}

\begin{coro}
\label{eponine}
L'entier $s(\st)$ est premier à $m$.
\end{coro}

\begin{proof}
Notons $g$ le degré de $\E$ sur $\F$, de sorte que~:
\begin{equation}
\label{ddpmmp}
d' = \frac{d}{(d,\gg)},
\quad
m = m' \cdot \frac{g}{(d,\gg)}.
\end{equation}
L'entier $s(\st)$ divise $d'$, et il est premier à $m'$ d'après le lemme 
\ref{lemsm}~; le résultat s'ensuit. 
\end{proof}

D'après Dipper et James \cite{Dippd2,DJ,James}, si $x\in\X$,
la réduction modulo $\ell$ de $\boldsymbol{\st}(\xt)$ est 
irréductible et cuspidale, et ne dépend que de la partie $\ell$-régulière 
de $\xt$, \ie de l'unique $y\in\mult\kk$ d'ordre premier à $\ell$
tel que l'ordre de $\xt y^{-1}$ soit 
une puis\-san\-ce de $\ell$.
Ceci définit une application surjective~:
\begin{equation}
\label{james}
y\mapsto\boldsymbol{\s}(y)
\end{equation}
de l'ensemble $\Y$ des 
parties $\ell$-régulières des éléments de $\X$ vers celui des (classes de)
représen\-ta\-tions irréductibles cuspidales $\ell$-modulaires de $\GB$~;
l'ensemble des antécédents de $\boldsymbol{\s}(y)$ est l'orbite de $y$ sous 
le groupe de Galois de $\kk$ sur $\kd$.

\begin{lemm}
Pour $y\in\Y$, soit $\s$ la représentation cuspidale lui
correspondant par \eqref{james}.
On a la relation~: 
\begin{equation}
\label{DEGY}
\deg(y) = \frac{m'}{\d(\s)}\cdot\frac{d'}{s(\s)}
\end{equation}
et $s(\s)$ est premier à $m'\d(\s)^{-1}$.
\end{lemm}

\begin{proof}
Si l'on note $m''$ le cardinal de l'orbite de $y$ sous 
$\Gal(\kk/\kd)$, alors l'entier $\d(\s)$ défini au paragraphe \ref{P32}
vérifie la relation $m'=\d(\s)\cdot m''$.
Comme dans le lemme précédent, on en déduit la relation \eqref{DEGY} et que 
$s(\s)$ est premier à $m''=m'\d(\s)^{-1}$.
\end{proof}

De façon analogue au corollaire \ref{eponine},
on en déduit le corollaire suivant.

\begin{coro}
\label{eponine3}
L'entier $s(\s)$ est premier à $m\d(\s)^{-1}$.
\end{coro}

Soit $x\in\X$, et soit $y\in\Y$ la partie $\ell$-régulière de $x$.
Soit $\st$ la représentation cus\-pidale $\ell$-adi\-que 
correspondant à $x$ et $\s$ sa réduction modulo $\ell$, qui correspond à $y$.
On pose~:
\begin{eqnarray*}
a(\st) &=& \frac{s(\s)}{s(\st)}, \\
w(\st) &=& \frac{\deg(x)}{\deg(y)} 
= \d(\s)a(\st).
\end{eqnarray*}
On a les propriétés suivantes. 

\begin{lemm}
\label{CN}
On a $(w(\st),m')=(w(\st),m)=\d(\s)$ et $(w(\st),s(\s))=a(\st)$.
\end{lemm}

\begin{proof}
Comme $s(\st)$ est premier à $m'$, il est aussi premier à $\d(\s)$. 
Multipliant par $a(\st)$, on en déduit que $(w(\st),s(\s))=a(\st)$.
Ensuite, $mk(\s)^{-1}$ étant premier à $s(\s)$, il est aussi premier à 
$a(\st)$.
Multipliant par $\d(\s)$, on en déduit que $(w(\st),m)=\d(\s)$.
L'entier $k(\s)$ divisant à la fois $w(\s)$ et $m'$, il s'ensuit 
que $(w(\st),m')=\d(\s)$.
\end{proof}

Notons $\e(\s)$ l'ordre de~:
\begin{equation}
\label{QEKDY}
q_\E^{\deg(y)\cdot\d(\s)}
\end{equation}
dans $(\ZZ/\ell\ZZ)^\times$. 

\begin{lemm}
\label{CN2}
Si $x\neq y$,
alors le plus grand diviseur de $a(\st)$ premier à $\ell$ est $\e(\s)$.
\end{lemm}

\begin{proof}
Dans $(\ZZ/n\ZZ)^\times$ (où $n$ désigne l'ordre de $x$), 
l'ordre de \eqref{QEKDY} est~:
\begin{equation*}
\frac{\deg(x)}{(\deg(y) \d(\s),\deg(x))}=a(\st).
\end{equation*}
Comme $x\neq y$, l'entier $n$ est divisible par $\ell$.
En projetant sur $(\ZZ/\ell\ZZ)^\times$, on en déduit que le plus grand 
diviseur de $a(\st)$ premier à $\ell$ est $\e(\s)$.
\end{proof}

\begin{rema}
\label{FLO}
\begin{enumerate}
\item
La condition $x\neq y$ signifie que $\s$ n'est pas supercuspidale. 
\item
Si $\rho$ contient le type simple maximal $\k\otimes\s$,
alors les relations \eqref{DEGY} et (\ref{PARA21}.3) entraînent que $\e(\s)$ 
est égal à l'entier $\e(\rho)$ défini par \eqref{DEFERHO}. 
\end{enumerate}
\end{rema}

\subsection{}
\label{PARADEFW}

Soit $\cuspit$ une représentation $\ell$-adique irré\-duc\-tible cuspida\-le
entière de $\G$ comme au paragraphe \ref{P13}.
Soient $\rho$ un facteur irré\-duc\-tible de $\r_\ell(\rt)$ et $a=a(\rt)$ sa longueur.
On pose~:
\begin{equation}
\label{DEFw}
w(\cuspit) = k(\cuspi) a(\cuspit).
\end{equation}
La représentation $\cuspit$ est $\ell$-super\-cuspidale au sens de la 
définition \ref{DEFlssintro} si et seulement si $w(\cuspit)=1$. 

\begin{rema}
Si $\rt$ contient un type simple maximal se décomposant sous 
la forme $\kt\otimes\st$,
alors \eqref{DEFSRHO} et le lemme \ref{EGkrks} entraînent
les égalités $a(\rt)=a(\st)$ et $w(\rt)=w(\st)$. 
\end{rema}

D'après le paragraphe \ref{PARA21}, 
il y a une représentation ir\-ré\-ductible super\-cuspidale $\alpha$ de 
$\G_{mk(\rho)^{-1}}$ tels que $\cuspi$ soit isomorphe à $\Sp(\alpha,k(\rho))$.

\begin{lemm}
\label{NASR}
On suppose que $\rho$ n'est pas supercuspidale. 
Alors l'entier $\ee(\cuspi)$ est égal au 
plus grand diviseur commun à $\ee(\alpha)$ et $s(\alpha)$.
\end{lemm}

\begin{proof}
Soit $i\in\ZZ$.
D'après le lemme \ref{STEP0}, 
c'est un multiple de $\ee(\cuspi)$ si et seulement si 
$\cuspi\nu^i$ est isomorphe à $\cuspi$.
Compte tenu de l'unicité du sous-quotient 
résiduellement non dégénéré $\Sp(\alpha,k)$, avec $k=k(\rho)>1$
ceci se produit si et seulement si $\cuspi$ est un sous-quotient de~:
\begin{equation*}
\alpha\nu^i\times\alpha\nu_\alpha^{}\nu^i\times\dots\times
\alpha\nu_{\alpha}^{k-1}\nu^i.
\end{equation*}
D'après la propriété d'unicité de $\alpha$, 
l'entier $\ee(\cuspi)$ divise $i$ si et seulement s'il y a un 
$t\in\ZZ$ tel que $\alpha\nu^i$ soit isomorphe à $\alpha\nu_{\alpha}^t$, ce qui, 
d'après le lemme \ref{STEP0}, équivaut à 
$i\in \ee(\alpha)\ZZ+s(\alpha)\ZZ$.
\end{proof}

\begin{lemm}
\label{wadm}
Si $w(\cuspit)>1$, alors son plus grand diviseur premier à $\ell$ est égal à 
$\ee(\alpha)$.
\end{lemm}

\begin{proof}
Supposons d'abord que $a(\cuspit)>1$, de sorte que son plus grand diviseur 
premier à $\ell$ est $\ee(\cuspi)$. 
Si $k(\cuspi)=1$, alors $\cuspi$ est égale à $\alpha$ et le résultat est immédiat. 
Sinon, le résultat découle du lemme \ref{NASR} et du fait que le plus grand 
diviseur de $k(\cuspi)$ premier à $\ell$ est égal à l'entier $\omega(\a)$
défini par \eqref{DEFOMEGA}.

Supposons que $k(\cuspi)>1$ et que $a(\cuspit)=1$.
Comme $\rho$ n'est pas supercuspidale, le lemme \ref{CN2} implique que 
$\ee(\cuspi)$ est égal à $1$, 
et le résultat découle à nouveau du lemme \ref{NASR} et du fait que le plus grand 
diviseur de $k(\cuspi)$ premier à $\ell$ est égal à $\omega(\a)$.
\end{proof}

Isolons le corollaire suivant, qui nous sera utile plus tard. 

\begin{coro}
\label{CoroIsole}
Supposons que $k(\cuspi)>1$ et que $a(\cuspit)=1$.
Alors $\ee(\cuspi)=1$.
\end{coro}

\section{Comptage}
\label{Section4}

On fixe un entier $m\>1$ et on pose $\G=\GL_m(\D)$.
Dans cette section, nous prouvons le critère de $\ell$-supercuspidalité
(\ie la proposition \ref{CongCusp}) 
et son complément, la proposition \ref{CalculW}.

\subsection{Preuve de la proposition \ref{CongCusp}}

Soit $\rt$ une représentation irréductible cuspidale $\ell$-adique entière de $\G$, 
et soit $\Oo(\rt)$ l'ensemble des classes inertielles $[\G,\rt']$ 
de représentations irréductibles cuspi\-dales $\ell$-adiques de $\G$ 
congrues à $[\G,\rt]$ modulo $\ell$, \ie
(voir le paragraphe \ref{P14} pour la notation) 
telles que~:
\begin{equation*}
\r_\ell([\G,\rt'])=\r_\ell([\G,\rt]).
\end{equation*}
Fixons un type simple maxi\-mal $(\J,\lt)$ dans la classe de $\G$-conjugaison 
correspondant à $[\G,\rt]$, 
et notons $\l$ la réduction de $\lt$ modulo $\ell$.
Alors $(\J,\l)$ est un $\flb$-type simple maximal correspondant à la classe inertielle 
de la représentation $\rho$ apparaissant dans \eqref{redcusp}, et l'entier 
$a=a(\rt)$ est l'indice du $\G$-normalisateur de $(\J,\lt)$ dans celui de 
$(\J,\l)$
(voir \cite{MSt} Section 3).

L'ensemble $\Oo(\rt)$ s'identifie donc à l'ensemble des classes de
$\G$-conjugaison de $\qlb$-types simples maximaux $(\J',\lt')$ 
tels que, si l'on note $\l'$ la réduction de $\lt'$ modulo $\ell$, on ait~:
\begin{enumerate}
\item 
les $\flb$-types simples maximaux $(\J',\l')$ et $(\J,\l)$ sont conjugués sous
$\G$~; 
\item
on a $(\N_\G(\J',\l'):\N_\G(\J',\lt'))=(\N_\G(\J,\l):\N_\G(\J,\lt))$~;
\end{enumerate}
où $\N_\G(\J,\l)$ désigne le normalisateur de $(\J,\l)$ dans $\G$.
Quitte à conjuguer,
on peut donc supposer que $\J'=\J$ et $\l'=\l$~; 
l'ensemble $\Oo(\rt)$ s'identifie donc à l'ensemble $\Tt(\J,\lt)$ des classes de 
$\N_\G(\J,\l)$-conjugaison de $\qlb$-types simples maximaux $(\J,\lt')$ de $\G$ 
tels que~:
\begin{enumerate}
\item
les représentations $\lt'$ et $\lt$ sont congruentes modulo $\ell$~;
\item
les paires $(\J,\lt')$ et $(\J,\lt)$ ont le même normalisateur dans $\G$.
\end{enumerate}
Fixons une décomposition de $\lt$ sous la forme $\kt\otimes\st$
et un isomorphisme de groupes de $\J/\J^1$ sur $\GB$
(voir le paragraphe \ref{PARA21}).
Le foncteur~:
\begin{equation*}
\widetilde{\tau}\mapsto\kt\otimes\widetilde{\tau}
\end{equation*}
définit une bijection entre les représentations irréducti\-bles 
cus\-pidales de $\GB$ et les 
ty\-pes simples maxi\-maux de $\G$ définis sur $\J$ et contenant $\kt$.
D'après \cite{SeSt2} Theorem 7.2,
sa réciproque induit une bijection de $\Tt(\J,\lt)$ sur l'en\-semble $\Cc(\st)$ des 
orbites, sous l'action du groupe de Galois $\Gal(\kd/\ke)$, 
de représentations irréducti\-bles cus\-pidales $\st'$
de $\GB$ telles que~: 
\begin{enumerate}
\item
les représentations $\st'$ et $\st$ sont congruentes modulo $\ell$~;
\item 
les orbites de $\st'$ et de $\st$ sous $\Gal(\kd/\ke)$ ont le même cardinal.
\end{enumerate}
Si $\xt\in\X$ correspond à $\st$, alors \eqref{green} 
induit une bijection entre $\Cc(\st)$ et l'ensemble $\EuScript{K}(\xt)$ 
des orbites des éléments $\xt'\in\X$, sous le groupe de 
Galois de $\kk$ sur $\ke$, tels que~:
\begin{enumerate}
\item
$x'$ et $x$ ont la même partie $\ell$-régulière~;
\item
les $\Gal(\kk/\ke)$-orbites de $x'$ et de $x$ ont le même cardinal,
\ie que $\deg(\xt')=\deg(x)$.
\end{enumerate}
Remarquons que cette dernière condition signifie que $\xt$, $\xt'$ ont le 
même stabilisateur sous l'action de $\Gal(\kk/\ke)$.
En prenant l'intersection avec $\Gal(\kk/\kd)$, on voit que tout 
$x'\in\mult\kk$ vérifiant cette condition appartient à $\X$.
On obtient ainsi une bijection entre $\Oo(\rt)$ et $\EuScript{K}(\xt)$~; 
on a donc prouvé le résul\-tat suivant. 

\begin{prop}
L'ensemble $\Oo(\rt)$ est fini, et son cardinal $t(\rt)$ est 
le nombre de $\Gal(\kk/\ke)$-orbites des $\xt'\in\mult\kk$ tels que 
$x,x'$ ont la même partie $\ell$-régulière et le même degré 
sur $\ke$.
\end{prop}

Écrivons $\xt$ sous la forme $yz$ où $y$ est d'ordre premier à $\ell$ et
$z$ d'ordre une puissance de $\ell$
(donc $y$ est la partie $\ell$-régulière de $x$).
L'application~:
\begin{equation}
\label{phichi}
z'\mapsto yz'
\end{equation}
est une bijection entre le $\ell$-sous-groupe de Sylow $\P_{\ell}$ de $\mult\kk$ 
et l'ensemble des éléments 
$x'\in\mult\kk$ dont la par\-tie $\ell$-régulière est $y$.
Étant donnés $z'\in\P_{\ell}$ et $k\in\ZZ$, remarquons que~:
\begin{equation}
\label{ponsonduterrail}
(yz')^{q_{\E}^{k}}=yz'
\quad\Leftrightarrow\quad 
\text{$y^{q_{\E}^{k}}=y$ et $(z')^{q_{\E}^{k}}=z'$}.
\end{equation}
Notons $\kk_{1}$ l'extension de $\ke$ engendrée par $y$.
Pour $z'\in\P_{\ell}$, notons $[\![z']\!]$ son orbite sous le 
groupe de Galois de $\kk$ sur $\kk_{1}$ et~:
\begin{equation*}
\deg_{1}(z') = {\rm card}([\![z']\!])
\end{equation*}
son degré sur $\kk_{1}$.
D'après \eqref{ponsonduterrail}, les éléments $yz$, $yz'$ ont le même 
degré sur $\ke$ si et seulement si~:
\begin{equation}
\label{Bateman}
\deg_{1}(z) = \deg_{1}(z').
\end{equation}
Notons $\Pp(z)$ l'ensemble des $[\![z']\!]$ 
pour $z'\in\P_\ell$ vérifiant \eqref{Bateman}.
On a prouvé le résultat suivant. 

\begin{lemm}
L'application \eqref{phichi} induit une bijection de $\Pp(z)$ sur 
$\EuScript{K}(\xt)$.  
\end{lemm}

Il ne nous reste plus qu'à calculer $\deg_{1}(z)$ en 
fonction des invariants associés à $\rt$.
Comme on a $a(\rt)=a(\st)$ et $\d(\rho)=\d(\s)$, et compte tenu de 
\eqref{DEGX} et \eqref{DEGY}, on en déduit que~:
\begin{equation}
\label{DEGZ}
\deg_{1}(z) = 
\frac{\deg(x)}{\deg(y)} = 
\d(\rho) a(\rt) = w(\rt)=w(\st).
\end{equation}
On obtient donc le résultat suivant. 

\begin{lemm}
\label{L3}
L'entier $w(\rt)t(\rt)$ est le nombre de $z'\in\P_\ell$ de degré $w(\rt)$ 
sur $\kk_{1}$.
\end{lemm}

Compte tenu de la relation $n(\rt)s(\rt)=f(\rt)$ 
(voir (\ref{PARA21}.3)),
l'extension de $\kk_{1}$ de degré $w(\rt)$ 
est de cardinal~: 
\begin{equation*}
q^{n(\rt)}.
\end{equation*}
On en déduit l'inégalité $t(\rt)\<c(\rt)$, et cette inégalité est une 
égalité si et seulement si $w(\rt)=1$, \ie si et seulement si $\rt$ est 
$\ell$-supercuspidale. 
Ceci met fin à la preuve de la proposition \ref{CongCusp}.

\subsection{Preuve de la proposition \ref{CalculW}}
\label{pa42}

Poussons maintenant plus loin les calculs dans le cas où 
$\r_\ell(\rt)$ n'est pas irréductible et super\-cuspidale, \ie que 
$w=w(\rt)>1$. 
Notons $\Q$ le cardinal de $\kk_{1}$ et, pour tout $n\>1$,
notons $f(n)=f_\Q(n)$ le nombre de $z'\in\P_\ell$ de degré $n$ 
sur $\kk_{1}$.
D'après le lemme \ref{L3}, on a donc~:
\begin{equation*}
\label{memento}
t(\rt)=\frac{f(w)}{w}.
\end{equation*}
Notons $v$ la valuation $\ell$-adique sur $\ZZ$, 
et notons $\kk'$ l'extension de $\kk_{1}$ de degré 
$w$ contenue dans $\kk$~; 
en partitionnant $\kk^{\prime\times}$ 
selon le degré de ses élément sur $\kk_{1}$, on obtient l'égalité~:
\begin{equation*}
\ell^{v(\Q^{w}-1)} = \sum\limits_{n |w} f(n).
\end{equation*}
Par inversion de Möbius, on a~:
\begin{equation*}
f(w) = \sum\limits_{n | w} 
\mu\left(\frac{w}{n}\right)
\ell^{v(\Q^{n}-1)}
\end{equation*}
où $\mu$ désigne la fonction de Möbius. 
Notons $w_0$ le plus grand diviseur de $w$ premier à $\ell$.

\begin{lemm}
L'ordre de $\Q$ dans $(\ZZ/\ell\ZZ)^\times$ est égal à $w_0$.
\end{lemm}

\begin{proof}
L'ordre de $z$ est de la forme $\ell^r$, $r\>0$.
Comme $w>1$, on déduit que $r\>1$.
La condition $\deg_{1}(z)=w$ signifie que l'ordre de $\Q$ 
dans $(\ZZ/\ell^{r}\ZZ)^\times$ est égal à $w$.
En projetant sur $(\ZZ/\ell\ZZ)^\times$, on en déduit que l'ordre de $\Q$ 
dans $(\ZZ/\ell\ZZ)^\times$ est égal au plus grand diviseur de $w$ 
premier à $\ell$, \ie $w_0$.
\end{proof}

On a donc~:
\begin{equation*}
f(w) = \sum\limits_{t\<v(w)} \sum\limits_{n | w_0} \ 
\mu(\ell^{v(w)-t})\mu\left(\frac{w_0}{n}\right)
\ell^{v(\Q^{n\ell^t}-1)}.
\end{equation*}
Si $v(w)=0$, on a $w=w_0>1$ et cela donne simplement~:
\begin{equation*}
f(w_0) = \sum\limits_{n | w_0} 
\mu\left(\frac{w_0}{n}\right)
\ell^{v(\Q^{n}-1)}.
\end{equation*}
On trouve que~:
\begin{equation*}
f(w_0) = \ell^{v(\Q^{w_0}-1)} + \sum\limits_{n | w_0 \atop{n\neq w_0}}
\mu\left(\frac{w_0}{n}\right)
=\ell^{v(\Q^{w_0}-1)}-1.
\end{equation*}
Supposons maintenant que $v(w)\>1$. 
Cela donne~:
\begin{equation*}
f(w) = \sum\limits_{n | w_0} \mu\left(\frac{w_0}{n}\right) \ell^{v(\Q^{n\ell^{v(w)}}-1)}
- \sum\limits_{n | w_0} 
\mu\left(\frac{w_0}{n}\right)\ell^{v(\Q^{n\ell^{v(w)-1}}-1)} \\ 
= f_{\Q^{\ell^{v(w)}}}(w_0)-f_{\Q^{\ell^{v(w)-1}}}(w_0).
\end{equation*}
Comme $\Q$ a le même ordre que $\Q^{\ell^k}$, $k\>0$,
dans $(\ZZ/\ell\ZZ)^\times$, à savoir $w_0$, 
on trouve que~:
\begin{equation*}
f(w) = \ell^{v(\Q^{u\ell^v}-1)}-\ell^{v(\Q^{u\ell^{v-1}}-1)}
= \ell^{v(\Q^{w}-1)-1}(\ell-1).
\end{equation*}
On trouve ainsi le résultat annoncé, en remarquant que $c(\rt)$ est égal à 
$\ell^{v(\Q^{w}-1)}$.

\subsection{}

Dans ce paragraphe, nous allons reformuler la proposition 
\ref{CongCusp} sous une 
forme analogue à celles de Vignéras \cite{Vigl} Proposition 2.3 
et Dat \cite{Datj} Proposition 2.3.2.

Fixons une uniformisante $\w$ de $\F$ et,
pour toute représentation irréductible cuspidale $\ell$-adique entiè\-re $\rt$ 
de $\G$, 
notons $\Oo(\rt,\w)$ l'en\-semble des classes de représentations 
irréductibles cuspidales $\ell$-adiques entiè\-res de $\G$ qui sont 
congrues à $\rt$ et dont le caractère central prend la même valeur que celui 
de $\rt$ en $\w$.
Soit $\mathfrak{c}(\rt)$ la plus grande puissance de $\ell$ divisant~:
\begin{equation*}
\frac{md}{n(\rt)} \cdot (q^{n(\rt)}-1).
\end{equation*}
On a le résultat suivant. 

\begin{prop}
\label{CongCusp2}
Soit $\rt$ une $\qlb$-représentation irréductible cuspidale et entière de $\G$. 
Alors l'en\-semble $\Oo(\rt,\w)$ est fini, de cardinal noté $\mathfrak{t}(\rt)$,
et on a~:
\begin{equation*}
\mathfrak{t}(\rt)\<\mathfrak{c}(\rt)
\end{equation*}
avec égalité si et seulement si $\rt$ est $\ell$-supercuspidale.  
\end{prop}

\begin{proof}
D'après la proposition \ref{CongCusp},
il suffit de prouver que $\mathfrak{t}(\rt)$ est le produit de 
$t(\rt)$ par la plus grande puissance de $\ell$ divisant
$md\cdot n(\rt)^{-1}$.
Remplaçons la cor\-respon\-dance bijective \eqref{sichuan} 
par la bijection~:
\begin{equation*}
\rho \leftrightarrow (\boldsymbol{\J},\Lambda)
\end{equation*}
entre classes d'isomorphisme de représentations irréductibles cuspidales 
de $\G$ et clas\-ses de conju\-gai\-son (sous $\G$) de types simples 
maximaux étendus de $\G$ (voir \cite{MSt} Théorème 3.11).

Soit $(\J,\lt)$ un type simple maximal contenu dans la 
$\qlb$-représentation irréductible cuspidale $\rt$, 
et soit $\widetilde{\boldsymbol{\J}}$ son normalisateur dans $\G$. 
D'après \cite{MSt} Proposition 3.1, il y a une unique représentation 
$\tL$ de $\widetilde{\boldsymbol{\J}}$ prolongeant $\lt$ dont l'induite à $\G$ est 
isomorphe à $\rt$.
Notons $\overline\Lambda$ la réduction modulo $\ell$ de $\tL$, 
qui est un prolongement de $\l$ à $\widetilde{\boldsymbol{\J}}$.

Soit $\rt'$ une représentation irréductible cuspidale $\ell$-adique entière de 
$\G$.  
Pour qu'elle soit con\-grue à $\rt$, il faut et il suffit qu'elle contienne un 
type simple maximal étendu $(\widetilde{\boldsymbol{\J}}{}',\tL')$ tel que 
$\widetilde{\boldsymbol{\J}}{}'$ soit égal à $\widetilde{\boldsymbol{\J}}$ 
et dont la réduction modulo 
$\ell$, notée $\overline\Lambda{}'$, soit égale à $\overline\Lambda$.
L'entier $\mathfrak{t}(\rt)$ est donc le nombre de classes de $\G$-conjugaison 
de $(\widetilde{\boldsymbol{\J}}{}',\tL')$ tels que
$\widetilde{\boldsymbol{\J}}{}'=\widetilde{\boldsymbol{\J}}$, 
$\overline\Lambda{}'=\overline\Lambda$ et 
$\tL'(\w)=\tL(\w)$.
Fixons une uniformisante $\w'$ de $\D'$ et posons~:
\begin{equation*}
\widetilde\w=(\w')^{d's(\rt)^{-1}}.
\end{equation*}
Le groupe 
$\widetilde{\boldsymbol{\J}}$ est engendré par $\J$ et $\widetilde\w$.
L'entier $\mathfrak{t}(\rt)$ est égal au produit $t(\rt)t_1(\rt)$ 
où $t_1(\rt)$ est le nombre de re\-pré\-sen\-tations $\tL'$ de 
$\widetilde{\boldsymbol{\J}}$ 
prolongeant $\lt$ telles que
$\overline\Lambda{}'(\widetilde\w)=\overline\Lambda(\widetilde\w)$ et 
$\tL'(\w)=\tL(\w)$.
Le nombre de re\-pré\-sen\-tations irréductibles de $\widetilde{\boldsymbol{\J}}$ prolongeant 
$\l$ et prenant une 
valeur fixée en $\w$ est égal à l'indice de $\mult\F\boldsymbol{\J}$ dans $\widetilde{\boldsymbol{\J}}$, 
\ie à~:
\begin{equation}
\label{T1}
e(\E:\F)s(\rt) = md\cdot n(\rt)^{-1},
\end{equation}
où $e(\E:\F)$ désigne l'indice de ramification de $\E$ sur $\F$.
Compte tenu de la condition supplémen\-tai\-re sur 
$\overline\Lambda{}'(\widetilde\w)$, on trouve 
que $t_1(\rt)$ est la plus grande puissance de $\ell$ divisant \eqref{T1}.
\end{proof}

\section{Preuve de la proposition \ref{lift}}

Soit $\rho$ une représentation irréductible cuspidale $\ell$-modulaire de 
$\G$ et soit $(\J,\k\otimes\s)$ un $\flb$-type simple maxi\-mal dans la classe de 
$\G$-conjugaison correspondant à $[\G,\rho]$.
D'après \cite{MSt}, pour que $\rho$ se relève à $\qlb$, il faut et 
suffit que $\s$, considérée comme une représentation irréductible cuspidale 
de $\GB$, se relève en une représentation 
irréductible cuspidale $\ell$-adique $\st$ telle que $s(\st)=s(\s)$.

Soit $y\in\Y$ correspondant à $\s$ par \eqref{james}.  
Pour qu'une telle représentation $\st$ existe, 
il faut et il suffit donc, d'après \eqref{DEGZ},
qu'il existe un $x\in\X$ dont la partie 
$\ell$-régulière soit $y$ et qui vérifie~:
\begin{equation*}
\deg(x) = \d(\s)\cdot\deg(y).
\end{equation*}

Si $\rho$ (donc $\s$) est supercuspidale, \ie si l'on a $\d(\s)=1$,
alors $x=y\in\X$ vérifie les conditions requises, et on retrouve bien le 
fait que toute représentation irréductible supercuspidale $\ell$-modulaire 
se relève. 

Supposons maintenant que $\rho$ est cuspidale mais pas 
supercuspidale, \ie que $\d(\s)>1$.
D'après le corollaire \ref{COROEPS1}, 
la proposition \ref{lift} peut être reformulée de la façon suivante. 

\begin{prop}
\label{lift2}
Soit $\rho$ une représentation irréductible cuspidale non 
supercuspidale $\ell$-mo\-du\-lai\-re de $\G$.
Pour que $\rho$ se relève à $\qlb$, il faut et il suffit que 
$s(\rho)$ et $\d(\rho)$ soient premiers entre eux
et que $\e(\rho)=1$.
\end{prop}

D'après la remarque \ref{FLO}, l'entier
$\e(\rho)$ est égal à l'entier $\e(\s)$ défini par \eqref{DEFERHO}. 

\begin{lemm}
\label{CN3}
Pour toute représentation irréductible cuspidale $\ell$-adique $\st$ relevant 
$\s$, le plus grand diviseur de $a(\st)$ premier à $\ell$ est $\e(\s)$.
\end{lemm}

\begin{proof}
Fixons un $x\in\X$ correspondant à $\st$ et de partie régulière $y$.
Comme $\rho$ (donc $\s$) n'est pas supercuspidale,
on a $x\neq y$.
Le résultat suit alors du lemme \ref{CN2}.
\end{proof}

Pour harmoniser les notations, posons $f(\s)=f(\rho)$.

\begin{lemm}
\label{LEM1}
Soit $z\in\P_\ell$ d'ordre $\ell^r$, $r\>0$.
On a $yz\in\X$ si et seulement si l'ordre de~:
\begin{equation}
\label{ORDREMYPROGRESSO}
q^{f(\s)\d(\s)^{-1}}
\end{equation}
dans $(\ZZ/\ell^r\ZZ)^\times$ est égal à $\d(\s)$.
\end{lemm}

Comme $y\in\Y$, il y a un $z\in\P_\ell$ (non trivial puisque $\s$ n'est pas 
supercuspidale) tel que $yz\in\X$.
Il y a donc un $r\>1$ tel que l'ordre de \eqref{ORDREMYPROGRESSO}
dans $(\ZZ/\ell^r\ZZ)^\times$ est égal à $\d(\s)$. 
En réduisant modulo $\ell$, on en déduit que son ordre dans 
$(\ZZ/\ell\ZZ)^\times$ est le plus grand diviseur de $\d(\s)$ 
premier à $\ell$.

\begin{lemm}
\label{LEM2}
Soit $z\in\P_\ell$ d'ordre $\ell^r$, $r\>0$.
On a $\deg(yz)=\d(\s)\cdot\deg(y)$ si et seulement si 
l'ordre de~:
\begin{equation}
\label{ORDREMYPROGRESSO2}
q^{f(\s)(\d(\s)s(\s))^{-1}}
\end{equation} 
dans $(\ZZ/\ell^r\ZZ)^\times$ est égal à $\d(\s)$.
\end{lemm}

Supposons d'abord que $\rho$ se relève à $\qlb$.
D'après le lemme \ref{CN} et le corollaire \ref{CoroIsole}, 
on trouve que $\d(\rho)$ est premier à $s(\rho)$ et que $\e(\rho)=1$.

Inversement, supposons que les conditions de la proposition \ref{lift} sont vérifiées. 
Soit $z\in\P_\ell$ d'ordre $\ell^r$, $r\>1$, tel que $yz\in\X$, et notons $n$ 
l'ordre de \eqref{ORDREMYPROGRESSO2} dans $(\ZZ/\ell^r\ZZ)^\times$.
D'après le lemme \ref{LEM1}, on a~:
\begin{equation}
\label{CRUX}
\frac{n}{(n,s(\s))} = \d(\s).
\end{equation}
L'hypothèse $\e(\rho)=1$ implique que~:
\begin{equation}
\label{CRUX2}
\frac{n}{(n,\d(\s))} = \ell^{t},
\quad t\>0.
\end{equation}
Si $\ell$ divise $\d(\s)$, alors $s(\s)$ est premier à $\ell$, et \eqref{CRUX} 
et \eqref{CRUX2} impliquent que $n=\d(\s)$.

En revan\-che, si $\d(\s)$ est premier à $\ell$, écrivons $n=\d(\s)n'$ avec 
$n'=(n,s(\s))=\ell^t$.
On peut remarquer que $t=v(n)$.
Alors l'élément~:
\begin{equation*}
z^{\ell^{v(n)}}\in\P_\ell
\end{equation*}
qui est d'ordre $\ell^{r-v(n)}$, 
vérifie la condition du lemme \ref{LEM2} car l'ordre de 
\eqref{ORDREMYPROGRESSO2} dans $(\ZZ/\ell^{r-v(n)}\ZZ)^\times$ est égal à 
$n\ell^{-v(n)}=\d(\s)$. 
Comme $\d(\rho)$ est premier à $s(\rho)$, il vérifie aussi la condition du 
lemme \ref{LEM1}. 
Ceci met fin à la preuve de la proposition \ref{lift}.

\begin{rema}
\label{Soitditenpassant}
Posons $k=k(\s)$ et $s=s(\s)$,
et notons $\tau$ l'unique représentation irré\-ducti\-ble
super\-cuspidale de $\GL_{m'k^{-1}}(\kd)$ telle que $\s$ soit 
isomorphe à la représentation non dégénérée no\-tée 
$\sp(\tau,k)$ au paragraphe \ref{P32}. 
Le plus grand diviseur de $k$ premier à $\ell$ est
$\e(\tau)(s,\e(\tau))^{-1}$ et $\e(\s)$ est égal à $(s,\e(\tau))$.
La condition de la proposition \ref{lift} s'écrit donc $\e(\rho)=1$ et 
min$(v(k),v(s))=0$.
\end{rema}

\section{Preuve de la proposition \ref{lifta}}
\label{SECTION5}

Soit $\rho$ une représentation irréductible cuspidale $\ell$-modulaire de 
$\G$, et soit $a>1$.
On cherche à quelle condition $\rho$ admet un $a$-relèvement, \ie une 
représentation irréductible cuspidale $\ell$-adique entière $\rt$ de $\G$ dont la 
réduc\-tion modulo $\ell$ contienne $\rho$ et soit de longueur $a$. 
Il faut et suffit pour cela que $\s$ se relève en une représentation 
irréductible cuspidale $\ell$-adique $\st$ telle que $s(\s)=a\cdot s(\st)$.
D'après les lemmes \ref{CN} et \ref{CN2}, on a des conditions nécessaires~: 
\begin{enumerate}
\item
$a=\e(\s)\ell^u$ avec $u\>0$.
\item
$a$ divise $s(\s)$ et $s(\s)a^{-1}$ est premier à $\d(\s)$.
\end{enumerate}
Supposons donc qu'elles sont vérifiées. 

\begin{rema}
En particulier, si $\s$ n'est pas supercuspidale,
les lemmes \ref{CN} et \ref{CN2} montrent que 
$\e(\s)$ divise $s(\s)$.
\end{rema}

Soit $y\in\Y$ correspondant à $\s$ par \eqref{james}.  
Pour qu'une telle représentation $\st$ existe, 
il faut et il suffit donc, d'après \eqref{DEGZ},
qu'il existe un $x\in\X$ dont la partie 
$\ell$-régulière soit $y$ et qui vérifie~:
\begin{equation*}
\deg(x) = a\d(\s)\cdot\deg(y).
\end{equation*}
Comme $y\in\Y$, il y a un $z\in\P_\ell$ (non trivial car $a>1$) 
tel que $yz\in\X$.
Il y a donc un entier $r\>1$ tel que l'ordre de \eqref{ORDREMYPROGRESSO}
dans $(\ZZ/\ell^r\ZZ)^\times$ est égal à $\d(\s)$. 
En réduisant modulo $\ell$, on en déduit que son ordre dans 
$(\ZZ/\ell\ZZ)^\times$ est le plus grand diviseur de $\d(\s)$ 
premier à $\ell$.

\begin{lemm}
\label{LEM2a}
Soit $z\in\P_\ell$ d'ordre $\ell^r$, $r\>0$.
On a $\deg(yz)=a\d(\s)\cdot\deg(y)$ si et seulement si 
l'ordre de~:
\begin{equation}
\label{ORDREMYPROGRESSO2a}
q^{f(\s)(\d(\s)s(\s))^{-1}}
\end{equation} 
dans $(\ZZ/\ell^r\ZZ)^\times$ est égal à $a\d(\s)$.
\end{lemm}

Soit $z\in\P_\ell$ d'ordre $\ell^r$, $r\>1$, tel que $yz\in\X$, et soit $n$ 
l'ordre de \eqref{ORDREMYPROGRESSO2a} dans $(\ZZ/\ell^r\ZZ)^\times$.
D'après le lemme \ref{LEM1}, on a~:
\begin{equation*}
\label{CRUXa}
\frac{n}{(n,s(\s))} = \d(\s).
\end{equation*}
En particulier, on a~:
\begin{equation*}
\label{CRUXabis}
v(n)=v(k)+{\rm min}(v(n),v(s(\s))).
\end{equation*}
Si l'on note $n_0$ le plus grand diviseur de $n$ premier à $\ell$, on a~:
\begin{equation*}
\label{CRUX2a} 
\frac{n_0}{(n_0,\d(\s))} = \e(\s).
\end{equation*}
On déduit de l'hypothèse sur $a$ que $(n_0,s(\s))=\e(\s)$, 
puis que $n_0=\d_0(\s)\e(\s)$, 
où $\d_0(\s)$ désigne le plus grand diviseur de 
$\d(\s)$ premier à $\ell$. 

On cherche un $t\in\{1,\dots,v(q^{f(\s)}-1)\}$ 
tel que l'ordre de \eqref{ORDREMYPROGRESSO2a} dans $(\ZZ/\ell^t\ZZ)^\times$ 
soit égal à $a\d(\s)$.
D'après l'hypo\-thèse sur $a$, cela impliquera automatiquement que 
l'ordre de \eqref{ORDREMYPROGRESSO} est $\d(\s)$. 
Soit~:
\begin{equation*}
t_0 = v(\Q^{n_0}-1)
\end{equation*}
où $\Q$ est défini par \eqref{ORDREMYPROGRESSO2a}
(voir aussi le paragraphe \ref{pa42}).
On a donc~:
\begin{equation*}
v(q^{f(\s)}-1) = t_0 + v(s(\s)) + v(k).
\end{equation*}
Posons $t=t_0+u+v(k)$ (on a bien $t\<v(q^{f(\s)}-1)$ car $u\<v(s(\s))$).
Alors l'ordre de \eqref{ORDREMYPROGRESSO2a} 
dans $(\ZZ/\ell^t\ZZ)^\times$ est égal à $n_0\ell^u=a\d(\s)$. 

\begin{rema}
Compte tenu de la remarque \ref{Soitditenpassant}, la condition de la
proposition \ref{lifta} se 
résume à $u\in\{0,\dots,v(s(\rho))\}$ et min$(v(k(\rho)),v(s(\rho))-u)=0$.
\end{rema}

\section{Preuve de la proposition \ref{comptagedeclasses}}
\label{SEC7}

Dans toute cette section, on fixe des entiers $n,w\>1$ tels que $w$ divise 
$n$.
Nous prouvons ici~la formule de comptage 
\eqref{comptagedeclassesFORMULEintro}, que l'on complètera 
dans la section suivante en le théorème \ref{ComptagePIIntro}. 

\subsection{}
\label{ENDO}

Dans ce paragraphe, nous rappelons brièvement quelques 
attributs des $\F$-endoclasses de carac\-tè\-res simples
dont nous aurons besoin.
Pour plus de détails, nous renvoyons le lecteur à \cite{BSS}. 

Soit $\A$ une $\F$-algèbre centrale simple de dimension finie,
soit $[\La,n_\La,0,\b]$ une strate simple de $\A$ et soit 
$\t\in\Cc(\La,0,\b)$ un caractère simple.
Le couple $([\La,n_\La,0,\b],\t)$ 
défi\-nit un ps-caractère dont l'endo-classe 
-- qui est une classe d'équivalence de ps-ca\-rac\-tères --
est notée $\TT$.
Les entiers~:
\begin{eqnarray*}
f(\TT) &=& f(\F[\b]:\F), \\
\deg(\TT) &=& [\F[\b]:\F],
\end{eqnarray*} 
\ie le degré résiduel et le degré 
de $\F[\b]$ sur $\F$ respectivement, 
ne dépendent pas du choix de $\b$ mais uniquement de $\TT$.
Le nombre rationnel positif~:
\begin{equation*}
l(\TT) = -v_{\F}(\b),
\end{equation*}
où $v_{\F}$ désigne la valuation sur $\F[\b]$ normalisée 
en donnant la valeur $1$ à une uniformisante de $\F$, ne dépend pas non plus 
du choix de $\b$ mais uniquement de $\TT$.

Si $\rho$ est une $\R$-représentation irréductible cuspidale de $\mult\A$ 
(ici $\R$ désigne $\qlb$ ou $\flb$), il existe un couple 
$([\La,n_\La,0,\b],\t)$ comme ci-dessus tel que la 
restriction de $\rho$ au pro-$p$-sous-groupe $\H^1(\b,\La)$
contienne $\t$.
L'endoclasse $\TT$ définie par ce couple ne dépend que de la classe d'isomorphisme 
de la représentation $\rho$,
et le nombre rationnel $l(\TT)\>0$ est le niveau normalisé 
(ou aussi profon\-deur) de $\rho$.

\subsection{}

Soit $\D$ une $\F$-algèbre centrale simple de degré réduit $d$ divisant $n$,
et soit $\TT$ une $\F$-endoclasse de degré $\gg$ divisant $nw^{-1}$.
On définit un entier $m\>1$ par la relation $md=n$ et on pose~:
\begin{equation*}
d' = \frac{d}{(d,\gg)},
\quad
m' = \frac{m(d,\gg)}{\gg}.
\end{equation*}
Pour tout $u\>1$ divisant $m$, 
notons $\Aa(\D,\TT,w,u)$ l'en\-sem\-ble des classes inertielles de 
re\-pré\-sen\-ta\-tions 
irréductibles cuspidales $\ell$-adiques $\rt$ de $\GL_{u}(\D)$ telles que~: 
\begin{enumerate}
\item
$w(\rt)=w$~;
\item
l'endoclasse de $\rt$ est égale à $\TT$.
\end{enumerate}
Remarquons que, pour qu'il y ait une représentation 
irréductible cuspidale $\ell$-adique $\rt$ de $\GL_{u}(\D)$ 
d'endoclasse $\TT$, il faut et il suffit que l'entier $u$ soit 
de la forme~:
\begin{equation} 
\label{quark}
u = \frac{\gg}{(d,\gg)}\cdot u'
\end{equation}
où $u'$ est un diviseur de $m'$.
Posons maintenant~:
\begin{equation}
\label{Pietranera}
\Aa(\D,\TT,w) = 
\bigcup\limits_{u|m} \Aa(\D,\TT,w,u)
\end{equation}
et notons $\Aa_{\ell}(\D,\TT,w)$ l'ensemble des réductions modulo $\ell$ 
des éléments de \eqref{Pietranera}. 
L'endoclasse~$\TT$ étant fixée, cet ensemble est 
fini, et son cardinal sera noté $\boldsymbol{a}_{\ell}(\D,\TT,w)$.

\subsection{}

Dans ce paragraphe, on suppose que $\TT$ est la $\F$-endoclasse nulle $\0$, 
et on va calculer $\boldsymbol{a}_{\ell}(\D,\0,w)$.
Étant donné un entier $u\>1$ divisant $m$,
toute représentation irréductible cuspidale $\ell$-mo\-dulaire $\s$ de $\GL_u(\kd)$ 
définit un type simple maximal de niveau $0$ de $\GL_u(\D)$ |
donc une classe inertielle $\Omega(\s)$ de représentations 
cuspidales de niveau $0$ de $\GL_u(\D)$.
L'application~:
\begin{equation}
\label{Gina}
\s \mapsto \Omega(\s)
\end{equation}
est surjective (toutes les classes inertielles de représentations 
cuspidales de niveau $0$ de $\GL_u(\kd)$ sont atteintes) et ses 
fibres sont les classes de conjugaison sous le groupe de Galois de 
$\kd$ sur $\ke$.

\begin{lemm}
\label{riddle}
Soit $\s$ une représentation irréductible cuspidale $\ell$-mo\-dulaire de 
$\GL_u(\kd)$. 
Alors $\Omega(\s)$ appartient à $\Aa_\ell(\D,\0,w)$ si et seulement 
s'il existe une représentation irréductible cuspidale $\ell$-adique 
$\st$ de $\GL_u(\kd)$ dont la réduction modulo $\ell$ soit $\s$ et 
telle que $w(\st)=w$.
\end{lemm}

Notant $\BB_{\ell}(q,m,d,w)$ l'image réciproque de $\Aa_\ell(\D,\0,w)$ par 
\eqref{Gina}, qui est décrite par le lemme \ref{riddle}, on obtient ainsi~:
\begin{equation*}
\boldsymbol{a}_{\ell}(\D,\0,w) = \sum_{\s} \frac{s(\s)}{d}
\end{equation*}
où $\s$ décrit l'ensemble $\BB_{\ell}(q,m,d,w)$.

Fixons une clôture algébrique $\overline{\kk}$ de $\kd$ et notons 
$\kl$ l'extension de $\ke$ de degré $nw^{-1}$ incluse dans 
$\overline{\kk}$. 
Pour tout $y\in\kl^{\times}$, posons~:
\begin{eqnarray*}
\deg(y) &=& \text{degré de $y$ sur $\ke$}, \\
{\upepsilon}(y) &=& \text{ordre de $q^{\deg(y)}$ dans $(\ZZ/\ell\ZZ)^\times$}.
\end{eqnarray*}
et notons $\YY_{\ell}(q,n,w)$ l'ensemble des $y\in\kl^{\times}$, 
d'ordre premier à $\ell$, 
tels que ${\upepsilon}(y)$ soit égal à $w_0$, 
le plus grand diviseur de $w$ premier à $\ell$.
Nous allons définir une application surjective~:
\begin{equation*}
\YY_{\ell}(q,n,w) \to \BB_{\ell}(q,m,d,w)
\end{equation*}
dépendant du choix de $\D$. 

Soit $y\in\YY_{\ell}(q,n,w)$. 
Notons~:
\begin{equation*}
r(y) = \frac{\deg(y)}{(\deg(y),d)}
\end{equation*} 
le degré de $y$ sur $\kd$ et $\boldsymbol{\uptau}_\D(y)$ la re\-pré\-senta\-tion 
irréductible supercus\-pidale $\ell$-modulaire du groupe $\GL_{r(y)}(\kd)$ 
correspon\-dant à $y$ par \eqref{james}. 
Posons $s(y)=s(\boldsymbol{\uptau}_\D(y))$ et~:
\begin{equation*}
\d(y)=\frac{w}{(w,s(y))}.
\end{equation*}
On a le lemme suivant. 

\begin{lemm}
Il existe une unique représentation irréductible cuspidale~:
\begin{equation*}
\boldsymbol{\upsigma}_\D(y) 
\end{equation*} 
dont le support supercuspidal soit égal à $k(y)\cdot\boldsymbol{\uptau}_\D(y)$.
\end{lemm}

\begin{proof}
D'après le paragraphe \ref{P32}, il suffit de prouver que 
le plus grand diviseur de $k$ pre\-mier à $\ell$, noté $k_0$, 
est égal à l'ordre de $q^{dr(y)}$ dans $(\ZZ/\ell\ZZ)^\times$.
Comme on a ${\upepsilon}(y)=w_0$ d'une part et $r(y)d = s(y)\deg(y)$
d'autre part, cet ordre est égal à~:
\begin{equation*}
\frac{w_0}{(w_0,s(y))} = k_0,
\end{equation*}
ce qui met fin à la démonstration. 
\end{proof}

\begin{lemm}
Pour tout $y\in\YY_{\ell}(q,n,w)$,
la représentation $\boldsymbol{\upsigma}_\D(y)$ appartient à $\BB_{\ell}(q,m,d,w)$.
\end{lemm}

\begin{proof}
Il faut d'abord prouver que le degré de $\boldsymbol{\upsigma}_\D(y)$ divise 
$m$. 
D'abord, on a~:
\begin{equation*}
s(y)=\frac{d}{(\deg(y),d)}.
\end{equation*}
Par hypothèse sur $y$, il existe un entier $t\>1$ tel que $n=wt\cdot\deg(y)$.
On en déduit que~:
\begin{equation*}
r(y)=\frac{m}{(m,wt)}
\quad\text{et}\quad
s(y)=\frac{wt}{(m,wt)}
\quad\text{et}\quad
k(y)=\frac{(m,wt)}{((m,wt),t)},
\end{equation*}
puis que l'entier $\deg(\boldsymbol{\upsigma}_\D(y))=k(y)r(y)$ divise $m$. 
D 'après le début de la section \ref{SECTION5}, il ne reste 
qu'à vé\-ri\-fier que 
$a=(w,s(y))$ divise $s(y)$ et que $s(y)a^{-1}=s(y)(w,s(y))^{-1}$ est premier à $k(y)$, 
ce qui est im\-mé\-diat. 
\end{proof}

Ceci définit une application
$\boldsymbol{\upsigma}_\D : y \mapsto{} \boldsymbol{\upsigma}_\D(y)$
de $\YY_{\ell}(q,n,w)$ dans $\BB_\ell(q,m,d,w)$.

\begin{prop}
L'application $\boldsymbol{\upsigma}_\D$ est surjective, 
et ses fibres sont les classes de conjugaison sous l'automorphisme de 
Frobenius $x\mapsto x^{q^d}$.
\end{prop}

\begin{proof}
Soit $\s\in\BB_\ell(q,m,d,w)$. 
Il y a une unique représentation irréductible supercuspidale 
$\ell$-modulaire $\boldsymbol{\uptau}(\s)$ telle que le support supercuspidal de 
$\s$ soit égal à $\d(\s)\cdot \boldsymbol{\uptau}(\s)$. 
Soit $y\in\overline{\kk}{}^\times$ un paramètre de James pour 
$\boldsymbol{\uptau}(\s)$. 
Il est d'ordre premier à $\ell$ et vérifie ${\upepsilon}(y)=w_0$, 
mais il est \textit{a priori} dans une extension de $\ke$ de degré 
$n\d(\s)^{-1}$. 
Toutefois, par hypo\-thèse sur $\s$, l'entier $w\d(\s)^{-1}$ divise $s(\tau)$. 
On en déduit que $y$ est dans une extension de $\ke$ de degré $nw^{-1}$, 
\ie que $y\in\YY_\ell(q,n,w)$, ce qui prouve la surjectivité. 
Ensuite, on a~:
\begin{equation*}
\label{FabricedelDongo}
\deg(\s) = \frac{\deg(y)}{(\deg(y),d)}\cdot\frac{w}{(w,s(y))}
\quad\text{et}\quad
\d(\s) = (w,\deg(\s)) = \frac{w}{(w,s(y))}.
\end{equation*}
Il s'ensuit que l'application $\s\mapsto\boldsymbol{\uptau}(\s)$ est injective. 
\end{proof}

On en déduit que~:
\begin{eqnarray*}
\boldsymbol{a}_{\ell}(\D,\0,w) 
&=& \sum_{y} \frac{s(\boldsymbol{\upsigma}_\D(y))}{d}\cdot\frac{1}{\deg(\boldsymbol{\upsigma}_\D(y))} \\
&=& \sum_{y} \frac{1}{\deg(y)}
\end{eqnarray*}
(où $y$ décrit $\YY_\ell(q,n,w)$),
valeur que l'on note $\boldsymbol{y}_{\ell}^{1}(q,n,w)$.

\subsection{}

Supposons maintenant que $\TT$ est quelconque, de degré
$g$ divisant $nw^{-1}$.  
On pose~:
\begin{equation*}
q(\TT)=q^{f(\TT)},
\quad
n(\TT) = m'd' = \frac{n}{\gg},
\end{equation*}
où $f(\TT)$ désigne le degré résiduel de l'endoclasse $\TT$.
Fixons un entier $u$ de la forme \eqref{quark} et une réalisation 
$([\La,n_\La,0,\b],\t)$, où 
$\t$ est un caractère simple ($\ell$-modulaire) attaché à la strate 
simple $[\La,n_\La,0,\b]$ de $\Mat_u(\D)$.
On suppose que 
l'intersection entre l'ordre hérédi\-taire associé à $\La$ 
et le cen\-tralisateur de $\F[\b]$ dans $\Mat_u(\D)$ est maximal.
Fixons aussi une $\b$-extension $\k$ de ${\t}$ et posons 
$\J=\J(\b,\La)$. 
L'application $\s\mapsto \k\otimes\s$ induit une surjection~:
\begin{equation}
\label{elevteria}
\s \mapsto [\J,\k\otimes\s] \leftrightarrow [\GL_u(\D),\rho]
\end{equation}
entre classes d'isomorphisme de représentations 
irréductibles cuspidales $\ell$-modulaires du groupe $\GL_{u'}(\kd)$ et 
classes inertielles de représentations irréductibles cuspidales 
$\ell$-modulaires du groupe $\GL_{u}(\D)$ d'endoclasse $\TT$.
L'image d'une représentation $\s$ par \eqref{elevteria} appartient à  
$\Aa_{\ell}(\D,\TT,w)$ si et seulement si $\s$ est dans $\BB_{\ell}(q(\TT),m',d',w)$, 
et l'ensemble des antécédents d'une classe inertielle $[\GL_u(\D),\rho]$ par 
\eqref{elevteria} est de cardinal $s(\rho)$.
On obtient ainsi~:
\begin{equation*}
\boldsymbol{a}_{\ell}(\D,\TT,w) 
= \sum_{\s} \frac{s(\s)}{d'}
= \boldsymbol{y}_{\ell}^{1}(q(\TT),n(\TT),w)
\end{equation*}
où $\s$ décrit l'ensemble $\BB_{\ell}(q(\TT),m',d',w)$. 

\subsection{}

Finalement, si l'on fixe un nombre rationnel $j\>0$,
et si l'on pose~:
\begin{equation*}
\Aa_\ell(\D,w,j) = \bigcup\limits_{l(\TT)\<j} \Aa_\ell(\D,\TT,w),
\end{equation*}
alors on obtient l'égalité~:
\begin{equation*}
\boldsymbol{a}_{\ell}(\D,w,j) 
= \sum\limits_{l(\TT)\<j} \boldsymbol{a}_{\ell}(\D,\TT,w) 
= \sum\limits_{l(\TT)\<j} \boldsymbol{y}_{\ell}^{1}(q(\TT),n(\TT),w)
= \boldsymbol{a}_{\ell}(\F,w,j),
\end{equation*}
ce qui met fin à la 
preuve de la proposition \ref{comptagedeclasses} et du corollaire \ref{DonaTartt}. 

\section{Compatibilité de la classification de Zelevinski aux 
congruences}
\label{P12}

Dans cette section, nous prouvons les propositions \ref{ConjRedModlSpehIntro}
et \ref{Zcongintro} et les théorèmes \ref{CongSpehIntro} et
\ref{ComptagePIIntro}. 

Rappelons que
$\R$ désigne un corps algébriquement clos de caractéristique différen\-te de $p$. 
Pour $\pi\in\RA(\R)$ et 
$\s\in\XA(\R)$, on notera $[\pi:\s]$ la multi\-plicité de $\s$ dans $\pi$.

\subsection{}
\label{P81}

Un {\it segment} est un couple $[\cuspi,n]$ formé d'une classe d'isomorphisme 
de représentation irré\-duc\-tible cuspidale $\cuspi$ et d'un entier $n\>1$. 
Dans \cite{MSc} §7.2, nous associons à un tel segment $[\cuspi,n]$ une 
sous-représentation irréductible de l'induite~:
\begin{equation}
\label{INDRHON8}
\cuspi\times\cuspi\nu_{\cuspi}^{}\times\dots\times\cuspi\nu_{\cuspi}^{n-1}
\end{equation}
notée $\Z(\cuspi,n)$.
Quand $\R$ est le corps des nombres complexes, 
elle est uniquement déterminée par cette propriété~; 
sa duale de Zelevinski est une représentation essentiellement 
de carré intégrable qui est l'unique quotient irréductible de \eqref{INDRHON8}.
Pour un corps $\R$ général, 
nous aurons seulement besoin de la propriété suivante de $\Z(\cuspi,n)$.
Notons $m$ le degré de $\cuspi$.

\medskip

\begin{tabular}{lp{14cm}}
{(\ref{P81}.1)} & 
\textit{Pour tout $k\in\{1,\dots,n-1\}$, on a
$\rp_{(mk,m(n-k))}(\Z(\cuspi,n))=\Z(\cuspi,k)\otimes
\Z(\cuspi\nu_{\cuspi}^k,n-k)$.} \\
\end{tabular}

\medskip

\noindent
Ajoutons-y deux propriétés de 
l'application $[\cuspi,n]\mapsto\Z(\cuspi,n)$.
On renvoie au paragraphe \ref{GeEm1} pour la notation $\rho\cdot\mu$.

\medskip

\begin{tabular}{lp{14cm}}
{(\ref{P81}.2)} & 
\textit{Pour tout caractère $\mu$ de $\mult\F$, 
on a $\Z(\cuspi\cdot\mu,n)=\Z(\cuspi,n)\cdot\mu$.} \\
\end{tabular}

\medskip

\begin{tabular}{lp{14cm}}
{(\ref{P81}.3)} & 
\textit{Si $\cuspi'$ est une représentation irréductible cuspidale 
et $n'\>1$ un entier, alors $\Z(\cuspi',n')$ est isomorphe à 
$\Z(\cuspi,n)$ si et seulement si les segments $[\cuspi,n]$ et 
$[\cuspi',n']$ sont égaux.} \\
\end{tabular}

\medskip

\noindent
On introduit les définitions suivantes. 

\begin{defi}
\label{BenOuiMonGars}
Soit $\pi$ une représentation de la forme $\Z(\cuspi,n)$, 
avec $\cuspi$ irréductible cuspidale de degré $m$ et $n\>1$.
\begin{enumerate}
\item
On pose $n(\pi)=n(\cuspi)$.
\item
On appelle \textit{classe de torsion} de $\pi$ l'ensemble 
$\langle\pi\rangle$ des classes de $\pi\chi$
où $\chi$ décrit les caractères non ramifiés de $\G_m$.
\end{enumerate}
\end{defi}

\begin{rema}
D'après les propriétés (\ref{P81}.2) et (\ref{P81}.3), 
l'entier $n(\pi)$ est le nombre 
de carac\-tè\-res non ramifiés $\chi$ de 
$\G_m$ tels que $\pi\chi=\pi$, 
et $\langle\pi\rangle$ est l'ensemble des représentations
$\Z(\cuspi\chi,n)$ où $\chi$ décrit les caractères non ramifiés de $\G_m$, 
\ie l'ensemble des représentations
$\Z(\cuspi',n)$ où $\cuspi'$ décrit les représentations de la classe d'inertie
de $\cuspi$.

Dans le cas où $\pi$ est cuspidale, \ie lorsque $n=1$, 
l'ensemble $\langle\pi\rangle$ est la classe d'inertie de $\pi$.
\end{rema}

\begin{defi}
\begin{enumerate}
\item
Etant donné un entier $m\>1$, nous noterons~:
\begin{equation*}
\label{Biclou}
\Zz(\G_m,\R)
\end{equation*}
l'ensemble des classes de représentations irréductibles de $\G_m$ de la forme 
$\Z(\cuspi,r)$ où $r\>1$ décrit 
les diviseurs de $m$ et $\cuspi$ les re\-pré\-sentations cuspidales de 
$\G_{mr^{-1}}$.
Les représentations irréductibles de cette for\-me 
sont appelées les \textit{représentations de Speh} de $\G_m$.
\item
Si $\R$ est de caractéristique $>0$, une représentations de Speh 
$\Z(\rho,r)$ de $\G_m$ comme ci-dessus est dite \textit{super-Speh} si $\cuspi$ est 
supercuspidale~; nous noterons~:
\begin{equation*}
\label{BiclouSS1}
\Zz_1(\G_m,\R)
\end{equation*}
l'ensemble des classes de représentations super-Speh 
de $\G_m$.
\end{enumerate}
\end{defi}

\subsection{}
\label{P51}
\label{LM2}

Fixons 
une $\R$-représentation irréductible cuspidale $\cuspi$ de degré $m\>1$.
Soit $\Seg_\cuspi$ l'ensemble des segments de la forme $[\cuspi\chi,n]$ où 
$n\>1$ est un entier et $\chi$ un ca\-rac\-tère non ra\-mifié de $\G_{m}$,
et soit $\XA_\cuspi$ l'en\-sem\-ble des classes de représentations irréductibles 
qui sont des sous-quotients 
d'in\-dui\-tes de la forme $\cuspi\chi_1\times\dots\times\cuspi\chi_n$ où $n\>0$ 
est un entier et $\chi_1,\dots,\chi_n$ des caractères non ramifiés 
de $\G_{m}$. 
Pour tout ensemble $\X$, notons $\NN(\X)$ l'ensemble des appli\-ca\-tions 
de $\X$ dans $\NN$ à support fini.
Nous construisons dans \cite{MSc} une application surjective~:
\begin{equation*}
\label{Melusine}
\NN(\Seg_{\cuspi})\ \ffr{\Z}\ \XA_{\cuspi}
\end{equation*}
qui est bijec\-tive lorsque $\cuspi$ est supercuspidale, 
et coïncide avec la classifica\-tion de Zele\-vin\-ski \cite{Ze2}
lorsque $\R$ est le corps des nombres complexes.

Il sera commode d'employer un formalisme légèrement différent 
de celui de \cite{MSc}. 
Un \textit{segment for\-mel} est une pai\-re $[a,n]$ formée de deux entiers 
$a\in\ZZ$ et $n\>1$.
L'entier $n\>1$ est la \textit{lon\-gueur} du segment formel.
Si $[a,n]$ est un segment formel, on pose~:
\begin{equation*}
[a,n]\boxtimes\cuspi=[\cuspi\nu_{\cuspi}^{a},n]
\end{equation*}
qui est un segment au sens donné plus haut.
On note $\Seg$ l'ensemble des segments formels,
et on ap\-pelle \textit{multisegment formel} un élément de $\NN(\Seg)$.
On a par linéarité une application $\upmu\mapsto\upmu\boxtimes\cuspi$ 
de $\NN(\Seg)$ dans $\NN(\Seg_{\cuspi})$.
Par linéarité également, on définit la longueur d'un multisegment formel.

Soient $\upmu=[a_1,n_1]+\dots+[a_r,n_r]$ et
$\upnu=[c_1,m_1]+\dots+[c_s,m_s]$ deux multisegments for\-mels de 
même longueur, dont les segments formels 
sont supposés être rangés par longueur décrois\-sante. 
On écrit $\upmu\trianglelefteq\upnu$ si~: 
\begin{equation*}
\sum\limits_{i\<k}n_i\<\sum\limits_{i\<k}m_i
\end{equation*}
pour tout $k\in\{1,\dots,{\rm min}(r,s)\}$.
Ceci définit une relation d'ordre $\trianglelefteq$ 
entre multisegments for\-mels de même longueur. 

Pour $\upmu,\upnu\in\NN(\Seg)$ comme ci-dessus, notons~:
\begin{equation*}
\textsf{m}(\upmu,\upnu,\cuspi)
\end{equation*}
la multiplicité~de 
$\Z(\upnu\boxtimes\cuspi)$ dans la re\-pré\-sentation 
$\Z([a_1,n_1]\boxtimes\cuspi)\times\dots\times\Z([a_r,n_r]\boxtimes\cuspi)$.
Lorsque $\upmu,\upnu$ varient, ces multiplicités vérifient la condition 
suivante.

\medskip

\begin{tabular}{lp{14cm}}
{(\ref{P51}.1)} & 
\textit{On a ${\sf m}(\upmu,\upmu,\cuspi)=1$ et, 
si ${\sf m}(\upmu,\upnu,\cuspi)\neq0$ alors $\upmu\trianglelefteq\upnu$.} \\
\end{tabular}

\medskip

\noindent
Ceci caractérise les représentations 
$\Z(\upmu\boxtimes\cuspi)$, $\upmu\in\NN(\Seg)$, 
à isomorphisme près. 

\begin{rema}
Au sujet de la dépendance de ces multiplicité en $\cuspi$, 
voir la proposition 4.15 et la remarque 4.18 de \cite{MSi}.
\end{rema}

Nous renvoyons à \cite{MSc} pour la notion de représentation 
irréductible rési\-duelle\-ment non dégéné\-rée.
Soit $\upmu$ un multisegment formel comme ci-dessus, 
de longueur $n$,
et dont les segments formels sont supposés être rangés par longueur 
décrois\-sante. 
On note $\upmu^\vee$ la partition de $n$ conjuguée à $(n_1,\dots,n_r)$. 
Par \cite{MSc} Proposition 9.19, le module de Jacquet~:
\begin{equation*}
\rp_{m\cdot\upmu^\vee}(\Z(\upmu\boxtimes\cuspi))
\end{equation*}
contient un sous-quotient irréductible résiduellement non dégénéré, 
unique et apparaissant avec multiplicité $1$~: on le note 
$\Sp(\cuspi,\upmu^{\vee})$.
(Si l'on écrit $\upmu^{\vee}$ sous la forme $(a_1,\dots,a_s)$ avec $s\>1$,
alors $m\cdot\upmu^{\vee}$ désigne la famille $(ma_1,\dots,ma_s)$.)

\begin{rema}
Dans le cas particulier où $\upmu^\vee=(n)$, on retrouve la re\-pré\-sen\-tation 
résiduel\-le\-ment non dégénérée $\Sp(\cuspi,n)$ du paragraphe \ref{P32}.
\end{rema}

Si $\upmu,\upnu$ sont deux multisegments for\-mels de même longueur,
on note~:
\begin{equation*}
{\textsf{e}}(\upmu,\upnu) = {\textsf{e}}(\upmu,\upnu,\rho) 
\end{equation*}
la multiplicité de $\Sp(\cuspi,\upmu^{\vee})$ dans le module de Jacquet 
$\rp_{m\cdot\upmu^\vee}(\Z(\upnu\boxtimes\cuspi))$.
D'après \cite{MSc} Proposition 9.19, on a la propriété suivante. 

\medskip

\begin{tabular}{lp{14cm}}
{(\ref{P51}.2)} & 
\textit{On a ${\sf e}(\upmu,\upmu)=1$ et, si ${\sf e}(\upmu,\upnu)\neq0$ alors 
$\upnu\trianglelefteq\upmu$.} \\
\end{tabular}

\medskip

De ceci on déduit le lemme suivant, qui nous sera utile dans la section \ref{S11}.

\begin{lemm}
\label{lemm:cle} 
Soient $\upmu_1,\dots,\upmu_r$ des multisegments formels de 
longueur $n\>1$, et soient des en\-tiers relatifs $a_1,\dots,a_r\in \ZZ$.
On suppose que les représentations irréductibles~: 
\begin{equation*}
\pi_i=\Z(\upmu_i\boxtimes\rho),
\quad
i\in\{1,\dots,r\}
\end{equation*}
sont deux à deux distinctes et rési\-duellement dégénérées. 
On pose $\Pi=a_1\pi_1+\dots+a_r\pi_r$, et on suppose que, 
pour tout $k\in\{1,\dots,n-1\}$, on a $r_{(mk,m(n-k))}(\Pi)=0$. 
Alors $\Pi=0$.
\end{lemm}

\begin{proof}
Par hypothèse, et d'après \cite{MSc} Corollaire 8.5, 
les partitions $\upmu_1^\vee,\dots,\upmu_r^\vee$ 
sont toutes distinctes de $(n)$. 

Supposons que $\Pi$ soit non nul. 
On choisit un $i$, parmi ceux pour lesquels $a_i \neq 0$, 
tel que $\upmu_i^\vee$ soit minimale. 
Alors $\Sp(\rho,\upmu_i^\vee)$ apparaît avec multiplicité $1$ dans le module 
de Jacquet $\rp_{m\cdot\upmu_i^\vee}(\pi_i)$ et n'apparaît pas dans le module de 
Jacquet de $\rp_{m\cdot\upmu_i^\vee}(\pi_j)$ pour $j \neq i$. 
Par transitivité des foncteurs de Jacquet, 
ceci contredit notre hypothèse $r_{(mk,m(n-k))}(\Pi)=0$ pour tout $1\<k\<n-1$. 
\end{proof}

\subsection{}

On suppose maintenant que $\R$ est le corps $\qlb$,
et on fixe un multisegment formel $\upmu$ de longueur $n\>1$.
On a une application~:
\begin{equation}
\label{Zrt}
\cuspit\mapsto\Z(\upmu\boxtimes\cuspit)
\end{equation}
qui à toute $\qlb$-représentation irréductible cuspidale $\cuspit$ de degré $m\>1$ 
associe une repré\-sen\-tation irréductible de 
degré $mn$.
Si $\cuspit$ est entière, alors son image par \eqref{Zrt} l'est aussi.
Les paragraphes \ref{Racaille} à \ref{CasGenZcong} seront consacrés à la preuve 
du résultat suivant (voir la proposition \ref{Zcongintro}). 

\begin{prop}
\label{Zcong}
Si $\cuspit_1$, $\cuspit_2$ sont des $\qlb$-représentations 
irréductibles cuspidales entières con\-gruentes de $\G_m$, $m\>1$, 
alors les représentations 
$\Z(\upmu\boxtimes\cuspit_1)$ et $\Z(\upmu\boxtimes\cuspit_2)$ 
sont congruentes. 
\end{prop}

\subsection{}
\label{Racaille}

Nous commençons par des préliminaires techniques. 
Supposons à nouveau que $\R$ est un corps algébriquement clos de 
caractéristique différen\-te de $p$. 
On pose $\G=\G_m$ avec $m\>1$. 

Soit $[\La,n_\La,0,\b]$ une strate simple de $\Mat_m(\D)$.
Comme dans le paragraphe \ref{PARA21}, 
on suppose que l'inter\-section de l'ordre héré\-di\-taire associé à $\La$ 
avec le centralisateur de $\E=\F[\b]$ dans $\Mat_m(\D)$ est un ordre maximal.
On fixe un $\R$-caractère simple $\t$ attaché à cette strate et une 
$\b$-extension $\k$ de $\t$.
On reprend toutes les notations du paragraphe \ref{PARA21}. 

On note $\KM$ le foncteur $\pi\mapsto\Hom_{\J^1}(\k,\pi)$ 
de la catégorie des $\R$-représentations lisses de $\G$ dans 
celle des $\R$-représentations de $\J/\J^1$ (identifié à $\GB$) 
en faisant agir $\J$ sur $\KM(\pi)$ par la formule~:
\begin{equation*}
x\cdot f = \pi(x) \circ f \circ \k(x)^{-1}
\end{equation*}
pour $x\in\J$ et $f\in\KM(\pi)$.
Ce fonc\-teur a été étudié dans \cite{MSt,SSb}.
Il est exact et envoie repré\-sen\-tations de longueur 
finie sur représentations de longueur finie.
Il induit donc un morphisme de groupes~: 
\begin{equation} 
\label{DEFKMR}
\RA(\G,\R) \to \RA(\GB,\R)
\end{equation}
que l'on note encore $\KM$, et qui a les propriétés suivantes.

\medskip

\begin{tabular}{p{1cm}p{14cm}}
(\ref{Racaille}.1) & 
\textit{Pour tout $\pi\in\XA(\G,\R)$ et tout caractère non ramifié 
$\chi$ de $\G$, on a $\KM(\pi\chi)=\KM(\pi)$. }
\end{tabular}

\medskip

\noindent
Notons ${\bf\Theta}$ l'endo-classe définie par la paire 
$([\La,n_\La,0,\b],\t)$.
On renvoie à \cite{MSt} \S5.2 pour les deux propriétés suivantes.

\medskip

\begin{tabular}{p{1cm}p{14cm}}
(\ref{Racaille}.2) & 
\textit{Etant donné $\pi\in\XA(\G,\R)$, pour que $\KM(\pi)$ 
soit non nul, il faut et suffit que 
le support cuspidal de $\pi$ soit formé de représentations cuspidales 
d'endo-classe ${\bf\Theta}$. }
\end{tabular}

\medskip

\begin{tabular}{p{1cm}p{14cm}}
(\ref{Racaille}.3) & 
\textit{Si $\rho\in\XA(\G,\R)$ est cuspidale et d'endo-classe ${\bf\Theta}$, 
alors $\KM(\rho)$ est cuspidale et chacun de ses facteurs 
irréductibles apparaît avec multiplicité $1$. 
}
\end{tabular}

\medskip

\noindent
Plus précisément, selon \cite{MSt} Lemme 5.3, 
si $\rho$ contient le type simple maximal $\k\otimes\s$, 
alors $\KM(\rho)$ est la somme des $\Gal(\kd/\ke)$-conjugués de $\s$, 
\ie~:
\begin{equation*}
\KM(\rho) = \s_1 + \dots + \s_b
\end{equation*}
où $\{\s_1,\dots,\s_b\}$ est l'orbite de $\s$ sous $\Gal(\kd/\ke)$.
Le nombre 
$b=b(\rho)$ de $\Gal(\kd/\ke)$-conjugués de $\s$ vérifie l'identité 
$b(\rho)s(\rho)=d'$, où $s(\rho)$ est l'entier défini par \eqref{DEFSRHO}. 

\subsection{}
\label{Floufi}

Procédant comme dans \cite{MSt} Remarque 5.7,
on définit à partir de $\KM$, 
pour tout entier $n\>1$, un foncteur 
de la catégorie des $\R$-représentations lisses de $\G_{mn}$ 
dans celle des $\R$-représentations de $\GL_{m'n}(\kd)$.
Prenant la somme directe des morphismes de groupes qu'ils définissent
et prolongeant par $0$ sur les $\RA(\G_{k},\R)$ pour les $k\>1$ non 
multiples de $m$, on forme~un morphisme de bigèbres~: 
\begin{equation}
\label{AEC}
\RA(\R) \to \bigoplus\limits_{k\>0} \RA(\GL_{k}(\kd),\R) 
\end{equation}
que l'on note encore $\KM$, les structures de bigèbre étant définies 
au moyen des foncteurs d'induction et de restriction paraboliques
(voir \eqref{VentreDieu} et \eqref{VentreDieu2}).

Etant donnés une $\R$-représentation irréductible cuspidale $\cuspif$ de $\GB$ 
et un entier $n\>1$, l'al\-gè\-bre des endo\-mor\-phis\-mes de l'in\-dui\-te 
parabolique 
$\cuspif^{\times n}$ est 
une algèbre de Hecke de type $\A$, et cette induite possède une unique 
sous-représentation irréductible
correspondant au ca\-rac\-tè\-re trivial de cette algèbre 
(voir James \cite{James} ainsi que \cite{MSc} \S3.3 et \cite{MSf} \S4.2).
On désigne par $z(\cuspif,n)$ cette sous-re\-pré\-sentation irréductible de
l'in\-dui\-te parabolique $\cuspif^{\times n}$.

\begin{prop}
\label{predeldongo}
Soit $\cuspi$ une $\R$-représentation irréduc\-ti\-ble cuspidale de degré $m$ 
et d'endo-classe $\TT$,
et soient $\cuspif_1,\dots,\cuspif_b$ les facteurs irréductibles de 
$\KM(\cuspi)$. 
Pour tout $n\>1$, on a~: 
\begin{equation*}
\KM(\Z(\cuspi,n)) = \sum\limits_{\a}
z(\cuspif_1,n_1)\times z(\cuspif_2,n_2) \times\dots\times z(\cuspif_b,n_{b}) 
\end{equation*}
où $\a$ décrit l'ensemble des familles 
$(n_1,n_2,\dots,n_b)$ d'entiers naturels de somme $n$.
\end{prop}

\begin{proof}
On procède par récurrence sur $n\>1$, 
le cas où $n=1$ étant immédiat. 
Fixons un entier $n\>2$ et notons $\A=\A(\cuspi,n)$ 
l'image de $\Z(\cuspi,n)$ par $\KM$.
Si $\a$ est une famille de $b$ entiers positifs ou nuls 
de somme $n$, notée $\a=(n_1,\dots,n_b)$, on a grâce à (\ref{P81}.1)~:
\begin{equation*}
\rp_{m'\cdot\a}(\A(\cuspi,n)) =
\KM(\Z(\cuspi,n_1))\otimes\KM(\Z(\cuspi\nu_{\cuspi}^{n_1},n_2))\otimes\dots\otimes
\KM(\Z(\cuspi\nu_{\cuspi}^{n_1+\dots+n_{b-1}},n_b))
\end{equation*}
où $\rp_{m'\cdot\a}$ désigne le foncteur de restriction parabolique 
sur la catégorie des $\R$-représentations de $\GL_{m'n}(\kd)$ 
qui correspond à la famille $m'\cdot\a=(m'n_1,m'n_2,\dots,m'n_b)$.
Compte tenu de (\ref{Racaille}.1) et de (\ref{P81}.2),
cette repré\-sen\-ta\-tion est égale à 
$\A(\cuspi,n_1)\otimes\A(\cuspi,n_2)\otimes\dots\otimes\A(\cuspi,n_b)$.

D'après \cite{MSc} Proposition 7.21, pour chaque $i\in\{1,\dots,b\}$, 
la représentation $z(\cuspif_i,n_i)$ appa\-raît dans $\A(\cuspi,n_i)$.
Par adjonction, on en déduit que l'induite~:
\begin{equation}
\label{ALBA}
z(\a) = 
z(\cuspif_1,n_1)\times z(\cuspif_2,n_2) \times\dots\times z(\cuspif_b,n_{b}) 
\end{equation}
qui est irréductible d'après \cite{MSf} Proposition 4.3 car 
$\s_1,\dots,\s_b$ sont deux à deux non isomorphes,
apparaît dans $\A$, et ce pour toute composition $\a$. 
En d'autres termes, la représentation~:
\begin{equation*}
\B = \sum\limits_{\a}
z(\cuspif_1,n_1)\times z(\cuspif_2,n_2) \times\dots\times z(\cuspif_b,n_{b})
= \sum\limits_{\a} z(\a)
\end{equation*}
est une sous-représentation de $\A$.

Etant donné $k\in\{1,\dots,n-1\}$, appliquons le 
foncteur $\rp_k=\rp_{(m'k,m'(n-k))}$.
Par hypothèse de récurrence, on trouve~:
\begin{equation*}
\rp_k(\A) = \sum\limits_{\a}\sum\limits_{\b} z(\a)\otimes z(\b)
\end{equation*}
où $\a$ décrit les compositions de $k$ et $\b$ celles de $n-k$.
Par ailleurs, d'après la formule de Mackey, on a~:
\begin{equation*}
\rp_k(\B) = 
\sum\limits_{\g} \rp_k(z(\g)) = 
\sum\limits_{\g} \sum\limits_{\delta} z(\delta)\otimes z(\g-\delta)
\end{equation*}
où $\g$ décrit les compositions $(g_1,\dots,g_b)$ de $n$ et $\delta$ les 
compositions $(d_1,\dots,d_b)$ de $k$ telles qu'on ait $d_i\<g_i$ pour tout 
$i$, et où l'on note $\g-\delta$ la composition de $n-k$ formée des entiers 
$g_i-d_i$. 
En réarrangeant l'ordre des termes, on trouve $\rp_k(\A)$. 
Par conséquent, 
la représentation $\A-\B$ est cuspidale dans le membre de droite de 
\eqref{AEC}. 

Supposons que cette différence n'est pas nulle.
Alors $\Z(\cuspi,n)$ est résiduellement non dégénérée 
au sens de \cite{MSc} Section 8 (voir aussi le paragraphe \ref{P32}), 
ce qui ne peut se produire que si $n=1$ d'après \cite{MSc} Corollaire 8.5. 
\end{proof}

\subsection{}
\label{Floufi6}

Dans ce paragraphe, on suppose que $\R$ est de ca\-rac\-téristique 
$\ell>0$. 
On fixe une $\R$-représenta\-tion irréductible cuspidale $\rho$ contenant 
le type simple maximal $\k\otimes\s$.
\textit{On suppose aussi que $\rho\nu_\rho$ est isomorphe à $\rho$ ou, 
de façon équivalente, que $\omega(\rho)=1$} (voir \eqref{DEFOMEGA}).

Soit $\upmu=[a_1,n_1]+\dots+[a_r,n_r]$~un multisegment for\-mel de 
longueur $n\>1$.
La re\-présentation 
$\cuspif^{\times n_1}\otimes\dots\otimes\cuspif^{\times n_r}$ 
contient un sous-quo\-tient irréductible non dégénéré, 
unique et apparaissant avec multiplicité $1$, noté $\sp(\cuspif,\upmu)$. 
Dans le cas particulier où $\upmu^\vee$ est égale à $(n)$, on retrouve la représentation 
non dé\-gé\-nérée $\sp(\cuspif,n)$ du paragraphe \ref{P32}.

Pour tout $n\>1$, notons $\rho^{\times n}$ le produit de $n$ copies de $\rho$.
D'après \cite{MSc} \S8.1, on a la propriété suivante. 

\medskip

\begin{tabular}{p{1cm}p{14cm}}
(\ref{Floufi6}.1) & 
\textit{$\Sp(\cuspi,\upmu)$ est l'unique sous-quotient irré\-duc\-tible de 
$\cuspi^{\times n_1}\otimes\dots\otimes\cuspi^{\times n_r}$ dont 
l'image par $\KM$ contienne $\sp(\cuspif,\upmu)$,
et celle-ci apparaît avec multiplicité $1$. }
\end{tabular}

\begin{rema}
Si $\rho\nu_\rho$ n'est pas isomorphe à $\rho$, 
il y a aussi une propriété analogue à (\ref{Floufi6}.1).
Cela nécessiterait d'introduire des notations supplémentaires et nous n'en 
aurons pas besoin.
\end{rema}

Notons $\RA(\G_{mn},\cuspi)$ le sous-groupe de 
$\RA(\G_{mn},\R)$ engendré par les sous-quotients irréductibles 
de $\cuspi^{\times n}$.

\begin{lemm}
\label{STEP4}
Supposons que $\rho\nu_\rho\simeq\rho$.
Pour tout $\Pi\in\RA(\G_{mn},\cuspi)$, on a~:
\begin{equation*}
[\rp_{m\cdot\upmu^\vee}(\Pi):\Sp(\cuspi,\upmu^\vee)] =
[\rp_{m'\cdot\upmu^\vee}(\KM(\Pi)):\sp(\cuspif,\upmu^\vee)].
\end{equation*}
\end{lemm}

\begin{proof}
Ecrivons~:
\begin{equation*}
\Pi = \sum\limits_{\pi} k(\pi)\cdot\pi
\end{equation*}
où $\pi$ décrit $\XA(\G_{mn},\R)$ et $k(\pi)\in\ZZ$.
Comme $\rho\nu_\rho\simeq\rho$,
le seul sous-quotient irréductible rési\-duel\-lement
non dégénéré pouvant apparaître dans $\rp_{m\cdot\upmu^\vee}(\pi)$ est 
$\Sp(\cuspi,\upmu^\vee)$.
On a donc~:
\begin{equation*}
\rp_{m\cdot\upmu^\vee}(\pi) = m(\pi)\cdot\Sp(\cuspi,\upmu^\vee) + \delta
\end{equation*}
avec $m(\pi)\in\NN$ et où $\delta$ ne contient aucun facteur irréductible 
résiduellement non dégénéré. 
Alors, d'après (\ref{Floufi6}.1), la représentation 
$\KM(\Sp(\cuspi,\upmu^\vee))$ contient 
$\sp(\cuspif,\upmu^\vee)$ avec multiplicité $1$, 
et $\KM(\delta)$ ne contient pas $\sp(\cuspif,\upmu^\vee)$.
\end{proof}

\begin{lemm}
\label{STEP5}
Supposons que $\rho\nu_\rho\simeq\rho$.
Alors la restriction de $\KM$ à $\RA(\G_{mn},\cuspi)$ est injective. 
\end{lemm}

\begin{proof}
D'après le paragraphe \ref{P51}, tout $\Pi\in\RA(\G_{mn},\cuspi)$ 
s'écrit comme 
combinaison $\ZZ$-linéaire de repré\-sen\-tations 
irréductibles $\Z(\upnu\boxtimes\cuspi)$
où les $\upnu$ sont des multisegments formels de longueur $n$.
Ayant sup\-po\-sé que $\cuspi\nu_\cuspi\simeq\cuspi$, on a~:
\begin{equation*}
[a,k]\boxtimes\cuspi = [0,k]\boxtimes\cuspi
\end{equation*}
pour tout segment formel $[a,k]$.
Ainsi le multisegment $\upnu\boxtimes\cuspi$ ne dépend que de $\cuspi$ 
et des longueurs des segments composant $\upnu$. 
On peut donc supposer que $\upnu$ décrit les partitions de $n$,
en identifiant une partition $(n_1,\dots,n_r)$ de $n$ avec le 
multisegment formel $[0,n_1]+\dots+[0,n_r]$.
Nous écrivons donc~:
\begin{equation*}
\label{DECPI}
\Pi = \sum\limits_{\upnu} {\sf a}(\upnu)
\cdot\Z(\upnu\boxtimes\cuspi)
\end{equation*}
où $\upnu$ décrit les partitions de $n$ et où les ${\sf a}(\upnu)$ sont 
des entiers relatifs. 
On suppose que $\KM(\Pi)=0$. 
Il s'agit de montrer que les entiers ${\sf a}(\upnu)$ sont nuls 
pour toute partition $\upnu$ de $n$.

Appliquant le lemme \ref{STEP4} à $\Pi$, on a 
pour toute partition $\upmu$ de $n$~:
\begin{equation*}
\sum\limits_{\upnu\trianglelefteq\upmu} {\textsf{e}}(\upmu,\upnu)
{\sf a}(\upnu) = 0
\end{equation*}
(voir le paragraphe \ref{P51} pour la définition de ${\textsf{e}}(\upmu,\upnu)$).
Supposons qu'il y ait une partition $\upmu$ telle que ${\sf a}(\upmu)\neq0$
et choisissons-la minimale pour cette propriété (pour l'ordre $\trianglelefteq$). 
Ecrivons~:
\begin{equation*}
{\textsf{e}}(\upmu,\upmu){\sf a}(\upmu) + 
\sum\limits_{\upnu\triangleleft\upmu}{\textsf{e}}(\upmu,\upnu){\sf a}(\upnu) = 0.
\end{equation*}
Comme ${\textsf{e}}(\upmu,\upmu)=1$ et comme les ${\sf a}(\upnu)$ 
sont tous nuls pour $\upnu\triangleleft\upmu$ par minimalité de $\upmu$, 
on trouve ${\sf a}(\upmu)=0$. 
\end{proof}

\subsection{}

Pour prouver la proposition \ref{Zcong}
nous procédons en trois étapes, la première étant le cas où $\upmu$ est un 
segment formel et où $\cuspit_2$ est inertiellement équivalente à $\cuspit_1$.

Fixons un entier $m\>1$, et posons $\G=\G_m$.

\begin{lemm}
\label{lemmenr}
Soient $\cuspit$ une $\qlb$-représentation irréductible cuspidale entière
et $\tc$ un $\qlb$-ca\-rac\-tère non ramifié entier de $\G$. 
Les conditions suivantes sont équivalentes~:
\begin{enumerate}
\item
Les représentations $\cuspit\tc$ et $\cuspit$ sont congruentes modulo $\ell$. 
\item
L'ordre de la réduction de $\tc$ mod $\ell$ divise l'entier $n(\cuspit)$
(voir le paragraphe \ref{PARA21}).
\item
Il existe un $\qlb$-caractère non ramifié entier $\tc_\ell$ de $\G$ 
tel que $\cuspit\tc=\cuspit\tc_\ell$ et dont la ré\-duc\-tion mod $\ell$ est triviale.
\end{enumerate}
\end{lemm}

\begin{proof}
Supposons que $\tc$ vérifie la condition 2. 
Notons $\chi$ sa réduction mod $\ell$ 
et $\tc_0$ le relèvement de Teich\-müller de $\chi$.
Par hypothèse, l'ordre de $\tc_0$ divise $n(\cuspit)$, 
ce qui implique que $\cuspit\tc_0=\cuspit$.
Le carac\-tère $\tc_\ell=\tc(\tc_0)^{-1}$ a une réduction mod $\ell$ 
triviale, et on a~:
\begin{equation*}
\cuspit\tc = \cuspit\tc_0\tc_\ell = \cuspit\tc_\ell,
\end{equation*}
\ie que la condition 2 implique la condition 3, elle-même impliquant 
aussitôt la condition 1. 

Supposons maintenant que les représentations $\cuspit\tc$ et $\cuspit$ sont 
congruentes modulo $\ell$. 
Soit $\cuspi$ un facteur irréductible de la réduction mod $\ell$ de $\cuspit$, 
et soit $a$ sa longueur (voir le paragraphe \ref{P13}).
Il y a donc un entier $i\in\{0,\dots,a-1\}$ tel que $\cuspi\chi=\cuspi\nu^i$.

Supposons d'abord que $a=1$. 
Dans ce cas, $\cuspit\tc$ est congrue à $\cuspit$ si et seulement si on a 
$\cuspi\chi=\cuspi$, \ie si et seulement si $\chi^{n(\cuspi)}=1$.
Grâce à \eqref{nan}, ceci est équivalent à $\chi^{n(\cuspit)}=1$.

Supposons maintenant que $a>1$.
Il y a donc un entier $i\in\ZZ$ et 
un $\flb$-caractère non ramifié $\xi$ de $\G$ tel que 
$\chi=\xi\nu^{i}$ et $\cuspi\xi=\cuspi$, cette dernière condition étant
équivalente à $\xi^{n(\cuspi)}=1$.
L'ordre de $\nu$ étant égal à $\a$, l'ordre de $q$ dans 
$(\ZZ/\ell\ZZ)^\times$, celui de $\chi$ divise le plus grand multiple 
commun à $\a$ et $n(\cuspi)$, qui vaut $\ee(\cuspi)n(\cuspi)$ 
d'après \eqref{DEFERHO}.
D'après \eqref{nan}, 
il existe un entier $v\>0$ tel que $n(\cuspit)$ soit égal à $\ee(\cuspi)n(\cuspi)\ell^v$, 
ce dont on déduit que $\chi^{n(\cuspit)}=1$.
\end{proof}

\begin{coro}
\label{coronr}
Soit $\cuspit$ une $\qlb$-représentation irréductible cuspidale entière de $\G$,
et soit $\tc$ un $\qlb$-caractère non ramifié entier de $\G$. 
Si  $\cuspit\tc$ est congrue à $\cuspit$ modulo $\ell$, 
alors $\Z(\cuspit\tc,n)$ est 
con\-grue à $\Z(\cuspit,n)$ modulo $\ell$ pour tout $n\>1$.
\end{coro}

\begin{proof}
D'après le lemme \ref{lemmenr}, nous pouvons supposer que la réduction modulo 
$\ell$ de $\tc$ est tri\-viale. 
Soit $\widetilde{\mu}$ le $\qlb$-caractère non ramifié de $\mult\F$ tel que 
$\tc=\widetilde{\mu}\circ{\rm Nrd}$.
On a~:
\begin{equation}
\label{grossepute}
\Z(\cuspit\tc,n) = \Z(\cuspit\cdot\widetilde{\mu},n) = 
\Z(\cuspit,n)\cdot\widetilde{\mu}
\end{equation}
qui est bien congrue à $\Z(\cuspit,n)$ modulo $\ell$.
\end{proof}

\subsection{Preuve de la proposition \ref{ConjRedModlSpehIntro}}
\label{Racaillette}

Nous prouvons maintenant la proposition \ref{Zcong} dans le cas d'un segment 
(cf. proposition \ref{redZ1}).
Pour cela, nous prouvons la proposition \ref{ConjRedModlSpehIntro}, 
plus précise. 

Fixons une $\qlb$-représentation irréductible cuspidale 
entiè\-re $\cuspit$ de $\G$ et un entier $n\>1$. 
Posons $a=a(\cuspit)$ et fixons un facteur irréductible $\cuspi$ 
de la réduction modulo $\ell$ de $\cuspit$.
La preuve se fait par récurrence sur $n$, le cas $n=1$ étant immédiat
(voir \eqref{redcusp}).
Posons~:
\begin{equation*}
\Delta = \Delta(\cuspit,n) = \r_\ell(\Z(\cuspit,n))
\end{equation*}
et notons $\Sigma=\Sigma(\cuspit,n)$ le membre de droite de 
\eqref{ForrZIntro}, \ie~:
\begin{equation*}
\Sigma(\cuspit,n) = \sum\limits
\Z(\cuspi,r_0)\times\Z(\cuspi\nu,r_1)\times\dots\times\Z(\cuspi\nu^{a-1},r_{a-1})
\end{equation*}
où la somme porte sur 
les familles $(r_0,\dots,r_{a-1})$ d'entiers $\>0$ de somme $n$. 
Avec les notations du paragraphe \ref{P13},
on commence par remarquer que, si $a=1$, alors le résultat est immédiat 
car, dans ce cas, $\Delta$ est égal à $\Z(\cuspi,n)$ d'après \cite{MSc} 
Théorème 9.39.

Supposons donc dorénavant qu'on a $a>1$.
Les représentations 
$\cuspi\nu_\cuspi$ et $\cuspi$ 
sont donc isomorphes (\cite{MSt} Corollaire 3.24), 
\ie que nous sommes dans les conditions d'application 
des lemmes \ref{STEP4} et \ref{STEP5}.
D'après la proposition \ref{propa}(2), le plus grand diviseur de $a$ 
premier à $\ell$ est égal à $\ee(\cuspi)$.
Pour alléger les notations, cet entier sera noté $e$ dans la suite de la preuve. 

Comme $\Z(\cuspit,n)$ est l'unique sous-représentation irréductible de~:
\begin{equation*}
\cuspit\times\cuspit\nu_{\cuspit}^{}\times\dots\times\cuspit\nu_{\cuspit}^{n-1}
\end{equation*}
et comme la réduction mod $\ell$ commute à l'induction parabolique 
(\cite{MSc} \S1.2.3), 
les sous-quotients ir\-ré\-ductibles de $\Delta$ et de $\Sigma$ 
sont, d'après \eqref{Redrt},
des sous-quotients d'induites de la forme~:
\begin{equation*}
\label{HardiPetit}
\cuspi\nu^{i_1}\times\dots\times\cuspi\nu^{i_n},
\quad
i_1,\dots,i_n\in\ZZ. 
\end{equation*}
D'après la proposition \ref{propa}(2),
si $\pi$ est un sous-quotient irréductible d'une telle induite, 
il y a donc une famille d'entiers naturels $\a=(n_0,\dots,n_{e-1})$,
de somme $n$, telle que $\pi$ soit un sous-quotient de 
l'induite~: 
\begin{equation}
\label{upthere}
\I(\cuspi,\a) = 
\cuspi^{\times n_0} \times (\cuspi\nu)^{\times n_1} \times\dots\times 
(\cuspi\nu^{e-1})^{\times n_{e-1}}.
\end{equation}
La première chose à faire est de prouver qu'une telle famille est unique. 

\begin{lemm}
\label{UniAlpha}
Soit $\g$ une famille d'entiers $\>0$ de somme $n$ telle 
que $\pi$ soit un sous-quotient de $\I(\cuspi,\g)$.
Alors $\g$ est égale à $\a$.
\end{lemm}

\begin{proof}
Si $\cuspi$ est supercuspidale, alors le résultat est une conséquence de 
l'unicité du support supercuspidal (voir \cite{MSc} Théorème 8.16).

Supposons que $\cuspi$ n'est pas supercuspidale. 
D'après la paragraphe \ref{P32}, 
il y a un unique diviseur $k=k(\cuspi)$ de $m$ et une 
représentation irréductible supercuspidale $\tau$ de $\G_{r}$ telles que 
$kr=m$ et $\cuspi$ soit isomorphe à $\Sp(\tau,k)$.
Ecrivons $\pi$ comme un sous-quotient de~:
\begin{equation*}
(\tau\times\tau\nu_\tau^{}\times\dots\times\tau\nu_{\tau}^{k-1})^{\times n_0} 
\times\dots\times
(\tau\nu^{e-1}\times\tau\nu_\tau^{}\nu^{e-1}\times\dots\times\tau 
\nu_{\tau}^{k-1}\nu^{e-1})^{\times n_{e-1}}.
\end{equation*}
Réarrangeant l'ordre des termes et posant $e'=\ee(\tau)$, 
on voit que cette induite est égale à $\I(\tau,\a')$ avec 
$\a'=(m_0,\dots,m_{e'-1})$ où pour tout $i'\in\{0,\dots,e'-1\}$ 
on note~:
\begin{equation*}
m_{i'} = \sum\limits_{(i,t)} n_i,
\end{equation*}
la somme portant sur l'ensemble $\Y(i')$ des couples 
$(i,t)\in\{0,\dots,e-1\}\times\{0,\dots,k-1\}$ tels que 
les repré\-sen\-tations 
$\tau\nu_{\tau}^{t}\nu^i$ et $\tau\nu^{i'}$ soient isomorphes.

Nous posons $s'=s(\tau)$ pour alléger les notations.
D'après le lemme \ref{STEP0} appliqué à $\tau$, 
un couple $(i,t)$ appartient à $\Y(i')$ si et seulement si $e'$ divise $i-i'+ts'$, 
auquel cas $i-i'$ est un multiple du plus grand 
diviseur commun à $e'$ et $s'$, égal à $e$ 
d'après le lemme \ref{NASR}.

Par conséquent, si $(i,t)$ appartient à $\Y(i')$, 
alors $i$ est le reste de $i'$ mod $e$.
Inversement, notons $i$ le reste de $i'$ mod $e$ et soit 
$t\in\{0,\dots,k-1\}$.
Compte tenu du lemme \ref{NASR}, on a~:
\begin{equation}
\label{Pitet}
(i,t)\in\Y(i') 
\quad\Leftrightarrow\quad
\text{$\frac{i'-i}{e}$ est le reste de $\frac{ts'}{e}$ mod $\frac{e'}{e}$}.
\end{equation}
D'après le lemme \ref{NASR} et \eqref{DEFOMEGA}, le quotient $e'e^{-1}$ 
est le plus grand diviseur de $k$ premier à $\ell$. 
Par conséquent \eqref{Pitet} admet 
$k'=kee'^{-1}$ solutions $t\in\{0,\dots,k-1\}$.
Pour $i'\in\{0,\dots,e'-1\}$, on a~:
\begin{equation}
\label{Oswald}
m_{i'} = k'n_{i}
\end{equation}
où $i\in\{0,\dots,e-1\}$ est le reste de $i'$ modulo $e$.

\begin{rema}
L'entier $e'e^{-1}$, qui est noté $\omega(\tau)$ au paragraphe \ref{P32}, 
est le plus petit entier $t\>1$ tel que $\tau\nu_{\tau}^t$ soit isomorphe à
$\tau$, et $k'$ est une puissance de $\ell$. 
\end{rema}

Le cas supercuspidal permet de conclure que les familles $\a'$ et $\g'$
sont égales, ce dont on déduit grâce à \eqref{Oswald}
que les familles $\a$ et $\g$ sont égales.
\end{proof}

Notons $\RA(\G_{mn},\cuspi,\nu)$ le sous-groupe de 
$\RA(\G_{mn},\flb)$ engendré par les sous-quotients irréducti\-bles 
des induites $\I(\rho,\a)$ où $\a$ décrit toutes les familles d'entiers 
$\>0$ de somme $n$. 
Notamment, les représentations $\Delta$ et $\Sigma$ appartiennent à ce sous-groupe.

Nous avons prouvé que, pour toute représentation irréductible $\pi$ 
dans $\RA(\G_{mn},\cuspi,\nu)$, il y a une unique famille~:
\begin{equation*}
\a=\a(\pi)=(n_0,\dots,n_{e-1})
\end{equation*} 
d'entiers $\>0$ de somme $n$ telle que $\pi$ soit un sous-quotient de 
$\I(\cuspi,\a)$ (voir \eqref{upthere}). 
La représen\-ta\-tion $\pi$ 
s'écrit de façon unique sous la for\-me d'une induite~:
\begin{equation*}
\ip_{m\cdot\a}(\pi_\a),
\quad
\pi_\a=\pi_0\otimes\dots\otimes\pi_{e-1}\in
\XA(\G_{mn_0}\times\dots\times\G_{mn_{e-1}},\R),
\end{equation*}
où $m\cdot\a$ est la famille $(mn_1,\dots,mn_{e-1})$ et où 
$\pi_i$ est, pour cha\-que entier $i\in\{0,\dots,e-1\}$,
un sous-quotient ir\-ré\-duc\-ti\-ble de l'induite $(\cuspi\nu^{i})^{\times n_i}$.

\begin{lemm}
\label{STEP1}
Soit $\pi$ une représentation irréductible dans $\RA(\G_{mn},\cuspi,\nu)$
et posons $\a=\a(\pi)$. 
Pour tout élément $\Pi\in\RA(\G_{mn},\cuspi,\nu)$,
la multiplicité de $\pi$ dans $\Pi$ est égale à la multiplicité de 
$\pi_\a$ dans $\rp_{m\cdot\a}(\Pi)$.
\end{lemm}

\begin{proof}
En effet, si l'on écrit~:
\begin{equation*}
\Pi = \sum\limits_{\tau} k(\tau)\cdot\tau \in\Gg(\G_{mn},\flb)
\end{equation*}
où $\tau$ décrit les classes de $\flb$-représentations irréductibles de 
$\G_{mn}$ et où $k(\tau)\in\ZZ$, alors~:
\begin{equation*}
\rp_{m\cdot\a}(\Pi) = \sum\limits_{\tau} k(\tau)\cdot\rp_{m\cdot\a}(\tau).
\end{equation*}
D'après le lemme géométrique \cite{MSc} \S2.4.3, 
le facteur irréductible 
$\pi_\a$ apparaît avec multi\-plicité $1$ dans $\rp_{m\cdot\a}(\pi)$.

Soit $\tau$ telle que $\pi_\a$ apparaisse dans $\rp_{m\cdot\a}(\tau)$.
Appliquant à nouveau le lemme géométrique, 
pour que la repré\-sen\-ta\-tion
$\rp_{m\cdot\a}(\tau)$ contienne un terme homogène, 
\ie un terme irréductible de la forme 
$\k_0\otimes\dots\otimes\k_{e-1}$ où $\k_i$ est pour chaque $i$
un sous-quotient irréductible de $\cuspi^{\times n_i}$, 
il faut et il suffit que $\a(\tau)$ soit égal à $\a(\pi)$. 
Ce terme homogène est alors unique, égal à $\pi_\a$.
On a donc $\tau_\a=\pi_\a$ ce qui, en induisant, donne $\tau=\pi$.
\end{proof}

Grâce à la propriété d'unicité du lemme \ref{UniAlpha},
tout $\Pi\in\RA(\G_{mn},\cuspi,\nu)$ peut être décomposé 
sous la forme~:
\begin{equation}
\label{SirThomasMore}
\Pi = \Pi_{0} + \dots + \Pi_{e-1} + \Pi_{{\rm mixte}} 
\end{equation}
où $\Pi_i$ contient les termes de $\Pi$ qui sont des 
sous-quotients irré\-duc\-ti\-bles de $(\cuspi\nu^i)^{\times n}$ et où 
$\Pi_{{\rm mixte}}$ contient les termes $\pi$ de $\Pi$ pour 
lesquels $\a(\pi)$ contient au moins deux valeurs non nulles
(elles sont donc alors toutes strictement inférieures à $n$). 

\begin{lemm}
\label{STEP2}
Soit $\cuspit$ une $\qlb$-représentation irréductible 
cuspidale entiè\-re de $\G$. 
On a~:
\begin{equation*}
\label{Dmixte}
\Delta_{{\rm mixte}}(\cuspit,n)=\Sigma_{{\rm mixte}}(\cuspit,n).
\end{equation*}
\end{lemm}

\begin{proof}
Soit $\a=(n_0,\dots,n_{e-1})$ une famille de $e$ entiers $\>0$ de somme $n$ 
dont deux au moins sont non nuls (on suppose bien sûr que $e\>2$ sinon il 
n'y a rien à prouver). 
Nous allons prouver que $\rp_{m\cdot\a}(\Delta)=\rp_{m\cdot\a}(\Sigma)$. 
Pour toute représentation irréductible $\pi$ dans 
$\RA(\G_{mn},\cuspi,\nu)$, le lemme \ref{STEP1} entraînera alors~:
\begin{equation*}
[\Delta:\pi] = [\rp_{m\cdot\a}(\Delta):\pi_\a] = 
[\rp_{m\cdot\a}(\Sigma):\pi_\a] = [\Sigma:\pi]
\end{equation*}
ce qui mettra fin à la preuve du lemme \ref{STEP2}.

Grâce à la propriété de transitivité des foncteurs de Jacquet, il suffit de
prouver cette égalité dans le cas où $\a$ n'a que deux valeurs non nulles
$k$ et $n-k$ avec $k\in\{1,\dots,n-1\}$.
On a alors~: 
\begin{equation}
\label{effort}
\rp_{m\cdot\a}(\Delta) 
= \r_\ell(\rp_{m\cdot\a}(\Z(\cuspit,n))) 
= \r_\ell(\Z(\cuspit,k)) \otimes \r_\ell(\Z(\cuspit\nu_{\cuspit}^{k},n-k)).
\end{equation}
Pour calculer ceci, on prouve le lemme suivant. 

\begin{lemm}
\label{Mamounette}
Soit $\widetilde{\omega}$ un $\qlb$-caractère non ramifié de $\G_m$ 
relevant $\nu^i$ pour un entier $i\in\ZZ$.
Alors $\rt\widetilde{\omega}$ est congrue à $\rt$ mod $\ell$.
\end{lemm}

\begin{proof}
Calculant la réduction mod $\ell$ de $\cuspit\widetilde{\omega}$ au moyen
de la formule \eqref{redcusp}, on trouve~: 
\begin{equation*}
\r_\ell(\cuspit\widetilde{\omega}) 
= (\cuspi+\cuspi\nu+\dots+\cuspi\nu^{a-1})\cdot\nu^{i}.
\end{equation*}
Ayant supposé que $a>1$, la proposition \ref{propa} implique que les 
représentations tordues $\cuspi\nu^{k}$, $k\in\ZZ$, apparaissent toutes, 
à isomorphisme près, parmi $\cuspi,\cuspi\nu,\dots,\cuspi\nu^{a-1}$.
\end{proof}

Appliquant le lemme \ref{Mamounette} et le corollaire \ref{coronr}, 
on déduit de \eqref{effort} que~:
\begin{equation*}
\rp_{m\cdot\a}(\Delta) 
= \Delta(\rt,k)\otimes\Delta(\rt,n-k)
= \Sigma(\rt,k)\otimes\Sigma(\rt,n-k),
\end{equation*}
la dernière égalité étant obtenue par hypothèse de récurrence. 

Par ailleurs, d'après le lem\-me géo\-métrique, on a~:
\begin{equation*}
\rp_{m\cdot\a}(\Sigma) = 
\sum\limits_{\g} \sum\limits_{\delta} 
\Z(\rho,\delta) \otimes \Z(\rho,\g-\delta) 
\end{equation*}
où $\g$ décrit les familles $(r_0,\dots,r_{a-1})$ d'entiers $\>0$ de somme $r$ 
et $\delta=(s_0,\dots,s_{a-1})$ les familles de somme $k$ telles que 
$0\<s_i\<r_i$ pour tout $i$.
Enfin, $\g-\delta$ désigne la famille des entiers $r_i-s_i$. 
Ceci est égal à $\Sigma(\rt,k)\otimes\Sigma(\rt,n-k)$,
ce dont on déduit le résultat annoncé. 
\end{proof}

Pour traiter les autres termes irréductibles, 
\ie les $\pi$ tels que $\a(\pi)$ n'a qu'une valeur non nulle
(donc égale à $n$), on commence par le lemme suivant. 

\begin{lemm}
\label{STEP3}
Pour toute $\flb$-représentation irréductible $\pi$ de $\G_{mn}$ 
et tout $i\in\ZZ$, les repré\-sen\-ta\-tions $\pi$ et $\pi\nu^i$ ont la 
même multiplicité dans $\Delta$ et dans $\Sigma$.
\end{lemm}

\begin{proof}
Il suffit de voir que $\Delta$ et $\Sigma$ sont invariants par torsion par 
le $\flb$-caractère $\nu$. 
Fixant un carac\-tère non ramifié $\widetilde{\omega}$ de $\G_m$ 
relevant $\nu$ et notant $|\ |_\F$ la valeur absolue sur $\mult\F$, on a 
l'égalité~:
\begin{equation*}
\label{QTESUFFPOUR}
\Delta\cdot|\ |_\F = \Delta(\cuspit\widetilde{\omega},n).
\end{equation*}
Le lemme \ref{Mamounette} et le corollaire \ref{coronr} impliquent 
que ceci est égal à $\Delta$.
Pour $\Sigma$, il suffit d'écrire~:
\begin{equation*}
\Sigma\cdot|\ |_\F = \sum\limits_{\g} 
\Z(\cuspi\nu,r_0)\times\Z(\cuspi\nu^2,r_1)\times\dots\times\Z(\cuspi\nu^{a},r_{a-1})
\end{equation*}
qui est bien égal à $\Sigma$ car la transformation 
$(r_0,r_1,\dots,r_{a-1})\mapsto(r_{a-1},r_0,\dots,r_{a-2})$ est bijec\-ti\-ve sur l'ensemble des 
familles de $a$ entiers $\>0$ de somme $n$.
\end{proof}

On déduit du lemme \ref{STEP3}, avec la notation introduite en 
\eqref{SirThomasMore}, que~:
\begin{equation}
\label{Dhom}
\Delta_i(\cuspit,n)=\Delta_0(\cuspit,n)\nu^i,
\quad
\Sigma_i(\cuspit,n)=\Sigma_0(\cuspit,n)\nu^i,
\quad
i\in\{0,\dots,e-1\}.
\end{equation}
Il ne nous reste donc plus qu'à prouver que $\Delta_0(\cuspit,n)$ 
et $\Sigma_0(\cuspit,n)$
sont égaux.
Soit $\k\otimes\s$ un type simple contenu dans $\rho$, 
et soit ${\KM}$ le morphisme \eqref{DEFKMR} associé à $\k$. 
Nous allons montrer que $\Delta$ et 
$\Sigma$ ont la même image par $\KM$.

\begin{lemm}
On a $\KM(\Delta)=\KM(\Sigma)$.
\end{lemm}

\begin{proof}
On peut supposer que $\rt$ contient un type simple de la forme 
$\kt\otimes\st$ où $\kt$ est une $\b$-extension relevant $\k$
et $\st$ une représentation cuspidale de $\GB$ relevant $\s$.
Si l'on note $\widetilde{\KM}$ le morphisme associé à $\kt$, alors 
on a la relation~:
\begin{equation*}
\r_\ell \circ \widetilde{\KM} = \KM \circ \r_\ell.
\end{equation*}
Notons $\phi$ l'automorphisme de Frobenius de $\Gal(\kd/\ke)$
et posons $\st_i=\st^{\phi^{i-1}}$ pour $i\in\{1,\dots,b(\rt)\}$.
Notant $\s_i$ la réduction mod $\ell$ de $\st_i$, on a $\s_i=\s_j$ 
si et seulement si $i$ et $j$ sont congrus mod $b(\rho)$.
Compte tenu de la relation entre les 
invariants $b$ et $s$ d'une part (voir le \S\ref{Racaille}) et la 
proposition \ref{propa}(3) d'autre part, 
le quotient de $b(\rt)$ par $b(\rho)$ est égal à $a$.
D'après \cite{MSf} Lemme 5.9 (voir aussi la preuve de \cite{MSi} Lemme 4.7), 
la réduction mod $\ell$ de $z(\st_i,r)$ est égale à $z(\s_i,r)$ pour tout 
$r\>1$. 
D'après la proposition \ref{predeldongo}, on a donc~:
\begin{eqnarray*}
\KM(\Delta) & = & \sum\limits_{\a}
z(\cuspif_1,n_1)\times z(\cuspif_2,n_2) \times\dots\times 
z(\cuspif_{b(\rt)},n_{b(\rt)}) \\
& = & \sum\limits_{\a} \prod\limits_{i=1}^{b} 
z(\s_i,n_i) \times z(\s_i,n_{i+b}) \times\dots\times z(\s_i,n_{i+(a-1)b})
\end{eqnarray*}
où $b=b(\rho)$ et $\a$ décrit l'ensemble des familles 
$(n_1,n_2,\dots,n_{b(\rt)})$ d'entiers naturels de somme $n$.
Par ailleurs, d'après (\ref{P81}.2) et (\ref{Racaille}.1), on a~:
\begin{eqnarray*}
\KM(\Sigma) &=& \sum\limits
\KM(\Z(\cuspi,r_0))\times\KM(\Z(\cuspi,r_1))
\times\dots\times\KM(\Z(\cuspi,r_{a-1})) \\
&=& \sum\limits \prod\limits_{j=0}^{a-1} 
\sum\limits_{\delta_j}
z(\cuspif_1,m_{1,j})\times z(\cuspif_2,m_{2,j}) \times\dots\times 
z(\cuspif_{b},m_{b,j})
\end{eqnarray*}
où la somme porte sur 
les familles $(r_0,\dots,r_{a-1})$ d'entiers naturels de somme $n$,
et où $\delta_j$ décrit les familles 
$(m_{1,j},m_{2,j},\dots,m_{b,j})$ d'entiers $\>0$ de somme $r_{j}$. 
Utilisant la notation \eqref{ALBA}, on a~:
\begin{equation*}
\KM(\Sigma) = \sum\limits 
z(\delta_0)\times z(\delta_1) \times \dots \times z(\delta_{a-1})
\end{equation*}
où $\delta_0,\dots,\delta_{a-1}$ décrivent les familles de $b$ entiers naturels de 
somme totale $n$.
Ecrivant $\delta_j$ sous la forme $(m_{1,j},n_{2,j},\dots,m_{b,j})$
(de sorte que la somme de tous les $m_{i,j}$, $i\in\{1,\dots,b\}$, 
$j\in\{0,\dots,a-1\}$ est égale à $n$), on note $\a$ la famille 
$(n_1,\dots,n_{b(\rt)})$ définie par~:
\begin{equation*}
n_{i+bj} = m_{i,j},
\quad
i\in\{1,\dots,b\},\ 
j\in\{0,\dots,a-1\}.
\end{equation*}
L'opération $(\delta_0,\dots,\delta_{a-1})\mapsto\a$ est injective, 
et son image donne toutes les familles 
de $b(\rt)$ entiers naturels de somme $n$.
On en déduit le résultat voulu.
\end{proof}

Appliquant conjointement le lemme \ref{STEP2}, l'identité \eqref{Dhom} et 
la propriété (\ref{Racaille}.2), on en déduit~:
\begin{equation*}
\label{EGA0}
\KM(\Delta^{}_0(\cuspit,n))=\KM(\Sigma_0(\cuspit,n)).
\end{equation*}
Appliquant le lemme \ref{STEP5}, 
cela donne $\Delta_0(\cuspit,n)=\Sigma_0(\cuspit,n)$, pour tout entier 
$n\>2$, ce qui met fin à la preuve de la proposition \ref{ConjRedModlSpehIntro}.
Elle a pour conséquence immédiate le résultat suivant. 

\begin{prop}
\label{redZ1}
Soient $\cuspit_1$, $\cuspit_2$ deux $\qlb$-représentations irréductibles cuspidales 
entiè\-res de $\G$, et soit un entier $n\>1$.
Les représentations irréductibles 
$\Z(\cuspit_1,n)$ et $\Z(\cuspit_2,n)$ sont congruentes 
si et seu\-le\-ment si $\cuspit_1$ et $\cuspit_2$ sont congruentes.
\end{prop}

\begin{proof}
Pour l'implication directe,
il suffit d'appliquer le foncteur de Jacquet $\rp_{\a}$ avec  
$\a=(m,\dots,m)$ et d'utiliser la formule~:
\begin{equation*}
\rp_\a(\Z(\cuspit,n)) = 
\cuspit\otimes\cuspit\nu_{\cuspit}^{}\otimes\dots\otimes
\cuspit\nu_{\cuspit}^{n-1}
\end{equation*}
et le fait que la réduction mod $\ell$ commute aux foncteurs de Jacquet 
(voir \cite{MSc} \S 1.2.4).
\end{proof}

\subsection{Preuve da la proposition \ref{Zcong} dans le cas général}
\label{CasGenZcong}

Soit $\cuspit$ une $\qlb$-représentation irréductible cus\-pi\-dale entière
de degré $m\>1$. 
Etant donné un multisegment for\-mel $\upmu$, écrivons~:
\begin{equation*}
\Z(\upmu\boxtimes\cuspit) = 
\sum\limits_{\upnu} 
\textsf{n}(\upmu,\upnu,\cuspit)\cdot\I(\upnu\boxtimes\cuspit) 
\end{equation*}
où l'on note~:
\begin{equation*}
\I(\upnu\boxtimes\cuspit) = 
\Z([a_1,n_1]\boxtimes\cuspit)\times\dots\times\Z([a_r,n_r]\boxtimes\cuspit)
\end{equation*}
pour tout multisegment formel $\upnu=[a_1,n_1]+\dots+[a_r,n_r]$ et où les 
$\textsf{n}(\upmu,\upnu,\cuspit)$ sont dans $\ZZ$. 
Avec les notations du paragraphe \ref{PARA21}, on a donc~:
\begin{equation*}
\sum\limits_{\uplambda} 
\textsf{m}(\upmu,\uplambda,\cuspit)\textsf{n}(\uplambda,\upnu,\cuspit)=
\left\{
\begin{array}{ll}
1 & \text{si $\upmu=\upnu$,} \\
0 & \text{sinon,}
\end{array}
\right.
\end{equation*}
$\uplambda$ décrivant les multisegments formels.
Réduisant modulo $\ell$, on obtient~:
\begin{equation*}
\r_\ell(\Z(\upmu\boxtimes\cuspit)) 
=\sum\limits_{\upnu} 
\textsf{n}(\upmu,\upnu,\cuspit)\cdot\r_\ell(\I(\upnu\boxtimes\cuspit)).
\end{equation*}
Soient $\cuspit_1$ et $\cuspit_2$ des $\qlb$-représentations 
irréductibles cuspidales entières con\-gruentes. 
D'après la pro\-po\-sition \ref{redZ1}, et comme $\r_\ell$ commute à l'induction 
parabolique, on a~:
\begin{equation*}
\r_\ell(\I(\upnu\boxtimes\cuspit_1)) = \r_\ell(\I(\upnu\boxtimes\cuspit_2)). 
\end{equation*} 
Par ailleurs, d'après \cite{MSi} Proposition 4.15, on a l'égalité 
$\textsf{n}(\upmu,\upnu,\cuspit_1)=\textsf{n}(\upmu,\upnu,\cuspit_2)$. 
Ceci met fin à la preuve de la proposition \ref{Zcong}.

\subsection{Preuve du théorème \ref{CongSpehIntro}}
\label{P7a}

Soit $\tp$ une représentation de Speh entière de $\G_m$,
qu'on écrit $\Z(\cuspit,r)$ avec $r$ un diviseur de $m$ et $\cuspit$ une 
représentation irréductible cuspidale entière de $\G_{mr^{-1}}$.
D'après la définition \ref{BenOuiMonGars},
on a $c(\tp)=c(\cuspit)$.
D'après \eqref{grossepute} et la proposition 
\ref{redZ1}, on a $t(\tp)=t(\cuspit)$.
Enfin, on a $w(\tp)=w(\cuspit)$ par définition.
Le résultat se déduit alors de la proposition \ref{CongCusp}.

\subsection{Preuve du théorème \ref{ComptagePIIntro}}
\label{P7b}

D'après le paragraphe \ref{PARA21}, l'application~:
\begin{equation}
\label{BIFfin}
(\cuspit,r) \mapsto \Z(\cuspit,r)
\end{equation}
est une bijection entre l'ensemble des paires formées d'un diviseur $r$ de 
$m$ et d'une classe d'iso\-morphisme de $\qlb$-repré\-sentation irréductible
cuspidale $\cuspit$ de $\G_{mr^{-1}}$, et l'ensemble $\Zz(\G_m,\qlb)$. 

Fixant un entier $w\>1$ et un nombre rationnel $j\>0$, 
l'application \eqref{BIFfin} induit une bijection entre~:
\begin{enumerate}
\item
classes d'iso\-morphisme de 
$\qlb$-repré\-sentations irréductibles cuspidales $\cuspit$
de degré divisant $m$, de niveau nor\-malisé inférieur ou égal à $j$,
et telles que $w(\cuspit)=w$~;
\item
et l'ensemble des $\tp\in\Zz(\G_m,\qlb)$ de niveau nor\-malisé inférieur ou 
égal à $j$. 
\end{enumerate}

Enfin, la proposition \ref{redZ1} et la propriété (\ref{P81}.2)
impliquent que, en réduisant mod $\ell$, on obtient une bijection~:
\begin{equation*}
\r_\ell(\langle\cuspit\rangle) \mapsto \r_\ell(\langle\Z(\cuspit,r)\rangle)
\end{equation*}
entre les ensembles finis $\Aa_\ell(\D,w,j)$ et $\Ee_\ell(\G_m,w,j)$,
ce qui prouve le théorème \ref{ComptagePIIntro}.

\section{La correspondance de Jacquet-Langlands dans le 
cas complexe} 
\label{JLcomplexe}

Dans cette section et la suivante, on fixe un entier $n\>1$
et une $\F$-algèbre à division cen\-trale de 
de\-gré réduit $n$ dont on note $\A$ le groupe multiplicatif. 
On désigne par $\CC$ le corps des nombres comple\-xes.

Soit $\G=\GL_m(\D)$ une forme in\-té\-rieure 
de $\H=\GL_n(\F)$. 
Traditionnellement, la corres\-pondan\-ce de Jacquet-Langlands locale 
est une bi\-jec\-tion de la série discrète de $\G$ vers celle de $\H$. 
Pour nous, il est préférable de choisir $\A$ plutôt que $\H$ comme 
groupe de référence. 

\subsection{}
\label{DEFLJ}

Soit $\Dd(\G,\CC)$ l'ensemble des classes 
de re\-pré\-sentations lisses com\-plexes, 
irréductibles et es\-sen\-tiellement de carré intégrable de $\G$. 
La correspon\-dan\-ce de Jacquet-Langlands locale 
(\cite{Rog,DKV,BaduJL}) est une bi\-jec\-tion~:
\begin{equation}
\label{JL}
\boldsymbol{j} : \Dd(\G,\CC)\to\XA(\A,\CC)
\end{equation}
carac\-térisée par une identité de caractères sur les 
classes de conjugaison elliptiques et régu\-liè\-res.
Elle préserve le caractère central, le degré formel \cite{DKV,Rog,BHL}
et le niveau normalisé {\cite{ABPS}}. 
Comme elle est com\-patible à la torsion par un 
carac\-tère de $\mult\F$, elle préserve aussi le nombre de torsion. 

Si $r$ est un diviseur de $m$ et $\cuspi$ une représentation irréductible 
cuspidale de $\GL_{mr^{-1}}(\D)$, l'induite parabolique~:
\begin{equation*}
\label{MariePierre}
\cuspi\times\cuspi\nu_\cuspi^{}\times\dots\times\cuspi\nu_{\cuspi}^{r-1}
\end{equation*}
admet un unique quotient irréductible, noté $\L(\cuspi,r)$~;
celui-ci est essen\-tiellement de carré intégra\-ble,
et tout élément de $\Dd(\G,\CC)$ s'obtient de cette façon 
(\cite{BHLS} \S2.2).

Soit $r\>1$, soient des entiers $m_1,\dots,m_r\>1$ 
de somme $m$ et, pour chaque $i\in\{1,\dots,r\}$, 
soit $\pi_i$ une représentation irréductible essentiellement de 
carré intégrable de $\GL_{m_i}(\D)$. 
Les induites para\-bo\-li\-ques de la forme~: 
\begin{equation}
\label{RSTD}
\pi_1\times\dots\times\pi_r \in \Gg(\G,\CC)
\end{equation}
forment une base du groupe abélien libre $\Gg(\G,\CC)$, 
appelée base {standard}. 
Il y a donc un unique morphisme surjectif de groupes abéliens~:
\begin{equation}
\label{JLmor}
\JL : \Gg(\G,\CC)\to\Gg(\A,\CC)
\end{equation} 
qui soit nul sur l'ensemble des induites parabo\-li\-ques de la forme 
\eqref{RSTD} avec $r\>2$, et qui coïncide avec la bijection $\boldsymbol{j}$ sur 
$\Dd(\G,\CC)$. 

\subsection{}

Etant donnée une représentation irréductible 
$\pi\in\XA(\G,\CC)$, on note $\pi^*$ sa duale de Zelevinski. 
L'image de $\Dd(\G,\CC)$ par l'in\-vo\-lution de Zelevinski est l'ensemble 
$\Zz(\G,\CC)$ des classes de représen\-ta\-tions de Speh de $\G$.
Pour le lemme suivant, voir aussi \cite{BaduJIMJ} \S3.5.

\begin{lemm}
Soit une représentation $\pi\in\Dd(\G,\CC)$,
qu'on écrit $\L(\cuspi,r)$
avec $r$ un divi\-seur de $m$ et $\cuspi$ une 
re\-pré\-sentation cuspidale de $\GL_{mr^{-1}}(\D)$.
Alors~:
\begin{equation*}
\label{Melmotte}
\JL\left(\pi^*\right)=(-1)^{r-1}\cdot\boldsymbol{j}(\pi).
\end{equation*}
\end{lemm}

\begin{proof}
Dans le groupe de Grothendieck de $\G$, on a~:
\begin{equation}
\label{BMX}
\pi^* = \sum\limits_{\a} (-1)^{r-n(\a)} \cdot \ip_\a\circ\rp_\a(\pi)
\end{equation}
où $\a$ décrit les familles d'entiers $\>1$ de somme $r$ et où $n(\a)$ est le 
nombre de termes de $\a$. 
En appliquant $\JL$, le seul terme de la somme qui contribue est celui 
correspondant à $\a=(r)$.
On en déduit le résultat voulu.
\end{proof}

Pour toute représentation $\pi\in\Zz(\G,\CC)$, 
qu'on écrit $\Z(\cuspi,r)$
où $r$ est un divi\-seur de $m$ et $\cuspi$ une 
re\-pré\-sentation cuspidale de $\GL_{mr^{-1}}(\D)$,
on pose $\epsilon(\pi) = (-1)^{r-1}$.
Ainsi, l'application~:
\begin{equation}
\label{poncelet}
\boldsymbol{j}^* : \pi \mapsto \boldsymbol{j}(\pi^*) = \epsilon(\pi)\cdot\JL(\pi) 
\end{equation}
est une bijection de $\Zz(\G,\CC)$ sur $\XA(\A,\CC)$.

\section{Le théorème principal}
\label{SEC3THPR}

On conserve les notations de la section \ref{JLcomplexe}. 
On fixe un isomorphisme de corps $\iota:\CC\simeq\qlb$.
On note $\TJL_\ell$ le morphisme de $\Gg(\G,\qlb)$ dans $\Gg(\A,\qlb)$
et $\tlj_\ell$ la bijection de $\Zz(\G,\qlb)$ dans $\XA(\A,\qlb)$ obtenus 
à partir de \eqref{JLmor} et \eqref{poncelet} 
par change\-ment du corps des coefficients. 

L'image d'une représentation de Speh $\ell$-adique $\tp$ de $\G$ par 
$\tlj_\ell$ s'appelle le \textit{transfert} de $\tp$ à $\A$.

\begin{rema}
Le morphisme de groupes $\TJL_\ell$ et la bijection
$\tlj_\ell$ peuvent aussi être définis de la même façon que 
$\JL$ et $\boldsymbol{j}^*$ l'ont été à partir de $\jl$,
au moyen de la correspondance~: 
\begin{equation*}
\label{JLtilde}
\tjl_\ell : \Dd(\G,\qlb)\to\XA(\A,\qlb)
\end{equation*}
obtenue à partir de \eqref{JL} par changement du corps des coefficients.
La correspondance $\tjl_\ell$ 
ne dé\-pend pas du choix de l'isomorphisme de corps 
$\iota:\CC\simeq\qlb$ car, 
pour tout automorphisme de corps $\a$ de $\CC$ et toute représentation
irréductible complexe $\pi$ de $\G$, le caractère de Harish-Chandra 
$\uptheta_{\pi^\a}$ de la représentation tordue $\pi^\a$ est égal à 
$\a\circ\uptheta_\pi$, où $\uptheta_\pi$ est le caractère de Harish-Chandra de $\pi$, 
et le signe $(-1)^{m-1}$ apparaissant dans la correspondance est invariant par 
$\a$.
Par conséquent, ni $\TJL_\ell$ ni $\tlj_\ell$ ne
dépendent du choix de l'isomorphisme de corps $\iota$. 
\end{rema}

\subsection{}

Dans \cite{Datj} 
Dat définit le caractère de Brauer ${\widetilde\uptheta}_\pi$ 
d'une $\flb$-représentation lisse ir\-ré\-duc\-tible $\pi$ 
d'un groupe réductif $p$-adique. 
Rappelons-en les principales propriétés pour le groupe $\G$.

Notons $\G^{{\rm crs}}$ 
le sous-ensemble ouvert de $\G$ formé des élément semi-simples 
réguliers et compacts modulo le centre de $\G$,
et notons $\G^{{\rm crs}}_{\ell'}$ le sous-ensemble de $\G^{{\rm crs}}$ formé 
des éléments d'ordre premier 
à $\ell$ modulo le centre (voir \cite{Datj} \S2.1 pour les définitions précises).

A toute représentation lisse ir\-ré\-duc\-tible $\pi\in\XA(\G,\flb)$ 
on associe une fonc\-tion~:
\begin{equation*}
{\widetilde\uptheta}_\pi\in\Cc^{\infty}(\G^{{\rm crs}}_{\ell'},\zlb)^\G
\end{equation*}
ap\-pe\-lée son caractère de Brauer, $\Cc^{\infty}(\G^{{\rm crs}}_{\ell'},\zlb)^\G$ 
étant l'espace des fonctions localement cons\-tan\-tes sur $\G^{{\rm crs}}_{\ell'}$, 
à valeurs dans $\zlb$ et $\G$-invariantes par conjugaison. 

L'application $\pi\mapsto{\widetilde\uptheta}_\pi$ 
induit un morphisme de groupes rendant
commutatif le diagramme~: 
\begin{equation*}
\begin{CD}
\Gg(\G,\qlb)^{{\rm e}} @>{\uptheta}>>
\Cc^{\infty}(\G^{{\rm crs}},\zlb)^\G\\
@V{\r_\ell}VV @VV {|\G^{{\rm crs}}_{\ell'}} V \\
\Gg(\G,\flb)@>>{{\widetilde\uptheta}}>\Cc^{\infty}(\G^{{\rm crs}}_{\ell'},\zlb)^\G\\
\end{CD}
\end{equation*}
où $\Gg(\G,\qlb)^{{\rm e}}$ désigne le sous-groupe de $\Gg(\G,\qlb)$ formé 
des représentations entières, 
où $\uptheta$ désigne le caractère de Harish-Chandra
et où le morphisme vertical de droite est la restric\-tion à 
$\G^{{\rm crs}}_{\ell'}$. 
Le morphisme $\r_\ell$ étant surjectif (\cite{MSc} Théorème 9.40), 
ce diagramme détermi\-ne ${\widetilde\uptheta}$ de façon unique. 

En particulier, lorsque $\G=\A$, le morphisme $\widetilde\uptheta$ est injectif 
(\cite{Datj} Proposition 2.3.1).

\begin{prop}
\label{Leighton}
Il existe un unique morphisme de groupes~:
\begin{equation}
\label{LJflb}
\JL_{\ell}:\Gg(\G,\flb)\to\Gg(\A,\flb)
\end{equation}
tel que le diagramme~:
\begin{equation*}
\begin{CD}
\Gg(\G,\qlb)^{{\rm e}} @>{\TJL_\ell}>>
\Gg(\A,\qlb)^{{\rm e}}\\
@V{\r_\ell}VV @VV {\r_\ell} V \\
\Gg(\G,\flb) @>>{\JL_\ell}> \Gg(\A,\flb)\\
\end{CD}
\end{equation*}
soit commutatif.
\end{prop}

\begin{proof}
La preuve de Dat (voir \cite{Datj} (3.1.2)) est encore valable ici. 
\end{proof}

\begin{rema}
\begin{enumerate}
\item
Lorsque $\G=\H$,
le morphisme $\JL_\ell$ est le morphisme noté ${\rm LJ}_{\flb}$ dans
\cite{Datj} Théorème 1.2.3.
\item
Soit $\A'$ le groupe multiplicatif d'une $\F$-algèbre à division de degré 
réduit $n$. 
Notons~:
\begin{equation*}
\JL'_\ell : \Gg(\G,\flb)\to\Gg(\A',\flb)
\end{equation*}
le morphisme obtenu en remplaçant $\A$ par $\A'$ dans \eqref{LJflb}, 
et notons
$\textbf{\textsf{P}}_\ell : \Gg(\A',\flb)\to\Gg(\A,\flb)$
celui obtenu en y remplaçant $\G$ par $\A'$.
Alors $\textbf{\textsf{P}}_\ell$ est bijectif, 
et $\JL_\ell^{}=\textbf{\textsf{P}}_\ell^{}\circ\JL'_\ell$.
En effet, c'est vrai sur $\CC$, donc sur $\qlb$, 
et le résultat suit par com\-pa\-tibilité à la ré\-duc\-tion mod $\ell$. 
\end{enumerate}
\end{rema}

\subsection{Preuve du théorème principal}

Pour tout entier $w\>0$, on pose~: 
\begin{equation}
\label{DefZl}
\Zz_{w}(\G,\qlb)=
\{\tp\in\Zz(\G,\qlb)\ |\ \tp \text{ est entière et $w(\tp)=w$}\}.
\end{equation}
Lorsque $w=1$, 
c'est l'ensemble des représentations irréductibles $\ell$-adiques entières 
$\ell$-\textit{super-Speh} de $\G$ au sens de la définition 
\ref{DEFlsspehintro} et de \cite{Datj}.
On note~:
\begin{equation*}
\Zz_{w}(\G,\flb) \subseteq \RA(\G,\flb)
\end{equation*}
l'image par $\r_\ell$ de l'ensemble $\Zz_{w}(\G,\qlb)$.

En général
$\Zz_{w}(\G,\flb)$ n'est pas formé de représentations irréductibles. 
Toutefois, pour $w=1$, l'en\-semble $\Zz_1(\G,\flb)$ est 
le sous-ensemble de $\Zz(\G,\flb)$ formé des représen\-ta\-tions super-Speh de
$\G$, \ie celles dont le support cuspidal est supercuspidal. 

Remarquons également que, 
dans le cas où $\G$ est le groupe $\H=\GL_n(\F)$, 
l'ensemble $\Zz_w(\H,\flb)$ est inclus dans $\Zz(\H,\flb)$ pour 
tout $w\>0$. 

\begin{theo}
\label{RogerCarbury}
Soit un entier $w\>0$.
\begin{enumerate}
\item 
Pour toute représentation $\pi\in\Zz_w(\G,\flb)$, il y a un signe
$\epsilon(\pi)\in\{-1,+1\}$ tel que~:
\begin{equation}
\label{imageausignepres}
\epsilon(\pi)\cdot\JL_\ell(\pi) 
\end{equation}
appartienne à $\Zz_w(\A,\flb)$~; 
on note $\boldsymbol{\jmath}{}_\ell^{\boldsymbol{*}}(\pi)$ la quantité 
\eqref{imageausignepres}, 
qu'on appelle le {\rm transfert} de $\pi$ à $\A$.
\item
L'application $\pi\mapsto\boldsymbol{\jmath}{}_\ell^{\boldsymbol{*}}(\pi)$ 
de $\Zz_w(\G,\flb)$ dans $\Zz_w(\A,\flb)$
est bijective.
\item 
La bijection $\tlj_{\ell}$ induit une bijection de 
$\Zz_{w}(\G,\qlb)$ dans $\Zz_{w}(\A,\qlb)$ 
et le diagramme~:
\begin{equation*}
\begin{CD}
\Zz_{w}(\G,\qlb) @>{\tlj_{\ell}}>> 
\Zz_{w}(\A,\qlb) \\ 
@V{\r_\ell}VV @VV {\r_\ell}V\\
\Zz_w(\G,\flb) @>>{\boldsymbol{\jmath}{}_\ell^{\boldsymbol{*}}}>
\Zz_w(\A,\flb) 
\end{CD}
\end{equation*}
est commutatif. 
\end{enumerate}
\end{theo}

\begin{rema}
Si $\pi\in\Zz_w(\G,\flb)$ est la réduction mod $\ell$ de 
$\tp\in\Zz_w(\G,\qlb)$, alors le signe $\epsilon(\pi)$ apparaissant dans 
\eqref{imageausignepres} est égal à $\epsilon(\tp)$.
\end{rema}

\begin{proof}
La preuve se fait par récurrence sur $w$.
Le théorème est vrai lorsque $w=0$, puisque dans ce cas tous les 
ensembles concernés sont vides. 
Fixons un entier $w\>1$ et supposons que le théorème est vrai pour tout
$w'\<w-1$. 

Supposons d'abord que l'ensemble $\Zz_{w}(\G,\qlb)$ est non vide.
Dans ce cas, le lemme \ref{wadm} montre que le plus grand diviseur de $w$ 
premier à $\ell$ divise $\ell-1$.

Soit $\tp\in\Zz_w(\G,\qlb)$ entière, et soit $\rt$ son image par 
$\tlj_\ell$.
Par hypothèse de récurrence, $\tlj_\ell$ envoie bijective\-ment la réunion des 
$\Zz_{w'}(\G,\qlb)$ pour $w'<w$ sur la réunion des 
$\Zz_{w'}(\A,\qlb)$ pour $w'<w$.
On en déduit que~:
\begin{equation}
\label{inegw}
w\<w(\rt).
\end{equation}
La bijection $\tlj_{\ell}$ étant compatible à la torsion par un caractère, on a 
l'égalité $n(\tp)=n(\rt)$, donc~:
\begin{equation}
\label{egc}
c(\tp)=c(\rt).
\end{equation}
On note $c$ cette valeur commune donnée par \eqref{egc}.

\begin{lemm}
\label{MrsBarbour}
On a $t(\tp)\<t(\rt)$.
\end{lemm}

\begin{proof}
La bijection $\tlj_{\ell}$ étant compatible à la torsion par un caractère, 
elle induit une bi\-jec\-tion entre classes de torsion de $\Zz(\G,\qlb)$ et 
classes inertielles de $\XA(\A,\qlb)$, encore notée $\tlj_{\ell}$.
Soit $\Oo(\tp)$ l'en\-semble des classes de torsion des 
éléments de $\Zz(\G,\qlb)$ qui sont 
entiers et congrus à la classe de torsion de $\tp$.

Si la classe de torsion $\langle\tp_1\rangle$ est dans $\Oo(\tp)$, 
le diagramme commutatif de la proposition \ref{Leighton} 
et la compatibilité de $\tlj_{\ell}$ à la tor\-sion montrent que l'image $\rt_1$ 
de $\tp_1$ a sa classe inertielle dans $\Oo(\rt)$. 
La bijection $\tlj_{\ell}$ induit donc 
une injection de $\Oo(\tp)$ dans $\Oo(\rt)$, ce dont on déduit l'inégalité 
$t(\tp)\<t(\rt)$. 
\end{proof}

\begin{lemm}
\label{MrBarbour}
On a $t(\rt)=t(\tp)$ et $w(\rt)=w$.
\end{lemm}

\begin{proof}
Partons des inégalités~:
\begin{equation*}
t(\tp) \< t(\rt) \< c
\end{equation*}
données par le lemme \ref{MrBarbour} et le théorème \ref{CongSpehIntro}.

Si $w=1$, alors $t(\tp)=c$ d'après le théorème \ref{CongSpehIntro}, 
ce dont on déduit l'égalité $t(\rt)=t(\tp)$.
On en déduit aussi que $t(\rt)=c$, ce qui implique que $w(\rt)=1$ d'après 
la proposition \ref{CongSpehIntro}. 

Si $1<w<\ell$, alors d'après le théorème \ref{CongSpehIntro} on a~:
\begin{equation*}
\frac{c-1}{w} = t(\tp) \< t(\rt) \< \frac{c-1}{w(\rt)}.
\end{equation*}
On en déduit que $w(\rt)\<w$, ce qui, avec \eqref{inegw}, entraîne $w(\rt)=w$, 
puis $t(\rt)=t(\tp)$.

Enfin, si $w\>\ell$, alors on a aussi $w(\rt)\>\ell$ d'après \eqref{inegw}.
Par conséquent, d'après le lemme \ref{wadm}, les entiers 
$w$ et $w(\rt)$ sont divisibles par $\ell$. 
D'après le théorème \ref{CongSpehIntro}, on a~:
\begin{equation*}
\frac{c(\ell-1)}{w\ell} = t(\tp) \< t(\rt) = \frac{c(\ell-1)}{w(\rt)\ell}.
\end{equation*}
On en déduit que $w(\rt)\<w$, ce qui, avec \eqref{inegw}, entraîne $w(\rt)=w$, 
puis $t(\rt)=t(\tp)$.
\end{proof}

L'égalité $t(\rt)=t(\tp)$ entraîne le corollaire suivant. 

\begin{coro}
\label{Gaston}
L'image de $\Oo(\tp)$ par $\tlj_{\ell}$ est égale à $\Oo(\rt)$.
\end{coro}

L'égalité $w(\rt)=w$ implique que $\tlj_{\ell}$ induit une injection~:
\begin{equation*}
\Zz_{w}(\G,\qlb)\to\Zz_{w}(\A,\qlb).
\end{equation*}
Le corollaire \ref{Gaston} implique l'existence d'une application injective 
$\boldsymbol{\jmath}{}_\ell^{\boldsymbol{*}}$ de $\Zz_w(\G,\flb)$ dans 
$\Zz_w(\A,\flb)$ telle qu'on ait l'égalité~:
\begin{equation*}
\r_\ell\circ\tlj_{\ell} = \boldsymbol{\jmath}{}_\ell^{\boldsymbol{*}}\circ\r_\ell
\end{equation*}
sur $\Zz_{w}(\G,\qlb)$.
(L'existence de $\boldsymbol{\jmath}{}_\ell^{\boldsymbol{*}}$ provient du fait que l'image de $\Oo(\tp)$ par
$\tlj_{\ell}$ est incluse dans $\Oo(\rt)$ et son injectivité de ce que 
cette image est exactement $\Oo(\rt)$.)
Le morphisme $\r_\ell$ étant sur\-jec\-tif, $\boldsymbol{\jmath}{}_\ell^{\boldsymbol{*}}$ est unique.
Pour prouver que $\boldsymbol{\jmath}{}_\ell^{\boldsymbol{*}}$ est bijective, nous
utilisons un argument de 
comptage. 

Fixons un nombre rationnel $j\>0$ et notons~: 
\begin{equation*}
\Ee_\ell(\G,w,j)
\end{equation*}
l'en\-sem\-ble des $\r_\ell(\langle\tp\rangle)$ 
où $\tp\in\Zz_w(\G,\qlb)$ est de niveau normalisé inférieur ou égal à $j$
(para\-graphe \ref{Mafoi}).
Par compatibilité à la torsion et comme $\tlj_\ell$ préserve le 
niveau normalisé, l'application $\boldsymbol{\jmath}{}_\ell^{\boldsymbol{*}}$ induit une application injective~:
\begin{equation}
\label{souffrance}
\Ee_\ell(\G,w,j)\to\Ee_\ell(\A,w,j).
\end{equation}
D'après le théorème \ref{ComptagePIIntro},
ces ensembles sont finis et de même cardinal~; l'application 
\eqref{souffrance} est donc bijective. 

Considérons maintenant un élément de $\Zz_w(\A,\flb)$,
que l'on écrit $\r_\ell(\rt)$.
Si l'on note $j$ le niveau normalisé de $\rt$,
alors $\r_\ell(\langle\tp\rangle)$ a un antécédent par \eqref{souffrance}, 
qu'on écrit $\r_\ell(\langle\tp\rangle)$, et on vérifie que $\r_\ell(\tp)$ est 
un antécédent de $\r_\ell(\rt)$ par $\boldsymbol{\jmath}{}_\ell^{\boldsymbol{*}}$.
Ainsi $\boldsymbol{\jmath}{}_\ell^{\boldsymbol{*}}$ est bijective, ce qui prouve l'assertion 2 du théorème. 

Pour toute représentation 
$\pi\in\Zz(\G,\flb)$, il existe donc un signe $\epsilon(\pi)$ tel que 
$\JL_\ell(\pi)=\epsilon(\pi)\cdot\boldsymbol{\jmath}{}_\ell^{\boldsymbol{*}}(\pi)$, ce qui prouve 
l'assertion 1 du théorème.

Soit maintenant $\rt$ dans $\Zz_{w}(\A,\qlb)$. 
Par surjectivité de $\boldsymbol{\jmath}{}_\ell^{\boldsymbol{*}}$ et de $\r_\ell$, 
la représentation 
$\r_\ell(\rt)$ a un antécédent $\tp$ dans $\Zz_{w}(\G,\qlb)$, 
ce qui prouve l'assertion 3. 

Pour terminer la démonstration du théorème \ref{RogerCarbury}, 
supposons enfin que l'ensemble $\Zz_{w}(\G,\qlb)$ est vide.
Dans ce cas, $\Ee_\ell(\G,w,j)$ est vide 
pour tout nombre rationnel $j\>0$.
D'après le théorème \ref{ComptagePIIntro}, l'ensemble 
$\Zz_{w}(\A,\flb)$ est vide lui aussi.
L'ensemble $\Zz_{w}(\A,\qlb)$ est donc également vide, 
ce dont on déduit que le théo\-rème est vrai dans ce cas. 
\end{proof}

Le corollaire suivant, qui généralise \cite{Datj} Théorème 1.2.4, s'obtient 
en appliquant le théorème \ref{RogerCarbury} avec $w=1$
(voir le corollaire \ref{superSpehIntrotro}).

\begin{coro}
\label{superSpeh}
La bijection $\boldsymbol{\jmath}{}_\ell^{\boldsymbol{*}}$ induit une bijection 
entre l'ensemble $\Zz_1(\G,\flb)$ 
des re\-pré\-sen\-ta\-tions super-Speh de $\G$ et $\XA(\A,\flb)$. 
\end{coro}

\subsection{Preuve du théorème \ref{MAINTHEOREM}}

Après avoir fait un détour par les représentations de Speh, 
nous montrons comment déduire le théorème \ref{MAINTHEOREM} 
du théorème \ref{RogerCarbury}.

Soit $\tp$ une représentation de $\Dd(\G,\qlb)$, qu'on écrit $\L(\cuspit,r)$ où $r$ 
est un diviseur de $m$ et $\cuspit$ une re\-pré\-sentation irréductible cuspidale 
$\ell$-adique de $\GL_{mr^{-1}}(\D)$.
Notant $\omega_{\cuspit}$ le caractère central de $\cuspit$, celui de $\tp$ est 
égal à~:
\begin{equation*}
\omega^r_{\cuspit}\cdot(\nu_{\cuspit}^{})^{r(r-1)/2}.
\end{equation*}
La représentation $\tp$ est entière si et seulement si $\cuspit$ l'est, 
et $\cuspit$ est entière si et seulement si $\omega_{\cuspit}$ l'est.
Par conséquent, $\tp$ est entière si et seulement si son caractère central l'est.
La correspondance~:
\begin{equation*}
\tjl_\ell : \Dd(\G,\qlb)\to\XA(\A,\qlb)
\end{equation*}
obtenue à partir de \eqref{JL} par changement du corps des coefficients 
préserve le caractère central.
On en déduit que $\tp$ est entière si et seulement si son image par 
$\tjl_\ell$ l'est.

Soient maintenant deux re\-pré\-sentations entières $\tp_1,\tp_2$ de $\Dd(\G,\qlb)$.
Appliquant l'involution de Zelevinski, on obtient deux re\-pré\-sentations 
$\tp_1^*,\tp_2^*\in\Zz(\G,\qlb)$, qui sont entiè\-res car de même caractère 
central que $\tp_1,\tp_2$.
D'après le théorème \ref{RogerCarbury}, les représentations $\tp_1^*,\tp_2^*$ 
sont congruentes mod $\ell$ si et seulement si leurs images~:
\begin{equation*}
\tlj_{\ell}(\tp_i^*)=\tjl_\ell(\tp_i),
\quad
i=1,2,
\end{equation*}
le sont.
Pour mettre fin à la preuve du théorème \ref{MAINTHEOREM},
il reste à vérifier que $\tp_1^*,\tp_2^*$ sont congruentes 
mod $\ell$ si et seulement si $\tp_1,\tp_2$ le sont, ce qui suit 
de \cite{MSb} Proposition A.3.

\section{Une formule de déterminant}
\label{S11}

Etant donnés une $\F$-algèbre à division centrale $\D$ de degré réduit $d$ 
et un corps algé\-bri\-que\-ment clos $\R$ de caractéristique différente de 
$p$, on rappelle que les notations $\XA(\D,\R)$ et $\RA(\D,\R)$ ont été 
définies en \eqref{vieillenot}, et on pose~:
\begin{equation*}
\RAT(\D,\R) = \prod\limits_{m\>0} \RA(\GL_m(\D),\R).
\end{equation*}
Il sera commo\-de d'utiliser le langage des séries formelles pour manipuler les
éléments de $\RAT(\D,\R)$.
Un élément $\Pi\in\RAT(\D,\R)$ sera noté~:
\begin{equation*}
\Pi = \sum\limits_{m\>0} \Pi_m\T^{md},
\quad
\Pi_m\in\RA(\GL_m(\D),\R)
\end{equation*}
où l'entier $md$ peut être considéré comme le ``degré déployé'' de $\Pi_m$.

\begin{defi}
\label{DEFSFS}
Si $\rho$ est une représentation irréductible cuspidale de $\GL_m(\D)$
à coefficients dans $\R$, on lui associe la série formelle~:
\begin{equation}
\label{DEFSFSRHO}
\SFS(\rho,\T) 
= \sum\limits_{t\>0} (-1)^t \Z(\rho,t) \T^{mdt} 
\end{equation}
dans $\RAT(\D,\R)$.
Plus généralement, pour tout $r\in\ZZ$ et tout entier $e\>1$, on pose~: 
\begin{equation*}
\SFS(\rho,e,r) = 
\sum\limits_{r+te\>0} (-1)^{r+te} \Z(\rho,r+te) \T^{md(r+te)},
\end{equation*}
la somme portant sur les entiers $t\in\ZZ$ tels que 
$r+te\>0$.
On a donc $\SFS(\rho,\T) = \SFS(\rho,1,0)$.
\end{defi}

\subsection{Preuve de la proposition \ref{FORUMEZD}}
\label{PSEG}

Supposons dans ce paragraphe que $\R$ est de caractéristique $\ell>0$.
La preuve de la proposition \ref{FORUMEZD} est inspirée de \cite{KL}. 
Fixons une $\R$-repré\-sen\-ta\-tion cuspidale $\rho$ de 
$\GL_m(\D)$ et posons~:
\begin{equation*}
e = 
\left\{
\begin{array}{ll}
\omega(\rho) & \text{si $\omega(\rho)>1$,} \\
\ell & \text{sinon.}
\end{array}
\right.
\end{equation*}
Etant donnés $a,b\in\ZZ$, il sera commode dans ce paragraphe 
d'introduire la notation~:
\begin{equation*}
\SS(a,b) = \SS_{\rho}(a,b) = \left\{
\begin{array}{rcll}
\Z(\rho\nu_\rho^a,b-a+1) & \in & \XA(\GL_{m(b-a+1)}(\D),\R) 
& \text{si $a\<b-1$,} \\
0 & \in & \RA(\D,\R) & \text{sinon,}
\end{array}
\right.
\end{equation*} 
et de définir la série formelle~: 
\begin{equation}
\label{NOTSFZ}
\SFZ(a,b) = \SFZ_{\rho}(a,b) = \sum\limits_{r\in\ZZ} 
(-1)^{re+b-a+1} \SS(a,b+re) \T^{md(re+b-a+1)}
\in \RAT(\D,\R)
\end{equation}
qui ne dépend que des classes de $a$ et $b$ modulo $e$, compte tenu de 
la définition de $e$. 
La représen\-tation $\rho$ et l'algèbre à division $\D$ 
étant fixées dans toute cette section,
il est commode de poser $\Y=-\T^{md}$ et d'écrire~: 
\begin{equation*}
\SFZ(a,b) = \sum\limits_{r\in\ZZ} \SS(a,b+re) \Y^{re+b-a+1}
\end{equation*}
pour alléger les notations.
Remarquons que, en termes de la nota\-tion introduite dans la définition 
\ref{DEFSFS}, on a simplement~:
\begin{equation}
\label{plustardvousverrez}
\SFZ(a,b) = \SFS(\rho\nu_\rho^a,e,b-a+1)
\end{equation}
mais la notation \eqref{NOTSFZ} sera plus com\-mode dans cette section. 

On note encore $\com$ la comultiplication de $\RAT=\RAT(\D,\R)$ 
définie à partir de \eqref{VentreDieu2}.

\begin{lemm}
On a~: 
\begin{equation*}
\label{Jaquetsegmentcomplete} 
\com(\SFZ(a,b)) = \sum\limits_{i\in \ZZ/e \ZZ} \SFZ(a,i)\otimes\SFZ(i+1,b). 
\end{equation*}
\end{lemm}

\begin{proof}
Compte tenu de (\ref{P81}.1), on a~:
\begin{eqnarray*}
\com(\SFZ(a,b)) & = & \sum\limits_{r\in\ZZ} \com(\SS(a,b+re)) \Y^{b+re-a+1} \\
&=& \sum\limits_{r\in\ZZ} \sum\limits_{k\in\ZZ} \SS(a,k) \Y^{k-a+1} \otimes 
    \SS(k+1,b+re) \Y^{b+re-k} \\
&=& \sum\limits_{k\in\ZZ} \SS(a,k) \Y^{k-a+1} \otimes \SFZ(k+1,b).
\end{eqnarray*}
Effectuant la division euclidienne de $k$ par $e$, on obtient~:
\begin{equation*}
\com(\SFZ(a,b)) = \sum_{i=0}^{e-1} \sum\limits_{t\in\ZZ} 
\SS(a,i+te) \Y^{i+te-a+1} \otimes \SFZ(i+1,b) 
\end{equation*}
ce qui donne le résultat annoncé. 
\end{proof}

Soient maintenant des entiers 
$a_1,b_1,\dots,a_r,b_r\in\ZZ$ avec $r\>1$.
On pose~:
\begin{equation*}
\SFD(a_1,\dots,a_r,b_1,\dots,b_r) = \det (\SFZ(a_i,b_j)) \in \RAT
\end{equation*}
qui ne dépend que des $a_i$ et des $b_j$ modulo $e$, 
mais qui dépend de l'ordre des termes. 

\begin{lemm}
Soient des entiers $a_1,b_1,\dots,a_r,b_r\in\ZZ$.
\begin{enumerate}
\item
Supposons qu'il existe $i\neq j$ tels que $a_i\equiv a_j$ mod $e$
ou $b_i\equiv b_j$ mod $e$.
Alors~:
\begin{equation*}
\SFD(a_1,\dots,a_r,b_1,\dots,b_r)=0.
\end{equation*}
\item
Soient $\s,\phi$ des permutations de $\{1,\dots,e\}$. 
Alors~:
\begin{equation*}
\SFD(a_{\s(1)},\dots,a_{\s(r)},b_{\phi(1)},\dots,b_{\phi(r)}) = 
\epsilon(\s)\epsilon(\phi)\cdot\SFD(a_1,\dots,a_r,b_1,\dots,b_r)
\end{equation*}
où $\epsilon(\s)$ désigne la signature de la permutation $\s$.
\end{enumerate}
\end{lemm}

\begin{prop}
Posons $\SFD=\SFD(a_1,\dots,a_r,b_1,\dots,b_r)$.
On a~:
\begin{equation*}
\label{Jaquetdeterminant} 
\com(\SFD)=\sum\limits_{1\<c_1 < c_2 < \dots < c_r \< e}
\SFD(a_1,\dots,a_r,c_1,\dots,c_r)\otimes
\SFD(c_1+1,\dots,c_r+1,b_1,\dots,b_r). 
\end{equation*}
\end{prop}

\begin{proof}
De l'égalité~:
\begin{equation*}
\SFD = \sum\limits_{\s} \epsilon(\s) \cdot \prod\limits_{i=1}^{e} 
\SFZ(a_i,b_{\s(i)})
\end{equation*}
on déduit~:
\begin{equation*}
\com(\SFD) 
= \sum\limits_{\s} \epsilon(\s) 
\sum\limits_{k_1,\dots,k_r}
\left(\prod\limits_{i=1}^{e} 
\SFZ(a_i,k_i) \otimes \prod\limits_{i=1}^{e} 
\SFZ(k_i+1,b_{\s(i)}) \right) 
\end{equation*}
où $k_1,\dots,k_r$ décrivent l'ensemble $\ZZ/e\ZZ$. 
Ceci est égal à~:
\begin{equation*}
\sum\limits_{k_1,\dots,k_r} \prod\limits_{i=1}^{e} 
\SFZ(a_i,k_i) \otimes \left( \sum\limits_{\s} \epsilon(\s) \cdot 
\prod\limits_{i=1}^{e} \SFZ(k_i+1,b_{\s(i)}) \right)
\end{equation*}
et le terme entre parenthèses est égal au déterminant 
$\SFD(k_1+1,\dots,k_r+1,b_{1},\dots,b_{r})$.
Celui-ci étant nul dès lors qu'il existe $i\neq j$ tels que 
$k_i=k_j$, on peut écrire~:
\begin{eqnarray*}
\com(\SFD) 
& = & \sum\limits_{1\<c_1<c_2<\dots<c_r\<e} 
\sum\limits_{\s} \prod\limits_{i=1}^{e} \SFZ(a_i,c_{\s(i)}) \otimes 
\SFD(c_{\s(1)}+1,\dots,c_{\s(r)}+1,b_1,\dots,b_r) \\
& = & \sum\limits_{1\<c_1<c_2<\dots<c_r\<e} 
\sum\limits_{\s} \epsilon(\s)
\prod\limits_{i=1}^{e} \SFZ(a_i,c_{\s(i)}) \otimes 
\SFD(c_{1}+1,\dots,c_{r}+1,b_1,\dots,b_r) \\
& = & \sum^{\phantom{e}}\limits_{1\<c_1<c_2<\dots<c_r\<e} 
\SFD(a_1,\dots,a_r,c_1,\dots,c_r) \otimes 
\SFD(c_1+1,\dots,c_r+1,b_1,\dots,b_r)
\end{eqnarray*}
comme annoncé.
\end{proof}

En particulier, si $r=e$, on a~:
\begin{equation*}
\com(\SFD(a_1,\dots,a_e,b_1,\dots,b_e)) 
= \SFD(a_1,\dots,a_e,1,\dots,e)
\otimes \SFD(2,\dots,e+1,b_1,\dots,b_e).
\end{equation*}
Pour que ceci ne soit pas nul, il faut que les $a_i$ mod $e$ 
soient tous distincts, donc définissent une permutation de $\ZZ/e\ZZ$,
notée $a$, et de même pour les $b_j$ mod $e$, qui forment une permutation 
$b$.
Supposant que ce soit le cas, et notant~:
\begin{equation}
\label{DEFNOTD}
\SFDD=
\SFDD(\rho,\T)=
\SFD(2,\dots,e+1,1,\dots,e)=\det (\SFZ(i+1,j)))
\end{equation} 
on trouve
$\SFD(a_1,\dots,a_e,b_1,\dots,b_e) = (-1)^{e-1}\epsilon(a)\epsilon(b)\cdot\SFDD$
et en particulier~:
\begin{equation}
\label{comDelta}
\com(\SFDD) = \SFD(2,\dots,e+1,1,\dots,e)
\otimes\SFD(2,\dots,e+1,1,\dots,e) 
= \SFDD\otimes\SFDD.
\end{equation}
Pour tout $r\>1$, notons $\D_r=\D(\rho,r)$ 
le coefficient de $\Y^{er}$ dans $\SFDD$, de sorte que~:
\begin{equation*}
\SFDD = \sum\limits_{r\>0} \D_r \Y^{er}.
\end{equation*}
De la formule \eqref{comDelta} on déduit le corollaire suivant.

\begin{coro}
\label{coro:jacquetpi}
Soit un entier $k\in\{0,\dots,mer\}$.
On a~:
\begin{equation*}
\rp_{(k,mer-k)}(\D_r) = \left\{
\begin{array}{ll}
0 & \text{si $k$ n'est pas multiple de $me$,} \\
\D_{s}\otimes\D_{r-s}
& \text{si $k=mes$ avec $s\in\{0,\dots,r\}$}.
\end{array}
\right.
\end{equation*}
\end{coro}

\begin{rema}
Développant le déterminant, on obtient la formule 
suivante~:
\begin{equation}
\label{formulex}
\D_r = \sum\limits_{(r_0,\dots,r_{e-1})} 
(-1)^{r }{\epsilon(r_0,\dots,r_{e-1}) }
\cdot \Z(\rho,r_0)\times\dots\times\Z(\rho\nu_\rho^{e-1},r_{e-1}) 
\end{equation}
où $(r_0,\dots,r_{e-1})$ décrit les familles d'entiers $\>0$ 
de somme $re$ et 
telles que $i\mapsto i+r_i-r$ mod $e$ 
soit une permutation de $\ZZ/e\ZZ$, de signature notée 
$\epsilon(r_0,\dots,r_{e-1})$.
\end{rema}

Posons maintenant $\rho_0=\Sp(\rho,e)$, qui est cuspidale de degré $me$
(voir le \S\ref{P32}).

\begin{prop}
\label{formuleZrho0}
\label{formulesynthetique}
Pour tout $r\>1$, on a $\D_r=(-1)^{r}\cdot \Z(\rho_0,r)$.
\end{prop}

\begin{proof}
On prouve la proposition par récurrence sur l'entier $r\geq 1$. 
D'après la formule \eqref{formulex}, la quantité 
$\D_1$ est une combinaison linéaire de termes de la forme~: 
\begin{equation*}
\Z(\rho,r_0)\times\dots\times\Z(\rho\nu_\rho^{e-1},r_{e-1}),
\quad r_0,\dots,r_{e-1}\>0,
\end{equation*}
avec $r_0+\dots+r_{e-1}=e$.
D'après \cite{MSc} Corollaire 8.5, 
le seul terme irréductible résiduellement non dégénéré dans $\D_1$ 
apparaît pour $r_0=\dots=r_{e-1}=1$.
Il est égal à $\rho_0$ et apparaît dans $\D_1$ avec le coefficient 
$-1$. 
Ecrivons~:
\begin{equation*}
\label{resdif}
\D_1+\rho_0 = a_1\pi_1+\dots+a_r\pi_r
\end{equation*}
où $\pi_i$ est une représentation irréductible 
résiduellement dégénérée de la forme $\Z(\upmu_i\boxtimes\rho)$ 
avec $\upmu_i$ un multisegment formel convenable, et où $a_i\in\ZZ$. 
D'après le corollaire \ref{coro:jacquetpi}, on a $\rp_{(t,me-t)}(\D_1)=0$ 
pour tout $1\<t\<me-1$, et on a un résultat analogue pour $\rho_0$
qui est cuspidale. 
On déduit du lemme \ref{lemm:cle} que $\D_1+\rho_0=0$, c'est-à-dire 
que $\D_1=-\rho_0$. 

On suppose maintenant la proposition prouvée pour tout $i\<r-1$. 
Si $r\>2$, la quantité $\D_r$ ne contient aucun terme irréductible 
résiduellement non dégénéré. 
Ecrivons~:
\begin{equation*}
\label{resdif}
\D_r - (-1)^{r}\cdot \Z(\rho_0,r) = a_1\pi_1+\dots+a_r\pi_r
\end{equation*}
avec les $a_i$ et les $\pi_i$ comme plus haut. 
D'après le corollaire \ref{coro:jacquetpi}, 
pour $s\in\{1,\dots,r-1\}$, on a~:
\begin{equation*}
\rp_{me\cdot(s,r-s)}(\D_r) = \D_{s}\otimes\D_{r-s}.
\end{equation*}
Par hypothèse de récur\-ren\-ce, on a~:
\begin{eqnarray*}
\D_{s}\otimes\D_{r-s} 
&=& (-1)^{s}(-1)^{r-s}\cdot 
\Z(\rho_0,s)\otimes\Z(\rho_0,r-s) \\
&=& (-1)^{r}\cdot\rp_{me\cdot(s,r-s)}(\Z(\rho_0,r))
\end{eqnarray*}
ce dont on déduit le résultat, à nouveau grâce au lemme \ref{lemm:cle}. 
\end{proof}

\begin{rema}
\label{noubliespas}
La formule de la proposition \ref{formuleZrho0} peut être résumée par 
l'identité~:
\begin{equation*}
\SFDD(\rho,\T) = \SFS(\rho_0,\T)
\end{equation*}
entre séries formelles (voir la définition \ref{DEFSFS} et 
la proposition \ref{FORUMEZD}). 
\end{rema}

\subsection{}
\label{P112}

Supposons dans ce paragraphe que $e$ est égal à $\ell$. 
Nous allons voir que le déterminant $\SFDD$ prend dans ce cas une forme 
particulière. 
En effet on a $\omega(\rho)=1$, \ie que 
$\rho\nu_\rho$ est isomorphe à $\rho$.
La série $\SFZ(a,b)$ ne dépend donc que de $b-a$ mod $e$.
Posant $\SFA_i=\SFZ(0,i)$ pour $i\in\ZZ/\ell\ZZ$, on a donc~:
\begin{equation*}
\SFDD = \det
\begin{pmatrix}
\SFA_1 & \SFA_2 & \dots & \SFA_{\ell} \\
\SFA_{\ell} & \ddots & \ddots & \vdots \\
\vdots & \ddots & \ddots & \SFA_2 \\
\SFA_{2} & \dots & \SFA_{\ell} & \SFA_1
\end{pmatrix}.
\end{equation*}
La proposition suivante montre que $\SFDD$ se factorise remarquablement 
dans l'anneau $\RAT\otimes_\ZZ\CC$.

\begin{prop}
\label{Flyte}
Supposons que $e=\ell$.
Fixons une racine de l'unité $\xi\in\mult\CC$ d'ordre $\ell$, 
et posons~:
\begin{equation*}
\SFS^{(i)} = \SFA_0 + \xi^{i}\SFA_1 + \dots + 
\xi^{i(\ell-1)}\SFA_{\ell-1} 
= \sum\limits_{r\>0} \xi^{ir} \Z(\rho,r) \Y^r
\end{equation*} 
pour tout $i\in\ZZ/\ell\ZZ$.
Alors on a~:
\begin{equation*}
\SFDD = \SFS^{(1)}\SFS^{(2)}\dots\SFS^{(\ell)}
\end{equation*}
dans $\RAT\otimes_\ZZ\CC$.
\end{prop}

\begin{proof}
Commençons par énoncer le résultat général suivant. 

\begin{lemm}
\label{lemmedetsym}
Soit un entier $n\>1$, et soit $\C$ un anneau commutatif 
possédant une racine de l'uni\-té $\xi\in\mult\C$ d'ordre $n$.
Pour tous $a_1,\dots,a_n\in\C$, on a~:
\begin{equation*}
\det
\begin{pmatrix}
a_1 & a_2 & \dots & a_{n} \\
a_{n} & \ddots & \ddots & \vdots \\
\vdots & \ddots & \ddots & a_2 \\
a_{2} & \dots & a_{n} & a_1
\end{pmatrix} 
= \prod\limits_{i=1}^{n} \left(\sum\limits_{j=1}^n \xi^{ij}a_j\right).
\end{equation*}
\end{lemm}

\begin{proof}
Notons $\Omega\in\Mat_n(\C)$ 
la matrice corres\-pon\-dant à la 
permutation $i\mapsto i-1$ mod $n$.
Elle est diagonalisable sur $\C$, de valeurs propres 
$1,\xi,\dots,\xi^{n-1}$.
La matrice 
$a_1\Omega+a_2\Omega^2+\dots+a_n\Omega^n$ 
l'est donc également et ses valeurs propres sont les~:
\begin{equation*}
u_i = \sum\limits_{j=1}^n \xi^{ij}a_j,
\quad 
i\in\{1,\dots,n\}.
\end{equation*}
Le résultat s'ensuit.
\end{proof}

\begin{rema}
Si $n$ est inver\-si\-ble dans $\C$,
l'application de $\C[\Omega]$ dans $\C^n$ définie par~: 
\begin{equation*}
\sum\limits_{i=1}^n a_i\Omega^i = 
\begin{pmatrix}
a_1 & a_2 & \dots & a_{n} \\
a_{n} & \ddots & \ddots & \vdots \\
\vdots & \ddots & \ddots & a_2 \\
a_{2} & \dots & a_{n} & a_1
\end{pmatrix}
\mapsto (u_1,\dots,u_n),
\quad
u_i = \sum\limits_{j=1}^n \xi^{ij}a_j
\end{equation*}
est un isomorphisme de $\C$-algèbres.
L'isomorphisme réciproque est donné par~: 
\begin{equation*}
(u_1,\dots,u_n) \mapsto 
\begin{pmatrix}
a_1 & a_2 & \dots & a_{n} \\
a_{n} & \ddots & \ddots & \vdots \\
\vdots & \ddots & \ddots & a_2 \\
a_{2} & \dots & a_{n} & a_1
\end{pmatrix},
\quad
a_i = \frac{1}{n}\cdot \sum\limits_{j=1}^n \xi^{-ij}u_j.
\end{equation*}
\end{rema}

Pour prouver la proposition \ref{Flyte}, partons de l'égalité~:
\begin{equation}
\label{factoSFDD1}
\SFDD = 
\prod\limits_{j=1}^{\ell}
\left(\sum\limits_{i=1}^{\ell} \xi^{ij} \SFA_i\right) 
\end{equation}
donnée par le lemme \ref{lemmedetsym} et remarquons que~:
\begin{equation*}
\SFA_i = \frac{1}{\ell}\cdot\sum\limits_{k=1}^{\ell} \xi^{-ik} \SFS^{(k)}.
\end{equation*}
Remplaçant dans \eqref{factoSFDD1}, on obtient~:
\begin{equation*}
\SFDD = 
\prod\limits_{j=1}^{\ell}\left(
\sum\limits_{k=1}^{\ell} \left(\frac{1}{\ell}\cdot 
\sum\limits_{i=1}^{\ell} \xi^{i(j-k)} \right) 
\SFS^{(k)}\right) 
\end{equation*}
ce qui donne le résultat annoncé, puisque la somme intérieure
vaut $0$ si $k\neq j$ et $1$ sinon.
\end{proof}

\begin{rema}
\label{FlyteRem}
Il sera utile de présenter la formu\-le de la proposition \ref{Flyte} sous une 
forme légèrement différente. 
Si l'on fixe $\omega\in\mult\CC$ telle que
$\omega^{md}$ soit une racine de l'unité d'ordre $\ell$, cette formu\-le devient~:
\begin{equation*}
\SFDD (\T) = \SFS(\T)\SFS(\omega\T)\dots\SFS(\omega^{\ell-1}\T)
\end{equation*}
indépendamment du choix de $\omega$.
\end{rema}

\begin{coro}
\label{ChouFleur}
Pour tout entier $v\>0$, on pose~:
\begin{equation*}
\SFS_v (\T) 
= \SFS(\Sp(\rho,\ell^v),\T)
= \sum\limits_{r\>0} \Z(\Sp(\rho,\ell^v),r) (-\T^{md\ell^v})^r
\end{equation*}
qui est la série associée par \eqref{DEFSFSRHO} 
à la représen\-ta\-tion cuspidale $\Sp(\rho,\ell^v)$,
et on fixe une racine de l'unité $\omega\in\mult\CC$ telle que 
$\omega^{md}$ soit d'ordre $\ell^v$.
Pour tout $v\>0$ on a l'égalité~:
\begin{equation*}
\SFS_{v} (\T) = \SFS (\T)\SFS (\omega\T)\dots\SFS (\omega^{\ell^v-1}\T)
\end{equation*}
dans l'anneau $\RAT\otimes_\ZZ\CC$.
\end{coro}

\begin{proof}
On suppose que $v\>1$. 
Notons $\pi$ la représen\-tation $\Sp(\rho,\ell^v)$
et notons $\k$ la représentation $\Sp(\rho,\ell^{v-1})$.
D'après la remarque \ref{FlyteRem}, on a~:
\begin{equation*}
\SFS_{v} (\T) = \SFS_{v-1}(\T)\SFS_{v-1}(\a\T)\dots\SFS_{v-1}(\a^{\ell-1}\T)
\end{equation*}
car $\pi$ est égale à $\Sp(\k,\ell)$,
où $\a$ est une racine de l'unité telle que $\a^{md\ell^{v-1}}$ soit 
d'ordre $\ell$.
On~peut donc supposer que $\a=\omega$.
Raisonnant par récurrence sur $v$, on a~:
\begin{equation*}
\SFS_{v-1} (\T) = \SFS (\T)\SFS (\mu\T)\dots\SFS (\mu^{\ell^{v-1}-1}\T)
\end{equation*}
où $\mu$ est une racine de l'unité telle que $\mu^{md}$ soit d'ordre $\ell^{v-1}$. 
On~peut donc supposer que $\mu=\omega^\ell$.
Ainsi $\SFS_v(\T)$ est le produit des $\SFS(\omega^{i+\ell j}\T)$ pour 
$i\in\{1,\dots,\ell\}$ et $j\in\{1,\dots,\ell^{v-1}\}$, 
ce qui donne la formule attendue. 
\end{proof}

\section{Le morphisme de Langlands-Jacquet modulo $\ell$}
\label{S12}

Fixons un nombre premier $\ell\neq p$
et une $\F$-algèbre à division centrale $\D$ de degré réduit $d$. 

Dans cette section, nous prouvons le théorème \ref{theoLJmodlcorpus}. 

\subsection{}
\label{BaduTLJ}

La $\ZZ$-algèbre commutative 
$\RA(\D,\qlb)$ est librement engendrée par l'ensemble~:
\begin{equation*}
\Dd(\D,\qlb) = \coprod\limits_{m\>1} \Dd(\GL_m(\D),\qlb).
\end{equation*}
Etant donné un $m\>1$, on a la correspondance de Jacquet-Langlands 
$\ell$-adique $\widetilde{\boldsymbol{\pi}}_\ell$ introduite au 
paragraphe \ref{introtpil}.
Fai\-sant varier $m\>1$, on obtient une application injective~:
\begin{equation}
\label{Terrier}
\widetilde{\boldsymbol{\pi}}_\ell : \Dd(\D,\qlb) \to \Dd(\F,\qlb)
\end{equation}
(encore notée $\widetilde{\boldsymbol{\pi}}_\ell$)
dont l'image est formée des représentations de degré divisible par $d$.

\begin{defi}[Badulescu \cite{BaduJIMJ} \S3.1]
Le \textit{morphisme de Langlands-Jacquet $\ell$-adique} est 
l'unique morphis\-me d'anneaux~:
\begin{equation*}
\TLJ_\ell : \RA(\F,\qlb) \to \RA(\D,\qlb)
\end{equation*}
coïncidant avec la réciproque de 
$\widetilde{\boldsymbol{\pi}}_\ell$ sur les représenta\-tions de 
$\Dd(\F,\qlb)$ 
de degré divisible par $d$, et prenant la valeur $0$ sur les autres. 
\end{defi}

\begin{prop}
\label{reduit}
Soit un entier $n\>1$ et soit $\rt$ une représentation
irréductible cuspidale $\ell$-adique de $\GL_n(\F)$.
\begin{enumerate}
\item
Etant donné un entier $r\>1$, pour que l'image de $\Z(\rt,r)$ par $\TLJ_\ell$ 
soit non nulle, il faut et il suffit que $r$ soit un multiple de l'entier~:
\begin{equation*}
s = \frac{d}{(d,n)}.
\end{equation*}
\item
Il y a un unique entier $m\>1$ et une unique représentation irréductible 
cuspidale $\ell$-adique $\rt'$ de $\GL_m(\D)$ tels que~:
\begin{equation}
\label{formuleLJZ}
\TLJ_\ell\Big(\Z(\rt,r)\Big) = (-1)^{r-r'}\cdot\Z(\rt',r')
\end{equation}
pour tout $r=r's$, avec $r'\>1$.
\item
On a $s=s(\rt')$ et le degré $m$ de la représentation $\rt'$ vérifie la 
relation $md=ns$.
\end{enumerate}
\end{prop}

\begin{proof}
Soit un entier $r\>1$.
Appliquant l'involution de Zelevinski à la représentation de Speh $\Z(\rt,r)$ 
(voir la formule \eqref{BMX} et \cite{BaduJIMJ} Théorème 3.16), 
l'image de $\Z(\rt,r)$ 
est non nulle si et seulement si $d$ divise $nr$, 
\ie si et seulement si $r$ est un multiple de $s$, auquel cas elle prend la 
forme \eqref{formuleLJZ} pour une unique représentation irréductible 
cuspidale $\ell$-adique $\rt'$.
Si l'on pose en particulier $r=s$, on trouve que $ns=md$.

Le fait que $s=s(\rt')$ provient de l'invariance du degré 
paramé\-trique \cite{BHJL3} 2.8 Corollary 1, et du fait 
que le degré paramé\-trique de $\rt'$ est par définition
égal au quotient de $md$ par $s(\rt')$.
\end{proof}

A l'aide de la notation introduite dans la définition \ref{DEFSFS}, 
la formule \eqref{formuleLJZ} est résumée par~:
\begin{equation}
\label{TLJSFS}
\TLJ_{\ell}\Big(\SFS(\rt,\T)\Big) = \SFS(\rt',\T).
\end{equation}
Nous aurons besoin d'une version un peu plus générale de ce résultat. 
On utilise à nouveau les notations introduites dans la définition 
\ref{DEFSFS}. 

\begin{lemm}
\label{quentin}
Soit $\rt$ une représentation irréductible cuspidale $\ell$-adique de 
$\GL_n(\F)$, $n\>1$, et soit des entiers $e\>1$ et $r\in\ZZ$.
On pose $e'=e(e,s)^{-1}$ et $s'=s(e,s)^{-1}$.
\begin{enumerate}
\item
Si $r$ n'est pas divisible par $(e,s)$, 
alors l'image de $\SFS_{\rt}(e,r)$ par $\TLJ_\ell$ est nulle.
\item
Si $r$ est divisible par $(e,s)$, 
alors~:
\begin{equation*}
\TLJ_\ell\Big(\SFS(\rt,e,r)\Big) = \SFS(\rt',e',r')
\end{equation*}
où $r'$ est l'unique entier compris entre $0$ et $e'-1$ tel que 
$sr'$ soit congru à $r$ mod $e$.
\end{enumerate}
\end{lemm}

\begin{proof}
La représentation de Speh $\Z(\rt,r+te)$
se transfère en un élément non nul de $\RAT(\D,\qlb)$
si et seule\-ment si $r+te$ est di\-vi\-si\-ble par $s$.
Pour que ce soit le cas, il faut que l'entier~$r$ ap\-par\-tienne à 
$e\ZZ+s\ZZ$.
Supposons que c'est le cas, et écrivons $r+t_0e=sr'$ 
avec $t_0,r'\in\ZZ$.
On peut supposer que $r'\in\{0,\dots,e'-1\}$.
Alors $r+te$ est divisible par $s$ si et seulement si $t-t_0$ 
est divisible par $s'$, \ie que $(t-t_0)e=hse'$ avec $h\in\ZZ$.
On obtient~:
\begin{equation*}
\TLJ_{\ell}(\Z(\rt,r+te)) = (-1)^{l(s-1)}\cdot\Z(\rt',r'+he')
\end{equation*}
avec $l\in\ZZ$ défini par $ls=r+te$, ce qui donne le résultat annoncé. 
\end{proof}

La $\ZZ$-algèbre commutative
$\RA(\D,\flb)$ est librement engendrée par l'ensem\-ble~:
\begin{equation*}
\Zz_{1}(\D,\flb) = 
\coprod\limits_{m\>1} \Zz_{1}(\GL_m(\D),\flb)
\end{equation*}
des représentations super-Speh $\ell$-modulaires.
Etant donnés un $m\>1$ et une $\F$-algèbre à division centrale $\A$ 
de degré réduit $md$, on a des bijections~: 
\begin{equation}
\label{compobijmodl}
\Zz_{1}(\GL_m(\D),\flb) \to \XA(\A,\flb) \to \Zz_{1}(\GL_{md}(\F),\flb)
\end{equation}
la première étant donnée par le corollaire \ref{superSpeh} et la seconde étant 
la réciproque de celle donnée par le corollaire \ref{superSpeh} appliqué à la 
forme déployée $\GL_{md}(\F)$.
Si $\Z(\rho,r)$ est une $\flb$-représentation super-Speh de $\GL_{md}(\F)$ 
et si $\Z(\rho',r')$ est son image réciproque dans
$\Zz_{1}(\GL_{m}(\D),\flb)$ par 
\eqref{compobijmodl}, 
on définit une application injective par~: 
\begin{eqnarray}
\notag
\Zz_1(\GL_{md}(\F),\flb) &\to& 
\RA(\GL_{m}(\D),\flb) \\
\label{2ndeEq}
\Z(\rho,r) & \mapsto & (-1)^{r-r'}\cdot\Z(\rho',r').
\end{eqnarray}
Le \textit{morphisme de Langlands-Jacquet mod $\ell$}
est l'unique morphis\-me d'an\-neaux~:
\begin{equation*}
\LJ_\ell : \RA(\F,\flb) \to \RA(\D,\flb)
\end{equation*}
coïncidant avec \eqref{2ndeEq} pour tout $m\>1$, 
et s'annulant 
en toute représentation super-Speh dont le degré n'est pas divisible par $d$.
Le résultat suivant précise le théorème \ref{caciqueintro}.

\begin{theo}
\label{theoLJmodlcorpus}
Le morphisme $\LJ_\ell$ est le seul morphisme d'anneaux rendant commutatif 
le diagramme~:
\begin{equation*}
\label{LJTlmorphisme2corpus}
\begin{CD}
\RA(\F,\qlb)^{{\rm e}} @>{\TLJ_\ell}>>
\RA(\D,\qlb)^{{\rm e}}\\
@V{\r_\ell}VV @VV {\r_\ell} V \\
\RA(\F,\flb)^{\phantom {\rm e}} @>>{\LJ_\ell}> 
\RA(\D,\flb)^{\phantom {\rm e}} \\
\end{CD}
\end{equation*}
où $\Gg(\G,\qlb)^{{\rm e}}$ désigne le sous-groupe de $\Gg(\G,\qlb)$ engendré
par les représenta\-tions irréductibles entières. 
\end{theo}

L'unicité provient du fait que $\r_\ell$ est surjective. 
Pour prouver que le diagramme est commuta\-tif, 
il suffit de prouver l'égalité~:
\begin{equation}
\label{resteafaire}
\Big(\LJ_\ell\circ\r_\ell\Big)(\tp) = \Big(\r_\ell\circ\TLJ_\ell\Big)(\tp)
\end{equation}
pour toute représentation de Speh $\ell$-adique entière $\tp\in\Zz(\F,\qlb)$. 
Si $\tp$ est $\ell$-super-Speh, \ie si $w(\tp)=1$, 
l'égalité recherchée est une 
conséquence immédiate de la définition de $\LJ_\ell$,
en vertu du théorème \ref{RogerCarbury}.
Supposons maintenant que $\tp$ n'est pas $\ell$-super-Speh
et écrivons-la sous la forme $\Z(\rt,r)$ où $\rt$ est 
une $\qlb$-représentation
irréductible cuspidale entière telle que $w(\rt)>1$.
Fixant $\rt$, nous allons prouver \eqref{resteafaire} pour tous les 
$r\>1$ en même temps. 
En d'autres termes, nous allons prouver l'égalité~:
\begin{equation}
\label{resteafaireserie}
\Big(\LJ_\ell\circ\r_\ell\Big)
\left(\SFS(\rt,\T) \right) = 
\Big(\r_\ell\circ\TLJ_\ell\Big) 
\left(\SFS(\rt,\T) \right) 
\end{equation}
pour toute représentation irréductible cuspidale $\ell$-adique entière $\rt$
telle que $w(\rt)>1$.

\subsection{Preuve du théorème \ref{theoLJmodlcorpus}}

Fixons 
une $\qlb$-représentation irréductible cuspidale entière $\rt$ 
de $\GL_n(\F)$ pour $n\>1$.
On pose $w=w(\rt)$ et on suppose que $w>1$.
On note $\rho$ la réduction mod $\ell$ de $\rt$.
La réduction mod $\ell$~de $\Z(\rt,r)$ est donc égale à $\Z(\rho,r)$ pour tout 
$r\>1$, ce que résumé l'identité
$\r_\ell\big(\SFS(\rt,\T)\big)=\SFS(\rho,\T)$.
Pour prouver le théorème \ref{theoLJmodlcorpus}, il s'agit de montrer que les 
séries~: 
\begin{eqnarray*}
\SFU &=&\r_\ell\circ\TLJ_\ell\Big(\SFS(\rt,\T)\Big), \\
\SFL &=& \LJ_\ell\Big(\SFS(\rho,\T)\Big)
\end{eqnarray*}
sont égales.

Ecrivons $\rho$ sous la forme $\Sp(\a,w)$ avec $\a$ 
supercuspidale, et écrivons $w$ sous la forme $\omega(\a)\ell^v$, 
où $v\>0$ est la valuation $\ell$-adique de $w$.
Pour alléger les notations, on pose $e=\omega(\a)$ et on note 
$\tau$ la représentation cuspidale $\Sp(\a,e)$,
de sorte que $\rho$ est égale à $\Sp(\tau,\ell^v)$.

Fixons un relèvement $\ell$-adique $\at$ de $\a$ 
(dont l'existence est assurée par \cite{Vigb} Paragraphe 5.10
et \cite{MSc} Théorème 6.11) 
et notons $\at'$ la représentation 
cuspidale dans $\XA(\D,\qlb)$ qui lui est associée par la 
proposition \ref{reduit}. 
D'après le théorème \ref{RogerCarbury}, 
la réduction mod $\ell$ de $\at'$ est irréductible car $\at'$ 
est $\ell$-supercuspidale~; on la note $\a'$. 
On pose $s_0=s(\at')$.
Notons qu'on a aussi $s_0=s(\a')$.

Soit $\rt'$ la représentation irréductible cuspidale entière de 
$\XA(\D,\qlb)$ associée à $\rt$ par la proposition \ref{reduit}. 
D'après~le théo\-rème \ref{RogerCarbury}, on a $w(\rt')=w$.
On pose $s=s(\rt')$ et on note $m$ le degré de $\rt'$, 
de sorte que $md=ns$. 
On pose $k=(w,m)$ et $a=wk^{-1}$.
D'après le lemme \ref{CN}, 
pour tout fac\-teur irréductible $\rho'$ de la réduction mod $\ell$ de $\rt'$, 
on a $a=a(\rt')$ et $k=k(\rho')$.

\begin{lemm}
La représentation $\Sp(\a',k)$ est 
un fac\-teur irréductible de la réduction mod $\ell$ de $\rt'$, et on a $s_0=as$.
\end{lemm}

\begin{proof}
Fixons une $\F$-algèbre à division centrale de 
degré réduit $ns$, et notons $\A$ son groupe multiplicatif.
Notons $\tp$ la représentation de Speh 
$\Z(\rt,s)$ et notons $\pi$ sa réduction mod $\ell$, 
qui est donc égale à $\Z(\rho,s)$.
Dans la base des représentations super-Speh, on écrit~:
\begin{equation*}
\pi = \sum\limits_{i=1}^{\omega(\a)} \textsf{n}_i \cdot \Z(\a\nu^i,ws) + \delta 
\end{equation*}
où $\delta$ est une combinaison linéaire d'induites de représentations 
super-Speh,
\ie dans le noyau du morphisme $\JL_\ell$ défini par la proposition 
\ref{Leighton}, 
et où les $\textsf{n}_i$ sont des entiers.
Appliquant $\JL_{\ell}$, on trouve~:
\begin{eqnarray*}
\notag
\JL_{\ell}(\pi) 
&=& \sum\limits_{i=1}^{\omega(\a)} \textsf{n}_i \cdot \JL_{\ell}(\Z(\a\nu^i,ws)) \\
\label{bla2}
&=& \sum\limits_{i=1}^{\omega(\a)} \textsf{n}_i \cdot (-1)^{ws-1} \cdot \a_\A\nu^i
\end{eqnarray*}
où $\a_\A=\boldsymbol{\jmath}{}_\ell^{\boldsymbol{*}}(\Z(\a,ws))$ 
est irréductible car $\at$ est $\ell$-supercuspi\-dale.
D'après la proposition \ref{Leighton}, on sait par ailleurs que 
$\JL_{\ell}(\pi)$ est égal à la réduction mod $\ell$ de 
$(-1)^{s-1}\cdot\tp_\A$ où $\tp_\A$ est le transfert de $\tp$ à $\A$. 
Comme $w(\tp_\A)=w$, on trouve~:
\begin{equation*}
\JL_{\ell}(\pi) 
= (-1)^{s-1}\cdot\r_\ell(\tp_\A) 
= (-1)^{s-1}\cdot\sum\limits_{i=1}^{w} \pi_\A\nu^i
\end{equation*}
où $\pi_\A$ est un facteur irréductible de la réduction mod $\ell$ de 
$\tp_\A$. 
On peut donc supposer que $\a_\A$ est égale à $\pi_\A$.

Fixons maintenant un facteur irréductible $\rho'$ de la réduction mod 
$\ell$ de $\rt'$, qu'on écrit sous la forme $\Sp(\b,k)$ avec $\b$ 
supercuspidale.
Faisant avec $\rt'$ ce qu'on vient de faire avec $\tp$, on déduit que, 
quitte à tordre $\rho'$ par une puissance de $\nu$, 
on peut supposer que le transfert de $\Z(\b,k)$ à $\A$
est égal à $\pi_\A$. 

Les représentations super-Speh $\Z(\a,ws)$ et $\Z(\b,k)$ ayant le même 
transfert à $\A$, on déduit de la définition de $\LJ_\ell$ 
et du corollaire \ref{superSpeh} que~:
\begin{equation*}
\Z\Big(\a',\frac{ws}{s_0}\Big) \simeq \Z(\b,k)
\end{equation*}
ce dont on déduit que $ws=ks_0$, \ie $as=s_0$,
et qu'on peut supposer que $\b=\a'$.
\end{proof}

Dorénavant, on fixe un facteur irréductible $\rho'$ de la réduction mod 
$\ell$ de $\rt'$, qu'on écrit $\Sp(\a',k)$.
On pose $k=\omega(\a')\ell^u$ où $u\in\{0,\dots,v\}$ désigne la valuation 
$\ell$-adique de $k$, et $e'=\omega(\a')$.

Rappelons que $w(\rt')=w>1$ d'après le théorème \ref{RogerCarbury}.
Le lemme \ref{wadm} implique donc que $e$ est égal à $\e(\a')$.
Par définition de $\omega(\a')$ et comme $s_0=s(\a')$, on trouve~:
\begin{equation}
\label{FOREP}
e' = \frac{e}{(e,s_0)}.
\end{equation}
On remarque aussi que $s(\rho')=s(\a')=s_0$.
On a donc $a=(w,s_0)$ d'après le lemme \ref{CN}.
Ainsi on retrouve (voir le corollaire \ref{eponine}) le fait que les entiers~:
\begin{equation*}
s = \frac{s_0}{(w,s_0)},
\quad
k = \frac{w}{(w,s_0)}
\end{equation*}
sont premiers entre eux.

D'après le corollaire \ref{ChouFleur}, on a~:
\begin{equation}
\label{FOR1V}
\SFL = \LJ_\ell(\SFS (\tau,\T)) \LJ_\ell(\SFS (\tau,\omega\T))
\dots \LJ_\ell(\SFS (\tau,\omega^{\ell^v-1}\T)).
\end{equation}
où $\omega\in\mult\CC$ est une racine de l'unité telle que 
$\omega^{\deg(\tau)}$ soit d'ordre $\ell^v$.
En outre, si $e\neq1$, la propo\-si\-tion \ref{formuleZrho0} implique que 
$\SFS(\tau,\T)$ est égale à $\SFDD(\a,\T)$.
(Pour la notation, voir \eqref{DEFNOTD}.)
D'après la proposition \ref{ConjRedModlSpehIntro}, on a~:
\begin{equation*}
\label{lordsebastian}
\SFU = \SFS(\rho',\T) \SFS(\rho'\nu,\T) \dots\SFS(\rho'\nu^{a-1},\T).
\end{equation*}
Notons $\tau'$ la représentation cuspidale $\Sp(\a',e')$,
de sorte que $\rho'$ est égale à $\Sp(\tau',\ell^u)$.
Appliquant le corollaire \ref{ChouFleur}, cela donne~:
\begin{equation}
\label{lordsebastian2}
\SFU = \prod\limits_{t=1}^{a}
\SFS (\tau'\nu^t,\T)\SFS (\tau'\nu^t,\xi\T)\dots\SFS (\tau'\nu^t,\xi^{\ell^u-1}\T)
\end{equation}
où $\xi$ est une racine de l'unité telle que $\xi^{\deg(\tau')d}$ 
soit d'ordre $\ell^u$.
Comme on a l'identité~:
\begin{equation*}
\ell^u \cdot \deg(\tau')d = \ell^v \cdot \deg(\tau)s
\end{equation*}
et comme $s$ est premier à $k$, 
donc à $\ell^u$, on en déduit que $\xi^{\deg(\tau)\ell^{v-u}}$ est d'ordre $\ell^u$.
On peut donc supposer que la racine de l'unité $\omega$ choisie en \eqref{FOR1V}
est égale à $\xi$, ce que nous ferons désormais. 

\subsection{Le cas où $e=1$}
\label{LCOE1}

Supposons que $e=1$, \ie que $w$ est une puissance de $\ell$. 
Alors $e'=1$ d'après \eqref{FOREP} et $\e(\a')=1$ d'après le lemme \ref{wadm}.
Par définition de $\LJ_\ell$, on a~:
\begin{equation*}
\LJ_\ell\Big(\SFS(\a,\T)\Big) 
= \r_\ell\Big(\TLJ_\ell\Big(\SFS(\at,\T)\Big)\Big)
= \r_\ell\Big(\SFS(\at',\T)\Big)
=\SFS(\a',\T).
\end{equation*}
Compte tenu de \eqref{FOR1V} et de \eqref{lordsebastian2},
on obtient d'une part~:
\begin{equation*}
\SFL = \SFS (\a',\T)\SFS (\a',\xi\T)\dots\SFS (\a',\xi^{w-1}\T)
\end{equation*}
et d'autre part~:
\begin{equation*}
\SFU = 
\Big(\SFS (\a',\T)\SFS (\a',\xi\T)\dots\SFS (\a',\xi^{k-1}\T)\Big)^a.
\end{equation*}
Pour en déduire que $\SFU=\SFL$, il suffit de voir que $w=ak$ et que,
pour tout entier $i\in\{1,\dots,w\}$, la série~:
\begin{equation*}
\SFS (\a',\xi^{i}\T) = \sum\limits_{r\>0} (\xi^{i\cdot\deg(\a')d})^{r} 
\Z(\a',r) (-\T^{\deg(\a')d})^{r}
\end{equation*}
ne dépend que de $i$ mod $k$ car $\xi^{\deg(\a')d}$ est d'ordre $k=\ell^u$.

\subsection{Le cas où $e\neq1$}

On suppose désormais que $e\neq1$.
On a donc~:
\begin{equation}
\label{FOR2V}
\SFL 
= \LJ_\ell(\SFDD (\a,\T)) \LJ_\ell(\SFDD (\a,\xi\T))
\dots \LJ_\ell(\SFDD (\a,\xi^{\ell^v-1}\T)).
\end{equation}
Nous calculons l'image de $\SFDD(\a,\T)$ par $\LJ_\ell$ dans le lemme 
suivant. 

\begin{lemm}
\label{derniereffort}
Supposons que $e\neq1$.
\begin{enumerate}
\item
Si $e'=1$, alors~:
\begin{equation*}
\LJ_\ell(\SFDD(\a,\T)) = 
\SFS (\a',\T)\SFS (\a'\nu,\T)\dots\SFS(\a'\nu^{e-1},\T).
\end{equation*}
\item
Si $e'\neq1$, alors~:
\begin{equation*}
\LJ_\ell(\SFDD(\a,\T)) = 
\SFDD (\a',\T)\SFDD (\a'\nu,\T)\dots\SFDD(\a'\nu^{e_0-1},\T).
\end{equation*}
avec $e_0=(e,s_0)$.
\end{enumerate}
\end{lemm}

\begin{proof}
Notons $\SFB$ 
le déterminant de la ma\-tri\-ce carrée de taille $e$ de terme général~: 
\begin{equation*}
\SFB(i,j)
= \SFS(\at\nu^{i+1},e,j-i)
= \sum\limits_{r\in\ZZ} \Z(\at\nu^{i+1},j-i+re) \Y^{re+j-i}
\end{equation*}
avec $\Y=-\T^{\deg(\a)}$.
On a par définition~: 
\begin{equation*}
\LJ_{\ell}\Big(\SFDD(\a,\T)\Big) = \r_\ell\circ\TLJ_{\ell}\Big(\SFB\Big)
\end{equation*}
et nous allons calculer $\TLJ_{\ell}(\SFB)$. 
Par le lemme \ref{quentin}, pour que l'image de $\SFB(i,j)$
par $\TLJ_\ell$ soit non nulle, il faut et il suffit que $j-i$ soit un multiple 
de $e_0$, auquel cas on a~:
\begin{equation*}
\TLJ_\ell(\SFB(i,j)) = \SFS(\at'\nu^{i+1},e',r_{ij})
\end{equation*}
où $r_{ij}$ est l'unique entier compris entre $0$ et $e'-1$ tel que 
$s_0r_{ij}$ soit congru à $j-i$ mod $e$.

Si $e'=1$, alors $\TLJ_\ell(\SFB(i,j))$ est égal à 
$\SFS(\at'\nu^{i+1},\T)$ si et seulement si $e$ divise $j-i$.
On a donc~:
\begin{equation*}
\TLJ_{\ell}(\SFB) = \SFS(\at',\T)\SFS(\at'\nu,\T)\dots\SFS(\at'\nu^{e-1},\T).
\end{equation*}
Réduisant mod $\ell$, on obtient le résultat voulu. 
Supposons maintenant que $e'\neq1$.

\begin{lemm}
Pour que $\a'\nu^i$ soit de la forme $\a'\nu_{\a'}^{l}$ pour $l\in\ZZ$, 
il faut et suffit que $i$ soit un multiple de $e_0$.
\end{lemm}

\begin{proof}
Pour que $\a'\nu^i\simeq\a'\nu_{\a'}^{l}$ pour un $l\in\ZZ$, il faut et suffit 
que $\e(\a')$ divise $i-ls_0$ pour un $l\in\ZZ$, 
\ie que $i$ appartienne à $e\ZZ+s_0\ZZ=e_0\ZZ$.
\end{proof}

\begin{lemm}
\label{cinvs}
On a $\a'\nu^{e_0}\simeq\a'\nu_{\a'}^{c}$ où $c\in\{1,\dots,e'-1\}$ est 
l'inverse de $s$ mod $e'$.
\end{lemm}

\begin{proof}
Soit un entier $i\in\ZZ$.
Pour que $\a'\nu^{e_0}\simeq\a'\nu_{\a'}^{i}$,
il faut et suffit que $e$ divise $e_0-is_0$,
\ie que $e'$ divise $1-is$.
\end{proof}

Notons $\SFM$ la matrice carrée de taille $e$ de terme général~:
\begin{equation*}
\SFM(i,j) = \r_\ell\circ\TLJ_\ell\Big(\SFB(i,j)\Big).
\end{equation*}
Rappelons que $\SFM(i,j)=0$ dès que $e_0$ ne divise pas $j-i$.
La forme particulière de cette matrice va nous permettre de factoriser 
son déterminant. 
Pour $1\<l\<e_0$, on note $\SFM_{l}$ la matrice carrée de taille $e'$ 
de terme général
$\SFM_l(x,y) = \SFM(l+xe_0,l+ye_0)$.
Comme fonction de $\SFM_l$, le dé\-ter\-mi\-nant 
$\det(\SFM)$ est multilinéaire alterné.
Il se factorise donc par $\det(\SFM_l)$. 
On trouve que $\det(\SFM)$ est égal à $\det(\SFM_1)\dots\det(\SFM_{e_0})$.
Compte tenu de la formule \eqref{plustardvousverrez}
et du lemme \ref{cinvs}, 
on a~:
\begin{equation*}
\SFM_l(sx,sy)
= \SFS(\a'\nu_{\a'}^{x}\nu^{1+l},e',y-x)
= \SFZ(\a'\nu^{l+1},x,y-1).
\end{equation*}
On trouve que $\det(\SFM_l) = \SFDD(\a'\nu^{l+1},\T)$,
ce qui implique la formule annoncée.
\end{proof}

Dans les deux cas ($e'=1$ et $e'\neq1$), 
raisonnant comme dans le paragraphe \ref{LCOE1}, on trouve~:
\begin{equation*}
\SFL 
= \Big(\LJ_\ell(\SFDD (\a,\T)) \LJ_\ell(\SFDD(\a,\xi\T))
\dots \LJ_\ell(\SFDD (\a,\xi^{\ell^u-1}\T))\Big)^{\ell^{v-u}}
\end{equation*}
à partir de \eqref{FOR2V},
car $\xi^{\deg(\a')d}$ est d'ordre $\ell^u$. 

Supposons que $e'=1$, et prouvons que $\SFU=\SFL$.
D'après le lemme \ref{derniereffort}, on a~:
\begin{equation*}
\SFL = \prod\limits_{l=1}^{e}
\Big(\SFS(\a'\nu^l,\T) \SFS (\a'\nu^l,\xi\T)
\dots \SFS (\a'\nu^l,\xi^{\ell^u-1}\T)\Big)^{\ell^{v-u}}.
\end{equation*}
D'autre part, comme $e=\e(\a')$ et $a=e\ell^{v-u}$, 
la formule \eqref{lordsebastian2} entraîne~:
\begin{eqnarray*}
\SFU &=& 
\prod\limits_{t=1}^{a}
\SFS (\a'\nu^t,\T)\SFS (\a'\nu^t,\xi\T)\dots\SFS 
(\a'\nu^t,\xi^{\ell^u-1}\T) \\
&=& \prod\limits_{t=1}^{e}
\Big(\SFS (\a'\nu^t,\T) \SFS (\a'\nu^t,\xi\T)\dots \SFS 
(\a'\nu^t,\xi^{\ell^u-1}\T)\Big)^{\ell^{v-u}}
\end{eqnarray*}
ce qui prouve l'égalité cherchée. 

Supposons que $e'\neq1$.
Remarquons que $e_0$ est la partie première à $\ell$ de $(w,s_0)=a$, 
qui vaut $1$ ou $\e(\rho')=\e(\tau')$.
D'après le lemme \ref{derniereffort}, on a~:
\begin{eqnarray*}
\SFL 
&=& \prod\limits_{l=1}^{e_0}
\Big(\SFDD(\a'\nu^l,\T) \SFDD (\a'\nu^l,\xi\T)
\dots \SFDD (\a'\nu^l,\xi^{\ell^u-1}\T)\Big)^{\ell^{v-u}} \\
&=& \prod\limits_{l=1}^{e_0}
\Big(\SFS(\tau'\nu^l,\T) \SFS (\tau'\nu^l,\xi\T)
\dots \SFS (\tau'\nu^l,\xi^{\ell^u-1}\T)\Big)^{\ell^{v-u}}
\end{eqnarray*}
ce qui, compte tenu de la formule \eqref{lordsebastian2} et du fait que 
$a=e_0\ell^{v-u}$, est égal à $\SFU$.



\providecommand{\bysame}{\leavevmode ---\ }
\providecommand{\og}{``}
\providecommand{\fg}{''}
\providecommand{\smfandname}{\&}



\begin{thebibliography}{10}

\bibitem{ABPS}
A.-M.~Aubert, P.~Baum, R.~Plymen et M.~Solleveld, 
\textit{Depth and the local Langlands correspondence}, prépublication (2013). 

\bibitem{BaduJL}
A.~I.~Badulescu,
\textit{Correspondance de Jacquet-Langlands en caractéristique non nulle}, 
{Ann. Scient. Éc. Norm. Sup. (4)} \textbf{35} (2002), 695--747.

\bibitem{BaduJIMJ}
A.~I.~Badulescu, 
\textit{Jacquet-Langlands et unitarisabilité}, 
{J. Inst. Math. Jussieu} \textbf{6} (2007), n°3, 349--379.

\bibitem{BHLS}
A.~I. Badulescu, G.~Henniart, B.~Lemaire et V.~S{\'e}cherre, 
\textit{Sur le dual unitaire de {${\rm GL}_r(D)$}},
  {Amer. J. Math.} \textbf{132} (2010), n°5, 1365--1396.

\bibitem{BSS}
P.~Broussous, V.~S{\'e}cherre et S.~Stevens, 
\textit{Smooth representations of {${\rm GL}(m,D)$}, {V}: 
    {endo-classes}}, {Do\-cu\-menta Math.} \textbf{17} (2012), 23--77. 

\bibitem{BHcgds}
C.~Bushnell et G.~Henniart,
\textit{Counting the discrete series for {${\rm GL}(n)$}}, 
{Bull. Lond. Math. Soc.} \textbf{39} (2007), n°1, 133--137. 


\bibitem{BHJL3}
C.~Bushnell et G.~Henniart,
\textit{The essentially tame Jacquet-Langlands correspondence for inner forms of
  ${\rm GL}(n)$}, {Pure Appl. Math. Q.} {\bf 7} (2011), n°3, 469--538. 

\bibitem{BHl}
C.~Bushnell et G.~Henniart,
\textit{Modular local Langlands correspondence for {${\rm GL}_n$}}, 
{Int. Math. Res. Not.} \textbf{15} (2014), 4124--4145. 

\bibitem{BHL}
C.~Bushnell, G.~Henniart et B.~Lemaire,
\textit{Caractère et degré formel pour les formes intérieures de {${\rm GL}(n)$} 
sur un corps local de caractéristique non nulle}, 
{Manuscripta Math.} {\bf 131} (2010), n°1-2, 11--24. 

\bibitem{Datl}
J.-F.~Dat,
\textit{Théorie de Lubin-Tate non-abélienne $\ell$-entière}, 
{Duke Math. J.} \textbf{161} (2012), n°6, 951--1010.

\bibitem{Datj}
J.-F.~Dat, 
\textit{Un cas simple de correspondance de Jacquet-Langlands modulo
  $\ell$}, {Proc. London Math. Soc.} \textbf{104} (2012), 690--727.

\bibitem{DKV}
P.~Deligne, D.~Kazhdan et M.-F.~Vign{\'e}ras, 
\textit{Repr\'esentations des alg\`ebres centrales simples {$p$}-adiques}, 
Representations of reductive groups over a local field, Hermann, 
Paris, 1984. 

\bibitem{Dippd2}
R.~Dipper, \textit{On the decomposition numbers of the finite 
  general linear groups. {II}}, {Trans. Amer. Math. Soc. } \textbf{292} 
  (1985), n°1, 123--133. 

 \bibitem{DJ}
 R.~Dipper et G.~James, \textit{Identification
   of the irreducible modular representations of {${\rm GL}_n(q)$}}, {J.
   Algebra} \textbf{104} (1986), n°2, 266--288.

\bibitem{Green} 
J.~A.~Green,
\textit{The characters of the finite general linear 
   groups}, {J. Algebra} \textbf{184} (1996), n°3, 839--851. 

\bibitem{James}
G.~James,
\textit{The irreducible representations of the finite general linear
  groups}, {Proc. London Math. Soc. (3)} \textbf{52} (1986), n°2,
  236--268.

\bibitem{JL}
H.~Jacquet et R.~P.~Langlands, 
\textit{Automorphic forms on ${\rm GL}(2)$}, 
Lecture Notes in Mathematics \textbf{114}, Springer, 1970.

\bibitem{KL}
A.~Kret et E.~Lapid, 
\textit{Jacquet modules of ladder representations}, 
{C. R. Math. Acad. Sci. Paris} \textbf{350} (2012), 
n°21-22, 937--940. 

\bibitem{MSb}
A.~M{\'{\i}}nguez et V.~S{\'e}cherre, 
\textit{Représentations banales de {${\rm GL}(m,D)$}}, 
{Compos. Math.} \textbf{149} (2013), 679--704.

\bibitem{MSc}
A.~M{\'{\i}}nguez et V.~S{\'e}cherre, 
\textit{Repr{\'e}sentations lisses modulo {$\ell$} de {${\rm GL}_m(\D)$}}, 
{Duke Math. J.} {\bf 163} (2014), 795--887. 

\bibitem{MSt}
A.~M{\'{\i}}nguez et V.~S{\'e}cherre, 
\textit{Types modulo $\ell$ pour les formes intérieures de ${\rm GL}_{n}$ sur un 
  corps local non archimédien}.
Avec un appendice par V.~Sécherre et S.~Stevens.
Proc. London Math. Soc. {\bf 109} (2014), n°4, 823--891.

\bibitem{MSf}
A.~M{\'{\i}}nguez et V.~S{\'e}cherre, 
\textit{Repr{\'e}sentations modulaires de {${\rm GL}_n(q)$} en 
caract\'eristique non naturelle}, 
Contemporary Math. {\bf 649} (2015).

\bibitem{MSi}
A.~M{\'{\i}}nguez et V.~S{\'e}cherre, 
\textit{L'involution de Zelevinski modulo $\ell$}, 
Represent. Theory {\bf 19} (2015), 236--262. 

\bibitem{Rog}
J.~Rogawski,
\textit{Representations of {${\rm GL}(n)$} and division algebras over 
a {$p$}-adic field}, {Duke Math. J.} {\bf 50} (1983), 161--196. 

\bibitem{SeSt1}
V.~S{\'e}cherre et S.~Stevens,
\textit{Repr\'esentations lisses de {${\rm GL}(m,D)$}, {IV} : {repr\'esentations
  supercuspidales}}, {J. Inst. Math. Jussieu} \textbf{7} (2008), n°3, 
  527--574.

\bibitem{SeSt2}
V.~S{\'e}cherre et S.~Stevens, 
\textit{Smooth representations of {${\rm GL}(m,D)$}, {VI}: {semisimple
  types}}, {Int. Math. Res. Not.} \textbf{13} (2012), 2994--3039. 

\bibitem{SSb}
V.~S{\'e}cherre et S.~Stevens, 
\textit{Block decomposition of the category of $\ell$-modular smooth 
  representations of ${\rm GL}_n(\F)$ and its inner forms}. 
Ann. Scient. Éc. Norm. Sup. {\bf 49} (2016), n°3, 669--709. 

\bibitem{Tadic}
M.~Tadi{\'c},
\textit{Induced representations of {${\rm GL}(n,A)$}
  for {$p$}-adic division algebras {$A$}}, 
{J. Reine Angew. Math.} \textbf{405} (1990), 48--77.

\bibitem{Vigb}
M.-F. Vign{\'e}ras,
\textit{Repr\'esentations {$l$}-modulaires d'un
  groupe r\'eductif {$p$}-adique avec {$l\ne p$}}, Progress in Mathematics,
  vol. 137, Birkh\"auser Boston Inc., Boston, MA, 1996.

\bibitem{Vigl}
M.-F. Vign{\'e}ras, 
\textit{Correspondance de Langlands semi-simple pour {$\GL_n(\F)$} 
  modulo {$\ell\neq p$}}, {Invent.  Math} \textbf{144} (2001), 
177--223. 

\bibitem{Vigw}
M.-F. Vign{\'e}ras, 
\textit{On highest Whittaker models and integral structures}, 
 {Contributions to Automorphic forms, Geometry and Number theory: 
   Shalikafest 2002}, John Hopkins Univ. Press, 2004, 773--801.

\bibitem{Ze2}
A.~V. Zelevinsky, \textit{Induced representations of reductive
  {${\mathfrak{p}}$}-adic groups. {II}. {O}n irreducible representations of
  {${\rm GL}(n)$}}, {Ann. Scient. Éc. Norm. Sup. (4)} \textbf{13}
  (1980), n°2, 165--210.

\end{thebibliography}
\end{document}